\documentclass[11pt]{amsart}
\usepackage{amsthm}
\usepackage{amsmath,amssymb,amscd,eucal}
 \usepackage{latexsym}
\usepackage{graphicx}
\usepackage[usenames,dvipsnames,svgnames,table]{xcolor}
\usepackage{times,tikz} 
\usepackage{color} 
 
\newtheorem{thm}[equation]{Theorem}
\newtheorem{cor}[equation]{Corollary}
\newtheorem{lemma}[equation]{Lemma}
\newtheorem{prop}[equation]{Proposition}

\newtheorem*{conj*}{Conjecture}
\newtheorem{conv}[equation]{Conventions}
\newtheorem{defn}[equation]{Definition}

\newtheorem{remark}[equation]{Remark}

\newtheorem{exam}[equation]{Example}

\numberwithin{equation}{section}

\renewcommand{\ker}{\mathsf{ker}}
\renewcommand{\gcd}{\mathsf{gcd}}
\newcommand{\overbar}[1]{\mkern 1mu\overline{\mkern-1mu#1\mkern-1mu}\mkern 1mu}

\newcommand{\FF}{\mathbb{F}}

\newcommand{\A}{\mathsf{A}} 
  
\newcommand{\be}{\breve{E}} 
\newcommand{\brf}{\breve{F}} 
\newcommand{\DD}{\mathsf{R}} 
\newcommand{\C}{\mathsf{C}}

\newcommand{\PP}{\mathsf{P}} 

\newcommand{\pr}{\mathsf{u}}

\newcommand{\ze}{\mathsf{z}_h} 
\newcommand{\ms}{\breve{q}}
\newcommand{\dms}{D_{\ms}}
 
\newcommand{\parp}{\partial_p}

\newcommand\chara{\mathsf {char}}
\newcommand{\pie}{\pi_{(h/\pi_h)}}
\newcommand{\vth}{\varsigma} 
\def\inder{ \mathsf {Inder_\FF}}
\def\der{ \mathsf {Der_\FF}}

\def\hoch{\mathsf{HH^1}}
\def\ad{\mathsf {ad}}
\def\im{\mathsf{im}}
\def\norm{\mathsf{N}_{\A_1}(\A_h)} 
\def\cent1{\mathsf{Z}(\A_1)}
\def\centh{\mathsf{Z}(\A_h)}
\def\modd{\, \mathsf{mod} \,}
\def\modu{\, \mathsf{mod} \,}
\def\degg{\, \mathsf{deg} \,}
\def\dimm{\, \mathsf{dim}\,}

\def\spann{\, \mathsf{span}}
\def\deghpi{\degg\frac{h}{\pi_h}}

\begin{document}  

 \title[Derivations of a Family of Subalgebras of the Weyl Algebra]
{Derivations of a Parametric Family of Subalgebras of the Weyl Algebra} 

\author[Benkart, Lopes, and Ondrus]{Georgia Benkart, Samuel A. Lopes$^*$, and Matthew Ondrus }

\address{ \noindent Department of Mathematics,  
 University of Wisconsin-Madison, 
 Madison, WI 53706-1388, USA, \ \  \emph{E-mail address}: \tt{benkart@math.wisc.edu}}
 
\address{\noindent CMUP, Faculdade de Ci\^encias, 
Universidade do Porto, 
Rua do Campo Alegre 687, 
4169-007 Porto, Portugal, \ \  \emph{E-mail address}: \tt{slopes@fc.up.pt}}
 
 \address{\noindent Mathematics Department,
Weber State University,
Ogden, UT  84408 USA,  \emph{E-mail address}: \tt{mattondrus@weber.edu}}
 
 \thanks{$^*$ The author was partially funded by the European Regional Development Fund through the programme COMPETE and by the Portuguese Government through the FCT -- Funda\c c\~ao para a Ci\^encia e a Tecnologia under the project PEst-C/MAT/UI0144/2013. \newline  \newline
\textbf{MSC Numbers (2010)}: Primary: 16S32, 16W25;  Secondary  16E40, 16S36, 17B40 \hfill \newline
\textbf{Keywords}:  Ore extensions, Weyl algebras, derivations, Hochschild cohomology, Witt algebra} 
\date{}

\vspace{-.25 truein}  
\begin{abstract} An Ore extension over a polynomial algebra $\FF[x]$ is either a quantum plane, a quantum Weyl algebra, or an infinite-dimensional unital associative algebra $\A_h$ generated by elements $x,y$,  which satisfy $yx-xy = h$, where $h\in \FF[x]$.  When $h \neq 0$, the algebra $\A_h$ is subalgebra of the Weyl algebra $\A_1$ and can be viewed as differential operators with polynomial coefficients.  This paper determines the derivations of  $\A_h$ and the Lie structure of the first Hochschild cohomology group $\hoch(\A_h) = \der(\A_h)/\inder(\A_h)$ of outer derivations over an arbitrary field.   In characteristic 0, we show that  $\hoch(\A_h)$ has a unique maximal nilpotent  ideal modulo which it is 0 or a direct sum of simple Lie algebras that are field extensions of  the one-variable Witt algebra.  In  positive characteristic,  we obtain decomposition theorems for $\der(\A_h)$ and $\hoch(\A_h)$ and describe the structure of  $\hoch(\A_h)$ as a module over the center of $\A_h$.  \end{abstract}   
\maketitle 

\begin{section}{Introduction}\end{section}
We consider a family of infinite-dimensional unital 
associative algebras $\,\A_h$ parametrized by a polynomial
$h$ in one variable, whose definition is given as follows: 

\smallskip

\begin{defn}
Let $\FF$ be a field, and let $h \in \FF[x]$.  The algebra  $\A_h$ is  the unital associative algebra over $\FF$ with generators $x$, $y$ and defining relation $yx = xy + h$ (equivalently,  $[y,x] = h$  where $[y,x] = yx-xy$). 
\end{defn}

These algebras arose naturally in considering  Ore extensions over a polynomial algebra $\FF[x]$.
Many algebras can be realized as iterated Ore extensions, and for that reason, Ore 
extensions have become a mainstay in associative theory.   Recall that an Ore extension $\A = \DD[y,\sigma, \delta]$  is built from a unital associative (not necessarily commutative) algebra $\DD$ over a field $\FF$, an $\FF$-algebra endomorphism $\sigma$ of $\DD$, and a $\sigma$-derivation of $\DD$, where by a  $\sigma$-derivation $\delta$  we mean that  $\delta$ is $\FF$-linear  and $\delta (rs) = \delta(r)s + \sigma (r) \delta (s)$ holds  for all $r,s \in \DD$.  Then  $\A = \DD[y, \sigma, \delta]$ is the algebra generated by $y$ over $\DD$ subject to the relation 
$$yr= \sigma(r)y + \delta(r) \qquad \hbox{\rm for all} \ r \in \DD.$$  Under the assumption that  $\DD = \FF[x]$ and $\sigma$ is an automorphism of $\DD$,  the following result holds.  (Compare  \cite{AVV87} and \cite{AD97}, which have a somewhat different division into cases.) 

\begin{lemma}\label{lem:poly}  Assume $\A = \DD[y,\sigma,\delta]$ is an Ore extension with $\DD = \FF[x]$,
a polynomial algebra over a field $\FF$ of arbitrary characteristic and $\sigma$ an automorphism of $\DD$.   Then $\A$ is isomorphic
to one of the following:
\begin{itemize}
\item[{\rm (a)}] a  quantum plane
\item[{\rm (b)}] a quantum Weyl algebra
\item[{\rm (c)}] an algebra $\A_h$  with generators $x,y$ and defining relation
$yx = xy + h$ for some polynomial $h \in \FF[x]$.  
\end{itemize}
\end{lemma}
The algebras $\A_h$ result  from taking $\DD = \FF[x]$,  $\sigma$ to be the identity automorphism, and $\delta : \DD \to \DD$ to be the derivation given by
\begin{equation}\label{eqn:defDelta}
\delta(f) = f' h,
\end{equation}
where $f'$ is the usual derivative of $f$ with respect to $x$.

Quantum planes and quantum Weyl algebras are examples of generalized Weyl algebras in the sense of
 \cite[1.1]{bavula:gwar93}, and as such,  have been studied extensively. In \cite{BLO1, BLO2},  we determined the center, normal elements, and prime ideals of the algebras $\A_h$, as well as the automorphisms and their invariants,  isomorphisms between two algebras $\A_g$ and $\A_h$, and the irreducible $\A_h$-modules   over any field $\FF$.   
Our aim in this paper is to compute the derivations and first cohomology group  of the algebras $\A_h$ over an arbitrary field.

When $h = 1$,  the algebra $\A_1$ is the Weyl algebra, and Sridharan \cite{Sr61}  showed that  when the characteristic of $\FF$ is 0,  the Hochschild cohomology of $\A_1$ vanishes in positive degrees. In particular,  the derivations of $\A_1$ are all inner when $\chara(\FF) = 0$,
since the first cohomology vanishes (compare \cite{Dix66} and  \cite{Dix96}).   In recent work \cite{GG14}, Gerstenhaber and Giaquinto have used  the fact that the Euler-Poincar\'e characteristic is invariant under deformation to compute the cohomology of the Weyl algebra, the quantum plane, and the quantum Weyl algebra under the assumption $\chara(\FF) = 0$.

Progress towards  determining the derivations of $\A_h$ for arbitrary $h$  has been made in \cite{Now04}, primarily in the characteristic 0 case. Theorem 9.1 of \cite{Now04} shows that when $\chara(\FF) = 0$, every derivation is inner if and only if $h \in \FF^*$ (in the notation used here).  Nowicki also establishes decomposition results (see  \cite[Thms.~10.1 and 11.2]{Now04})  for derivations of $\A_h$. These results can be obtained as special cases of  Theorem \ref{T:hochstruct} below, which gives a direct sum decomposition of $\der (\A_h)$.  In addition, we derive expressions for the Lie bracket in the quotient  $\hoch (\A_h) = \der (\A_h) / \inder (\A_h)$ of $\der(\A_h)$ modulo the ideal $\inder(\A_h)$ of inner derivations when $\chara (\FF) = 0$  and use these formulas to understand the structure of the Lie algebra $\hoch (\A_h)$ (see Theorem \ref{thm:hochdec}).  In Theorem \ref{thm:idealJ} 
and Corollary \ref{cor:HH1nilpotent}, we show that there is a unique maximal nilpotent ideal of $\hoch (\A_h)$ and explicitly describe the structure of the quotient by this ideal in terms of the one-variable Witt algebra (centerless Virasoro algebra).

When $\chara (\FF) = p>0$, not all derivations of $\A_1$ are inner (contrary to the statement in \cite{R73}).  In Section \ref{derA1}, we introduce two non-inner derivations $E_x$ and $E_y$ of $\A_1$ and use them  in Theorem \ref{thm:derA1} to describe $\der (\A_1)$ as well as $\hoch (\A_1)$.    Section \ref{sec:derChar-p} of the paper is devoted to  studying  $\der (\A_h)$ for arbitrary $h \neq 0$  in the  characteristic $p > 0$ case.  The  restriction map $\mathsf{Res}: \der (\A_h) \to \der ( \centh)$ from derivations of $\A_h$ to derivations of the center $\centh$ of $\A_h$  is a morphism of Lie algebras, and in the case $h = 1$,  this map is surjective with kernel $\inder (\A_1)$.  Viewing $\A_h$ as a subalgebra of $\A_1$ for $h \neq 0$ and applying results from Section 3 on  derivations of $\A_1$, we determine the kernel and image of $\mathsf{Res}$  in Proposition \ref{prop:kerres} and Theorem \ref{T:imres} respectively.  This enables us in Theorem \ref{T:decomp2} to exp
 licitly 
 determine all derivations of $\A_h$, for arbitrary $h \neq 0$, when $\chara (\FF) = p>0$.  To illustrate this result, we compute $\der(\A_h)$  for $h = x^m$ for any $m \geq 0$  (Corollary \ref{E:small h}) and for any $h \in \FF[x^p]$ (Example \ref{Ex:h=rho}).  In Proposition \ref{C:inner}, we provide a criterion for a derivation of $\A_h$ to be inner for general $h$,  and in Theorem \ref{T:pfreeHH1}, we  present necessary and sufficient conditions on $h$ for $\hoch(\A_h)$ to be free over $\centh$.   Propositions \ref{prop:bracketp} and  \ref{prop:bracketp2}  give formulas for the Lie brackets in $\der (\A_h)$.

Several well-known algebras have the form $\A_h$ for some $h \in \FF[x]$.  For example, $\A_0$ is the polynomial algebra $\FF[x,y]$; $\A_1$ is the Weyl algebra;  and the algebra $\A_x$ is the universal enveloping algebra of the two-dimensional non-abelian Lie algebra (there is only one such Lie algebra up to isomorphism).   The algebra $\A_{x^2}$ is often referred to as the Jordan plane.  It appears  in noncommutative algebraic geometry (see for example, \cite{SZ94} and \cite{AS95}) and exhibits many interesting features such as being Artin-Schelter regular of dimension 2.   In a series of articles \cite{shirikov05}--\cite{shirikov07-2},  Shirikov has undertaken an extensive study of the automorphisms, derivations, prime ideals, and modules of the algebra $\A_{x^2}$. Aspects  of the theory developed in \cite{shirikov05}--\cite{shirikov07-2} have been extended by Iyudu \cite{I12} to include results on varieties of finite-dimensional modules of $\A_{x^2}$ over 
algebraically closed fields of characteristic 0.  Cibils, Lauve, and Witherspoon \cite{CibLauWit09}  have used quotients of the algebra $\A_{x^2}$ and cyclic subgroups of their automorphism groups to construct new examples of finite-dimensional Hopf algebras in prime characteristic which are Nichols algebras. 

The universal enveloping algebras  $\mathsf{YM}(n)$  of the Yang-Mills algebras form another family of infinite-dimensional associative algebras which have been studied because of their connections with deformation theory.   Theorem 5.11 of \cite{HS} determines the Lie structure of the first Hochschild cohomology group of $\mathsf{YM}(n)$ over an algebraically closed field of characteristic 0.   This turns out to be finite dimensional and can be described in terms of the orthogonal  Lie algebra $\mathfrak{so}(n)$.  By contrast, $\hoch(\A_h)$ generally is infinite dimensional and related to the Witt algebra 
under the assumption $\FF$ has characteristic 0. 
 
There are striking similarities in the behavior of the algebras $\A_h$ as $h$ ranges over the polynomials in $\FF[x]$.  For that reason, we believe that studying them as one family provides much insight into their structure, derivations, automorphisms, and modules.
\smallskip

\noindent  {\bf Acknowledgments:}  \  We thank Andrea Solotar and Mariano Su\'arez-\'Alvarez for discussions about the Hochschild cohomology of the Weyl algebra in positive characteristic and for pointing out the argument in Remark~\ref{rem:dt}.  We are also  grateful to the referee for carefully reading the first version of this paper and providing helpful feedback.
 
\begin{section}{Preliminaries}\end{section}

In this section,  we recall some necessary background  from \cite{BLO1} and prove results required for our description of the derivations of $\A_h$.   We begin with facts about  embeddings. 
\smallskip

\begin{lemma}{\rm \cite[Sec.\ 3]{BLO1}}  \label{L:embed}   
\begin{itemize} \item[\rm {(a)}]  Suppose that $f$ and $g$ are nonzero elements of $\FF[x]$ and  $g = f r$ for some $r\in \FF[x]$.   Regard $\A_f = \langle x,y,1\rangle$ and $\A_g = \langle x,\tilde y,1\rangle$ with the relations
$yx-xy  = f$ and $\tilde y x - x \tilde y = g$ respectively.    Then the map   $\varepsilon:  \A_g \rightarrow \A_f$ with 
$x \mapsto x, \quad  \tilde y \mapsto  y r$  gives  an embedding of $\A_g$ into $\A_f$.   
 \item[\rm {(b)}]  For all $h  \in \FF[x]$,  $h \neq 0$,  there is an embedding of the algebra $\A_h$  into the Weyl algebra  $\A_1$.  If   $x, y$ are the generators of the Weyl algebra so that  $[y, x]=1$, then  $\A_h$ can be identified with the subalgebra 
$\A_h = \langle x, \hat y, 1 \rangle$ of $\A_1$ generated by $x$, $\hat y = yh$, and $1$. 
 \item[\rm {(c)}] Regard $\A_h \subseteq \A_1$ as in {\rm (b)}, and write $\DD = \FF[x]$.  Then 
\begin{equation}\label{eq:Ahexpress} \A_h = \bigoplus_{i \ge 0} \DD h^iy^i = \bigoplus_{i \ge 0} y^i h^i \DD.\end{equation}
\end{itemize}
\end{lemma}

Because we often use the embedding in Lemma \ref{L:embed}\,(b)  as a tool for proving results, 
and because the structure and derivations of $\A_0 = \FF[x,y]$ are very well understood, for the remainder of this paper
 we adopt the following conventions:
 \begin{conv}  \label{con:gens}  \qquad 
 
 \begin{itemize}
 \item  $\DD = \FF[x]$, and the polynomial $h \in \DD$ is nonzero;
 \item   the generators of the Weyl algebra $\A_1$ are  $x,\,y,\,1$ and $[y,x] = 1$;
 \item  the generators of the algebra $\A_h$ are $x,\,\hat y,\,1$ and $[\hat y,x] = h$;   
 \item  when $\A_h$ is viewed as a subalgebra of $\A_1$, then $\hat y = yh$.  
 \end{itemize}
 \end{conv}

The center of the Weyl algebra $\A_1$ 
is $\FF 1$ when $\chara(\FF) = 0$. When $\chara(\FF) = p > 0$, the center of $\A_1$  has been described  by Revoy in \cite{R73} (see also \cite{MakLim84}).  The next result describes the center of an
arbitrary algebra $\A_h$.  

\begin{thm}\label{L:center}{\rm \cite[Sec.~5]{BLO1}} 
Regard $\A_h  \subseteq \A_1$ as in Conventions \ref{con:gens}, and let $\centh$ denote the center of $\A_h$.  
\begin{enumerate}
\item[{\rm (1)}]  If $\chara(\FF) = 0$, then $\centh  = \FF 1$.
\item[{\rm (2)}]  If $\chara(\FF) = p > 0$, then $\centh$ is the polynomial subalgebra $\FF[x^p,\ze] =\FF[x^p, h^py^p]$ of $\A_1$, where
$$\ze = h^py^p = y^ph^p  = \hat y ( \hat y+h')( \hat y+2h') \cdots ( \hat y + (p-1)h') = \hat y^p - \frac{\delta^p(x)}{h} \hat y,$$
and  $\delta$ is the derivation of $\DD =\FF[x]$ with $\delta(f) = f' h$ for all $f \in \DD$.  Moreover  $\frac{\delta^p(x)}{h}  \in \centh \cap \FF[x] = \FF[x^p]$.
\item[{\rm (3)}]  If $\chara(\FF) = 0$, then $\A_h$ is free over its center  $\centh$  with basis  $\{x^i\hat y^j \mid  i,j \in \mathbb Z_{\geq 0}\}$.  If $\chara(\FF) = p>0$, then $\A_h$ is free over $\centh$ with  
basis $\{x^ih^jy^j \mid 0 \leq i,j < p\}$ or with basis $\{x^i\hat y^j  \mid 0 \leq i,j < p\}$.   \end{enumerate}
\end{thm}  \medskip

The {\it centralizer}  $\mathsf{C}_{\A_h}(x) 
= \{a \in \A_h \mid [a,x] = 0\}$ of $x$ in $\A_h$ has been calculated in \cite{BLO1},  and we summarize the results next. 

\begin{lemma}\label{L:clizer}{\rm \cite[Lem.\ 6.3]{BLO1}} \ \    $\mathsf{C}_{\A_h}(x) = \centh \DD$.  Hence,
\begin{equation*} 
  \mathsf{C}_{\A_h}(x)  = \begin{cases}  \DD = \FF[x] & \quad \hbox{\rm if } \ \ \chara(\FF) = 0,\\
 \FF[x, h^py^p] & \quad \hbox{\rm if } \ \ \chara(\FF) = p > 0. \end{cases}
 \end{equation*}
 In particular, $\mathsf{C}_{\A_1}(x) = \DD$ when $\chara(\FF) = 0$, 
and  $\mathsf{C}_{\A_1}(x) = \FF[x,y^p]$ when  $\chara(\FF) = p > 0$.   
\end{lemma} 

The \emph{normalizer}
\begin{equation}\label{eq:norm}  \norm = \{ u \in \A_1 \mid [u, \A_h] \subseteq \A_h \} \end{equation} 
of $\A_h $ in $\A_1$ is closely related to the derivations of $\A_h$, as 
\begin{equation}\label{eq:norm2}  u \in  \norm \iff \ad_u  \ \ \hbox{\rm restricts to a derivation of} \ \ \A_h, \end{equation} 
where $\ad_u$ is  the inner derivation of $\A_1$
given by $\ad_u(v) = [u,v] = uv-vu$. 

We begin with a computational lemma from \cite[Lem.\ 5.2]{BLO1} and then introduce a certain element $\pi_h \in \DD$ that depends upon $h$ and plays an essential role in describing $\norm$.   
\smallskip

\begin{lemma}\label{lem:identity_y^nf} Let $h \in \DD =\FF[x]$, and let $\delta: \DD \to \DD$ be the derivation with $\delta (f) =  f' h$ for all $f \in \DD$.  Then 
\begin{eqnarray}&& [\hat y^n, f] = \sum_{j=1}^n {n \choose j} \delta^j (f) \hat y^{n-j} \qquad \hbox{\rm in} \ \  \A_h \label{eq:Ahcom} \\ 
&&[y^n,f] = \sum_{j=1}^n {n \choose j} f^{(j)} y^{n-j} \ \  \qquad \hbox{\rm in} \ \ \A_1 \label{eq:A1com} \end{eqnarray}
where $f^{(j)} = (\frac{d}{dx})^j(f)$.  
\end{lemma} \smallskip

\begin{cor}\label{C:yhatprod} For all $r \in \DD$ and all $n \geq 0$,   
\begin{equation}\label{eq:yhatprod} [ry^n, \hat y]  = -(rh)' y^n + r \sum_{j=1}^{n+1} {{n+1}\choose j} h^{(j)} y^{n+1-j}. \end{equation}
\end{cor}  
\noindent {\it Proof. } \ Using \eqref{eq:A1com}, we have 
\begin{eqnarray*}\  \  [ry^n, \hat y] &=&  [ry^n, yh] = [ry^n, hy] + [ry^n, h'] \\
&=& r\sum_{j=1}^n {n \choose j} h^{(j)} y^{n+1-j} - hr' y^n + r \sum_{j=1}^n {n \choose j} h^{(j+1)} y^{n-j} \\
&=& -(rh)' y^n + r \sum_{j=1}^{n+1} {{n+1}\choose j} h^{(j)} y^{n+1-j}.  \hspace{3.5cm} \square  \end{eqnarray*}  

\begin{lemma}\label{L:defpi} Let $\DD = \FF[x]$.
\begin{itemize} \item[{\rm (i)}] There is a unique monic polynomial $\pi_{h}\in\DD$ such that 
\begin{equation*}
\forall \,  r\in\DD, \qquad h\mid h'r\iff \pi_{h}\mid r.
\end{equation*}
In particular,  $\pi_{h}\mid h$,   and $\pi_{h}=1$ if $h' = 0$. 
\item[{\rm (ii)}]  If $h\notin \FF$, write $h=\lambda \pr_{1}^{\alpha_{1}}\cdots \pr_{t}^{\alpha_{t}}$, where 
$\lambda\in\FF^{*}$, $t\geq 1$, $\alpha_{i}\geq 1$ for all $i$,  and the  $\pr_{i}$  are distinct monic primes
in $\DD$. 
\begin{itemize} 
\item[{\rm (a)}]   If $\chara(\FF) = 0$, then $\pi_h  = \pr_{1}\cdots \pr_t$. 
\item [{\rm (b)}]  If $\chara(\FF) = p > 0$,  then  $\pi_h = \displaystyle{\prod_{i, \, \pr_i^{\alpha_i} \not \in \FF[x^p]} \,  \pr_i}$,  and if $h \in \FF[x^p]$,  then  $\pi_h  = 1$. 
\end{itemize}
Hence, $\pi_{h}= \frac{h}{\gcd (h, h')}$. \end{itemize}
\end{lemma} 
\begin{proof}
Let $\mathsf{J} =\{ r\in\DD\mid \, h\ \mbox{divides}\ h'r \}$. Then $\mathsf{J}$ is an ideal of the principal ideal domain $\DD$, so there is a unique monic polynomial $\pi_{h}\in\DD$ that generates $\mathsf{J}$. This proves the existence and uniqueness of $\pi_{h}$.  Furthermore, it is clear that $\pi_h \mid h$ since $h \in \mathsf{J}$,  and that  $\pi_{h}=1$ if $h\in\FF$ or if $h \in \FF[x^p]$,
as $h' = 0$.

Assume $h \not \in \FF$ and  $h=\lambda \pr_{1}^{\alpha_{1}}\cdots \pr_{t}^{\alpha_{t}}$ as above. Set $\mathsf{u} = \pr_{1}\cdots \pr_t$. Then
\begin{equation*}
h'=\frac{h}{\mathsf{u}} \sum_{i=1}^t \alpha_{i}\pr_{1}\cdots \pr'_{i}\cdots \pr_t.
\end{equation*}
Given $r\in\DD$, it is easy to see that \, $h$ \, divides \, $h'r$ \, if and only if \, $\mathsf{u}$ divides $r\sum_{i=1}^t\alpha_{i}\pr_{1}\cdots \pr'_{i}\cdots \pr_t$. \, The latter occurs if and only if \, $\pr_{j}$\, divides \,$r\sum_{i=1}^t\alpha_{i}\pr_{1}\cdots \pr'_{i}\cdots \pr_t$ for every $j$.  This  is equivalent to having \, $\pr_{j}$ \, divide \, $r\alpha_{j}\pr_{1}\cdots \pr'_{j}\cdots \pr_t$ for every $j$. Hence, $h$ divides $h'r$ if and only if $\pr_{j}$ divides $r\alpha_{j}\pr'_{j}$ for every $j$. 

If ${\rm char}(\FF) = 0$, $\alpha_{j}\pr'_{j}\neq 0$ and has degree smaller than $\pr_{j}$, so $\pr_{j}$ divides $r$
for all $j$. Thus, $\pi_{h}=\pr_{1}\cdots \pr_t$.   If $\chara(\FF) = p > 0$, then $\pr_j^{\alpha_j} \in \FF[x^p]$ if and only if $\alpha_{j}\pr'_{j} = 0$, so $h$ divides $h'r$ if and only if $\pr_{j}$ divides $r$ for every $j$ such that $\pr_j^{\alpha_j} \notin \FF[x^p]$. It follows in this case that $\pi_h = \displaystyle{\prod_{i, \, \pr_i^{\alpha_i} \not \in \FF[x^p]} \,  \pr_i}$.   \end{proof} 

 \begin{defn}\label{D:varr}  When $\chara(\FF) = 0$, set $\varrho_h = 1$.   When $\chara(\FF) = p > 0$, let 
$h = \lambda \pr_1^{\alpha_1} \cdots \pr_t^{\alpha_t}$ be the factorization of $h$, where the $\pr_i$ 
are the distinct monic prime factors given in Lemma \ref{L:defpi},  and $\lambda \in \FF^*$.  After possibly renumbering,  assume
 $\pr_i \not \in \FF[x^p]$ for $1\leq i\leq \ell$ and $\pr_j  \in \FF[x^p]$ for $\ell < j \leq t$ (in case $\ell =0$, there are no such $\pr_{i}$,  and in case $\ell=t$,  there are no such $\pr_{j}$).  For each   $1\leq i\leq \ell$, take $k_{i}\geq 0$ and $0\leq \overbar{\alpha_i}<p$ so that $\alpha_{i}=k_{i}p+\overbar{\alpha_i}$.   Let  
 \begin{equation}\label{eq:vrhodef} \varrho_h =  \pr_1^{k_{1}p}\cdots \pr_\ell^{k_{\ell}p} \pr_{\ell+1}^{\alpha_{\ell+1}} \cdots \pr_t^{\alpha_t}.\end{equation} 
 \end{defn} 
 
In the characteristic $p > 0$ case,  $\varrho_h$ is the unique monic polynomial 
 of maximal degree in $\FF[x^p]$ dividing $h$, and
\begin{equation}\label{eq:hdec}  h = \begin{cases}  \lambda \varrho_h  & \qquad \hbox{if} \ \ h \in \FF[x^p] \\
\lambda \pr_1^{\overbar{\alpha_1}} \cdots \pr_\ell^{\overbar{\alpha_\ell}} \varrho_h & \qquad \hbox{if} \ \ h \not \in \FF[x^p]. \\ \end{cases} \end{equation} 
  To avoid separating considerations into cases, often we will write $h= \break
\lambda \pr_1^{\overbar{\alpha_1}} \cdots \pr_\ell^{\overbar{\alpha_\ell}} \varrho_h$ with the understanding that
the product $\pr_1^{\overbar{\alpha_1}}\cdots \pr_\ell^{\overbar{\alpha_\ell}}$ should be interpreted as being  1  if $\ell=0$.   Whenever $h \in \FF^\ast$, then $h$ is as in the first option of \eqref{eq:hdec} with $\varrho_h = 1$.

\begin{thm}\label{T:norm} Regard $\A_h \subseteq \A_1$ as in Conventions \ref{con:gens}. 
Let  $\pi_h \in \DD=\FF[x]$ be as in  Lemma \ref{L:defpi},  and set $a_n = \pi_h h^{n-1} y^n$ for all $n \geq 1$.   
\begin{itemize}
\item[{\rm(a)}] Assume  $a \in\A_{1}$ and write $a=\sum_{i\geq 0}r_{i}y^{i}$ with $r_{i}\in\DD$. Then the  following hold:
\begin{enumerate}  
\item [{\rm (i)}]  If $\chara (\FF)=0$,  then $a \in \norm \iff  \pi_h h^{i-1} \mid r_i$ for all $i\geq 1$. Hence, 
$\norm = \DD \oplus \bigoplus_{n \geq 1} \DD a_n$. 
 \item [{\rm (ii)}]   If $\chara (\FF)=p > 0$, then $a \in \norm \iff$
\begin{enumerate}
\item[{$\bullet$}] for all $i \not \equiv 0 \modd p$, \ $ \pi_h h^{i-1} \mid r_i$ 
\item[{$\bullet$}] for all \, $i  \equiv 0 \modd p$, \ $i > 0$, \  $h^{i-1} \mid r_i'$, \, or equivalently,
$r_{i} \in  c_{i}\varrho_h^{p-1}h^{i-p}+ \FF[x^p]$ for some  $c_i \in \DD$ with  
$c_i' \in \DD\left(\frac{h}{\varrho_h}\right)^{p-1}$.   
\end{enumerate}  
\end{enumerate}
\noindent In particular,  $a = \sum_{i\geq 0} r_i y^i \in \norm$ if and only if  $r_i y^i \in \norm$
for all $i \geq 0$.
\item[{\rm(b)}] For all $\FF$ and $n\geq 1$,   $\DD a_n  \subset \norm$,   and
$h' a_n$ and $\frac{h}{\pi_h} a_n$ are in $\A_h$.   
\end{itemize}
\end{thm}
\begin{proof}
For (a), suppose $a = \sum_{i \geq 0} r_i y^i$, where $r_i \in \DD$ for all $i$.     We  will treat the characteristic 0 and $p$ cases together by adopting the convention that $p = 0$ when $\chara(\FF) = 0$.  In that case,  the statement
$i \not \equiv 0 \modd p$ simply means $i \neq 0$, while $i \equiv 0 \modd p$ means $i = 0$.  

Now $a \in \norm$  exactly
when $[a,x]$ and $[a,\hat y]$ are in $\A_h$.    In particular,  
\begin{equation}\label{eq:ax} [a,x] \in \A_h \iff \sum_{i \not \equiv 0 \modd p}  i r_i y^{i-1} \in \A_h  \iff h^{i-1}\mid r_i \ \ \,  \forall \ \ i \not \equiv 0 \modd p \end{equation} 
by \eqref{eq:Ahexpress}.   Hence, we may assume  $a = \sum_{i \not \equiv 0 \modd p} s_i h^{i-1}y^i + \sum_{i \equiv 0 \modd p} r_i y^i$
for some $s_i \in \DD$.     Since $[a,x] \in \A_h$, it follows that $[a,g] \in \A_h$ for all $g \in \DD$.  
Therefore, $[a, \hat y] = [a,yh]  \in \A_h \iff  [a, hy] \in \A_h$.    Now using Lemma \ref{lem:identity_y^nf},  we have
\begin{eqnarray*} [a, hy] &=&  \sum_{i \not \equiv 0 \modd p} [s_i h^{i-1}y^i,hy]  + \sum_{i \equiv 0 \modd p} [r_i y^i,hy]  \\
&=&  \sum_{i \not \equiv 0 \modd p} s_i h^{i-1}\sum_{j=1}^i {i \choose j} h^{(j)} y^{i-j+1}  - \sum_{i \not \equiv 0 \modd p} (s_i h^{i-1})'h y^i \\
&& \   -    \sum_{i \equiv 0 \modd p} r_i' h y^i .
\end{eqnarray*} 
Since by \eqref{eq:Ahexpress} all the terms in the first sum with $j \geq 2$ belong to $\A_h$, 
we have 
\begin{eqnarray}\label{eq:ahy} [a,hy] \in \A_h&\iff&\hspace{-.25cm} \sum_{i \not \equiv 0 \modd p} s_i h^{i-1}h' y^i 
- \sum_{i \not \equiv 0 \modd p} s_i' h^iy^i   -    \sum_{i \equiv 0 \modd p} r_i' h y^i \in \A_h \nonumber \\
&\iff &\hspace{-.25cm}  \sum_{i \not \equiv 0 \modd p} s_i h^{i-1}h' y^i 
-    \sum_{i \equiv 0 \modd p} r_i' h y^i \in \A_h,  
\end{eqnarray}  
as  $s_i'h^i y^i \in \A_h$ for all $i \not \equiv 0$,  again  using  \eqref{eq:Ahexpress}.  

From this we deduce that  $h^i$ must divide $s_i h^{i-1}h'$ for all $i \not \equiv 0 \modd p$; that is, 
$h$ must divide $s_i h'$ for all such $i$.   By Lemma \ref{L:defpi}, this means that $\pi_h$
divides $s_i$ for each $i \not \equiv 0 \modd p$, and in turn this says that $\pi_h h^{i-1}$ divides $r_{i}$   for all $i \not \equiv 0 \modd p$.   In particular,  (i)  and the first assertion  of (ii) hold.

Now from \eqref{eq:ahy}, we also see that $h^{i-1} \mid r_i'$ for all $i \equiv 0 \modd p$, $i > 0$.  
 Note that  $h^{i-1} =  h^{i-p}h^{p-1} = (\frac{h}{\varrho_h})^{p-1}\varrho_h^{p-1}h^{i-p}$.
 Hence, we may write $r_i' = d_i v_i$, where $d_i \in \DD(\frac{h}{\varrho_h})^{p-1}$ and $v_i =
 \varrho_h^{p-1}h^{i-p} \in \FF[x^p]$.     Since $d_i v_i \in \im \frac{d}{dx}=  
\sum_{j=0}^{p-2} \FF[x^p]x^j$ and $v_i \in \FF[x^p]$,  it follows that  $d_i \in  
\sum_{j=0}^{p-2} \FF[x^p]x^j$.    Therefore $d_i = c_i'$ for some $c_i \in \DD$, and
$(c_iv_i)' = c_i' v_i = d_i v_i = r_i'$.      This gives  $r_i \in c_i v_i + \FF[x^p] = 
c_i \varrho_h^{p-1}h^{i-p} + \FF[x^p]$, as in (ii).        That $ r_i y^i \in \norm$ 
for every $r_i$ of this form for $i \equiv 0 \modd p$, $i > 0$,  can be shown by direct computation.    
This proves the remaining parts of (a).

The first part of (b) is an immediate consequence of (a)  except when $n \equiv 0 \, \modd p$ and $\chara(\FF) = p > 0$.   For  $a_{kp} = \pi_h h^{kp-1} y^{kp}$ with $k \geq 1$,   observe that  $[ra_{kp}, f] = 0$ for all $r,f \in \DD$
since $y^{kp} \in \cent1$.    Moreover,
\begin{eqnarray*}  [ra_{kp}, hy] &=& h[r\pi_hh^{kp-1},y] y^{kp} = -h (r\pi_h h^{kp-1})' y^{kp} \\
&=& -(r \pi_h)' h^{kp}y^{kp} + r\pi_h h' h^{kp-1}y^{kp}, \end{eqnarray*}
which is in $\A_h$ by \eqref{eq:Ahexpress} and  the fact that $h$ divides $\pi_h h'$ by Lemma \ref{L:defpi}. Now $h' a_n = h' \pi_h h^{n-1}y^n \in \A_h$ is a consequence of that fact too, and   $\frac{h}{\pi_h}a_n = h^n y^n \in \A_h$ is clear.  \end{proof}

 \begin{remark}\label{R:norm}  It  follows from Theorem \ref{T:norm}  that when $\chara(\FF) = 0$ and $\frac{h}{\pi_h} \in \FF^*$, then $\norm = \A_h$.  
 
 If  $\chara(\FF) = p > 0$, we set 
\begin{align} \label{eq:normsplit} \begin{split}  \norm_{\not \equiv 0} &= \ \norm \cap \Big(\bigoplus_{i \not \equiv 0 \modd p} \DD y^i\Big), \\
 \norm_{\equiv 0} &= \ \norm \cap \mathsf{C}_{\A_1}(x). \end{split} \end{align} 
Then every   $a \in \norm$ has a unique expression $a = b+c$ with
$ b \in \norm_{\not \equiv 0}$ and 
$c \in \norm_{\equiv 0}.$
In particular, when $\frac{h}{\pi_h} \in \FF^*$,  then $b \in \A_h$.  
\end{remark}  \smallskip

\begin{section}{Derivations of $\A_1$}\label{derA1} \end{section}

We will use derivations of $\A_1$ heavily in our investigations of derivations of $\A_h$.
In the next result, we provide a quick proof of the known fact that the derivations
of $\A_1$ are inner in the $\chara(\FF) = 0$ case, in part to establish the notation we will adopt later.

\medskip

\subsection{$\der(\A_1)$ when $\chara(\FF) = 0$} \hfil  \smallskip
  
\begin{prop}\label{P:Weyl0}{\rm (Cf. \cite[Lem. 4.6.8]{Dix96}).}  Assume $\chara(\FF) = 0$.   Then every derivation of the Weyl
algebra $\A_1$ is inner. \end{prop}

\begin{proof}    Suppose $D \in \der(\A_1)$.  Assume that
 $D(x) = \sum_{i \geq 0} d_i y^i$, where  $d_i  \in \DD=\FF[x]$ for all $i$.  
Set 
$$u = \sum_{i \geq 0}   \frac{d_i}{i+1} y^{i+1}.$$
Then $\ad_u(x) =  \sum_{i \geq 0} d_i y^i  = D(x)$, so that 
$E = D-\ad_u \in \der(\A_1)$ has the property that 
$E(x) = 0$.        

Then from $[E(y), x] + [y, E(x)] = E(1) = 0$, we determine that
$[E(y),x] = 0$.     Thus,  
$E(y) \in  \mathsf{C}_{\A_1}(x)= \DD$ by Lemma \ref{L:clizer}.    
Since $E(y) \in \DD$ and $\chara(\FF) = 0$,  there exists a $w \in \DD$ such that 
$w' = -E(y)$.     Then $\ad_w(x) = 0 = E(x)$ and $\ad_w(y) = [w,y] = -w' = E(y)$.   
Therefore $D-\ad_u = E = \ad_w$ and 
$D = \ad_u  +  \ad_w \in \inder(\A_1)$.  Hence,  
$\der(\A_1) = \inder(\A_1)$.   \end{proof} \smallskip

\subsection{$\der(\A_1)$ when $\chara(\FF) = p > 0$}  \hfil
\smallskip
 
\subsubsection{The derivations $E_x$ and $E_y$} \hfil 

\smallskip

Over fields of characteristic $p > 0$, the derivations $ (\ad_x)^p =\ad_{x^p}$ and $(\ad_y)^p =\ad_{y^p}$ are 0 on the Weyl algebra  $\A_1$.   However, $\A_1$ has two special derivations
 $E_x$ and $E_y$,  which are specified by
\begin{equation}\label{eq:ExEy} E_x(x) = y^{p-1},  \ \   E_x(y) = 0, \quad \hbox{\rm and} \quad   E_y(x) = 0, \ \ E_y(y) = x^{p-1}.
\end{equation}
Then  $zE_x$ and $zE_y$ are also derivations of $\A_1$ for every $z \in \cent1 = \FF[x^p,y^p]$.   Let
$\varphi$ be the anti-automorphism of $\A_1$ defined by
\begin{equation}\label{eq:psidef} \varphi(x) = y,  \quad \varphi(y) = x. \end{equation}    Then 
\begin{equation}\label{eq:psiEpsiinv}  \varphi E_x  \varphi = \varphi E_x  \varphi^{-1} = E_y,  \quad \hbox{\rm and} \quad  \varphi E_y  \varphi =  \varphi  E_y   \varphi^{-1} = E_x.\end{equation} 

\begin{lemma}\label{lem:dersum}  Assume $\A_1$ is the Weyl algebra over $\FF$, where $\chara(\FF) = p > 0$. 
Then 
$$\der(\A_1) =\cent1 E_x + \cent1 E_y +  \inder(\A_1).$$ \end{lemma} 
\begin{proof}   The right side is clearly contained in $\der(\A_1)$.  For the other containment, suppose $D \in \der(\A_1)$, and assume that $D(x) = \sum_{i \geq 0} d_i y^i$, where  $d_i  \in \DD$ for all $i$.  Set 
$$b = \sum_{ {i \not \equiv -1 \modd p}} \frac{d_i}{i+1} y^{i+1}.$$
Then $\ad_b(x) =  \sum_{i \not \equiv -1\modd p} d_i y^i  $, so that 
$E = D-\ad_b \in \der(\A_1)$ has the property that  \ $E(x) = \sum_{{ i \equiv -1 \modd p}}  d_i  y^i.$ 
     
Suppose that
$E(y) = \sum_{j \geq 0} e_j y^j$,  where $e_j \in \DD$ for all $j$.   
Then  $$0 = E(1) = [E(y),x] + [y, E(x)] = \sum_{j \geq 1} j e_j  y^{j-1} + \sum_{{ i \equiv -1 \modd p}} d_i' y^i, $$ from which we determine that 
$d_i' = 0$ for all  $i \equiv -1 \modd p$,  and $e_j = 0$
for all $j \not \equiv 0 \ \modd p$.      The first  implies   
$d_i \in \FF[x^p]$ for all such $i$, so that   $w = \sum_{i \equiv -1 \modd p} d_i y^{i-(p-1)}\in \cent1$ 
and $E(x) = w y^{p-1} = w E_x(x).$     As a result,  $F = E-wE_x$ has the property that $F(x) = 0$
and $F(y) = \sum_{j \equiv 0 \modd p} e_j y^j$.  

Now it is a  direct consequence of the decomposition $\DD = \bigoplus_{j=0}^{p-1} \FF[x^p] x^j$ and the fact that $\im \frac{d}{dx} = \bigoplus_{j=0}^{p-2} \FF[x^p] x^j$ that every $e \in \DD$ can be expressed as $e= cx^{p-1} -r'$ for some $r\in \DD$ and  a unique $c\in \FF[x^p]$.    Applying that result  to each 
$e_j$, we have that there exist  $c_j \in \FF[x^p]$ and $r_j \in \DD$, so that $e_j = c_j x^{p-1} - r_j'$.  
Then $F(y) = \sum_{j \equiv 0 \modd p} e_j y^j = \left (\sum_{j\equiv 0 \modd p} c_j y^j\right)x^{p-1} - \sum_{j \equiv 0 \modd p}r_j' y^j$.   Setting  $z = \sum_{j\equiv 0 \modd p} c_j y^j$ and $ c =  \sum_{j \equiv 0 \modd p}r_j y^j$, we see that $z \in \cent1$ and  $(F - zE_y - \ad_c)(x) = 0 = (F - zE_y - \ad_c)(y)$.      Consequently, 
$D = w E_x + z E_y + \ad_b + \ad_c \in \cent1 E_x + \cent1 E_y +  \inder(\A_1).$  \end{proof}    

\subsubsection{The action of $E_x$ and $E_y$ on $\A_1$} \hfil   \smallskip

The next lemma describes how $E_x$ and $E_y$ act on various elements of $\A_1$.  

\begin{lemma}\label{lem:ee}  Assume $\chara(\FF) = p >0$.  When $g \in \FF[x]$, let $g^{(k)} = \left(\frac{d}{dx}\right)^{k}(g)$,
and when $g \in \FF[y]$, let $g^{(k)} = \left(\frac{d}{dy}\right)^{k}(g)$.  Assume $\varphi$ is the anti-automorphism in \eqref{eq:psidef},    
and let \ $\parp: \FF[x]  \rightarrow \FF[x]$ \ be the $\FF$-linear map defined by
\begin{equation}\label{eq:projdef}
\parp\left(\sum_{i=0}^{p-1}r_i x^i \right)=\sum_{i=0}^{p-1}\frac{d}{d(x^p)}\big(r_i\big)\,x^i,\quad\ \  \mbox{for \  $r_i\in\FF[x^p]$.}
\end{equation}  \noindent Then the following hold in $\A_1$:

\begin{eqnarray*} 
&&\hspace{-1.25cm}\hbox{\rm(a)\ \ } E_x(x^n) = \sum_{k=1}^{p} {n \choose k} x^{n-k} (y^{p-1})^{(k-1)} \quad \quad \hbox{\rm for} \ \, n \geq 1; \\
&&\hspace{-1.25cm}\hbox{\rm(b)\ \ }E_x(g) = \sum_{k=1}^{p-1} \frac{(-1)^{k-1}}{k} g^{(k)} y^{p-k} -\parp(g)  \quad \hbox{\rm for all} \ \, g \in \FF[x]; \\
&&\hspace{-1.25cm}\hbox{\rm(c)\ \ }E_x  = - \frac{d}{d(x^p)}\ \ \hbox{\rm on} \ \, \FF[x^p] \ \ \hbox{\rm  and}  \ \  E_x(g^{p}) = -(g')^p  \ \ \hbox{\rm for all}
 \ \, g \in \FF[x]; \\
&&\hspace{-1.25cm}\hbox{\rm(d)\ \ }E_y(g) =  \sum_{k=1}^{p-1} \frac{(-1)^{k-1}}{k} x^{p-k}g^{(k)} -\varphi \parp (g(x)) \ \ \hbox{\rm for all} \ \, g \in \FF[y]; \\
&&\hspace{-1.25cm}\hbox{\rm(e)\ \ }E_y(\hat y) = E_y(y)h = x^{p-1} h; \\
&&\hspace{-1.25cm}\hbox{\rm(f)\ \ }E_x(\hat y) = h' y^p  + \sum_{k=1}^{p-2} \frac{(-1)^{k-1}}{(k+1)k} h^{(k+1)}y^{p-k}   - \parp(h) \, y  - \parp(h').
\end{eqnarray*}
\end{lemma}

\noindent{Proof.} \  Part (a) can be shown using induction on $n$ (the case $n=1$ saying $E_x(x) = y^{p-1}$).  
Assume $E_x(x^n) = \sum_{k=1}^n {n \choose k} x^{n-k} (y^{p-1})^{(k-1)}$, and 
substitute that expression into $E_x(x^{n+1}) = E_x(x^n)x +  x^n E_x(x)$.
Applying  the fact that $fx = xf + \frac{d}{dy}(f)$ for all $f \in \FF[y]$  to the first summand and simplifying gives  the desired expression for the  $n+1$ case.   Since 
$(y^{p-1})^{(k-1)} = 0$ for all $k > p$,  the index of summation need only go up to $p$.   

For (b),  we have using  ${{p-1}\choose {k-1}} = (-1)^{k-1}$ and  $(p-1)! = -1$  that 
\begin{align*}E_x(x^n) &=  \sum_{k=1}^{p} {n \choose k} x^{n-k} (y^{p-1})^{(k-1)} \\
&=  \sum_{k=1}^{p-1}\frac{(x^n)^{(k)}}{k!} {{p-1}\choose {k-1}} (k-1)! y^{p-k} - {n \choose p} x^{n-p}\\ 
&=  \sum_{k=1}^{p-1} \frac{(-1)^{k-1}}{k}(x^n)^{(k)}y^{p-k}  - {n \choose p} x^{n-p}. 
 \end{align*}
 
\noindent Now if $n=jp+\ell$  with $0\leq \ell <p$,  then $x^{n}= (x^{p})^j x^\ell$ and ${n \choose p} = j$, so $\parp (x^n)={n \choose p} x^{n-p}$.  
Thus, 
$$E_x(x^n) =  \sum_{k=1}^{p-1} \frac{(-1)^{k-1}}{k} (x^n)^{(k)}y^{p-k} -\parp(x^n),$$  
where $\parp$ is as in \eqref{eq:projdef}.   This, together with the linearity of derivations,  implies (b).  
 
As a special case of  (b),  we have  $E_x(x^{jp}) =  -j x^{(j-1)p}$ for all $j \geq 1$ so  that 
$E_x = - \frac{d}{d(x^p)}$ on $\FF[x^p]$.   
In particular, if  $g(x) =  \sum_{j\geq 0} \gamma_j x^j$,  then,  as claimed in (c),   
$$E_x(g^p) = \sum_{j\geq 0} \gamma_j^p E_x(x^{jp}) = -   \sum_{j\geq 1}  j\gamma_j^p  x^{(j-1)p}
= - \sum_{j\geq 1} j^p \gamma_j^p  x^{(j-1)p}
= -(g')^p.$$   

For (d), applying  the anti-automorphism $\varphi$ in \eqref{eq:psidef}   which interchanges $x$ and $y$, and using  
\eqref{eq:psiEpsiinv}, we have  $E_y(g(y)) =  \varphi E_x  \varphi^{-1}(g(y)) = \varphi (E_x(g(x)))$
for $g(y) \in \FF[y]$, and so (d) now follows from applying $\varphi$ to (b). 

Part (e) is apparent, and  (f) can be derived from the following calculation which uses the relation  $[y, \parp(f)]=\parp(f')$, for $f\in\DD$:
\begin{eqnarray*} \  \  E_x(\hat y) &=& E_x(yh) = yE_x(h) =  y \, \sum_{k=1}^{p-1} \frac{(-1)^{k-1}}{k} h^{(k)}y^{p-k} - y\parp(h)\\
&=&\sum_{k=1}^{p-1} \frac{(-1)^{k-1}}{k} \left(h^{(k)}y +h^{(k+1)}\right)y^{p-k} - \parp(h) \, y  - \parp(h') \\ 
 &=& h' y^p  + \sum_{k=1}^{p-2} \frac{(-1)^{k-1}}{(k+1)k} h^{(k+1)}y^{p-k} - \parp(h) \, y  - \parp(h'). \hspace{1.8cm} \square \end{eqnarray*}

We have the following consequence of this result.

\begin{thm}\label{thm:derA1}  Assume $\A_1$ is the Weyl algebra over $\FF$, where $\chara(\FF) = p > 0$. 
Then 
\begin{itemize}
\item[{\rm (a)}]  $\der(\A_1) =\cent1 E_x \oplus \cent1 E_y \oplus  \inder(\A_1),$ 
where $E_x,E_y \in \der(\A_1)$ are given by $E_x(x) = y^{p-1},\,E_x(y) = 0,\,
E_y(x) = 0, \, E_y(y) = x^{p-1}$.
\item[{\rm (b)}]  $\hoch(\A_1) =  \der(\A_1)/\inder(\A_1) \cong \der(\FF[t_1, t_2])$ as Lie algebras, 
where $t_1 = x^p$, $t_2 = y^p$.  \end{itemize}  \end{thm} 

\begin{proof} In Lemma \ref{lem:dersum}, we  have established that $\der (\A_1)$ is the sum of the terms on the right side of (a).   Suppose $D = wE_x + zE_y + \ad_a = 0$ for some $a \in \A_1$ and $z,w \in \cent1$.   Applying $D$ to $x^p$ and using
the fact that $x^p$ is central, we have from Lemma \ref{lem:ee}\,(c) that
$0 = D(x^p) = -w$.  Similarly, applying $D$ to $y^p$ gives $z = 0$.   Hence $\ad_a = 0$ also, and the
sum in (a) is direct.  

The map  $\mathsf{Res}: \der(\A_1) \rightarrow \der(\cent1)$  given by restricting 
a derivation of $\A_1$  to the center $\cent1 = \FF[t_1, t_2]$, where $t_1 = x^p, t_2 = y^p$, is clearly a morphism
of Lie algebras.   It follows from Lemma \ref{lem:ee} that $\mathsf{Res}(E_x)  = -\frac{d}{dt_1}$ and 
$ \mathsf{Res}(E_y) = - \frac{d}{dt_2}$.   
Hence $w E_x  + zE_y + \ad_a  \mapsto  -w \frac{d}{dt_1} - z\frac{d}{dt_2}$ for all $w,z\in\cent1$, which
shows the map is onto.   Now $\inder(\A_1)$ is in the kernel.   But since every
$D \in \der(\A_1)$ has the form  $D =  w E_x  + z E_y + \ad_a$, we see the kernel is
exactly $\inder(\A_1)$.  \end{proof}  \smallskip

\begin{remark}\label{rem:dt}  It is well known that $\der( \FF[t_1,t_2])$ is a free $\FF[t_1,t_2]$-module
of rank $2$ with basis $\frac{d}{dt_1}$, $\frac{d}{dt_2}$.  This Lie algebra is often referred to as
the  \emph{Witt algebra} in 2 variables.  A.\ Solotar and M.\ Su\'arez-\'Alvarez have pointed out to us one could alternately use the fact that  $\A_1$ is Azumaya over its center, combined with a result on the homology of Azumaya algebras  in \cite{CW}  and  the Van den Bergh duality between homology and cohomology (see \cite{rB09}),  to conclude that $\hoch(\A_1)$  is free of rank 2 over the center $\cent1$ when
$\chara(\FF) = p > 0$.  Theorem \ref{thm:derA1},  which also establishes this result,  identifies  explicit generators $E_x$ and $E_y$ for $\hoch(\A_1)$  over $\cent1$.    \end{remark}

\subsubsection{Lie brackets in $\der(\A_1)$ when  $\chara(\FF) = p > 0$}   \hfil    \smallskip

Next we determine the multiplication in $\der(\A_1)$.  
\begin{lemma}\label{lem:varpi} Assume $\chara(\FF) = p > 0$.  Then $[E_x,E_y]= \ad_\varpi$  where 
\begin{equation}\label{eq:varpi} \varpi = \sum_{n=1}^{p-1} \frac {(p-1-n)! }{n} x^n y^n.\end{equation}
\end{lemma}
\begin{proof}   It suffices to compute the action of $[E_x,E_y]$ on $x$ and $y$.  Using (a) of Lemma \ref{lem:ee} and the fact that ${{p-1} \choose k} = (-1)^k$ for $0 \leq k \leq p-1$, we have
\begin{align*}[E_x,E_y](y) &= E_x(x^{p-1}) = \sum_{k=1}^{p-1}{{p-1}\choose k} x^{p-1-k}(y^{p-1})^{(k-1)} \\
&= -\sum_{k=1}^{p-1}(k-1)!\,x^{p-1-k}y^{p-k} = - \sum_{n=1}^{p-1} (p-1-n)! \,x^{n-1} y^{n}.
\end{align*} 
Then 
$$ [E_x,E_y](x) = -E_y(y^{p-1}) =  \sum_{n=1}^{p-1}  (p-1-n)!\, x^{n} y^{n-1}$$
upon applying $\varphi$ to the relation above.   However, if $\varpi$ is as in \eqref{eq:varpi}, then   
$$\ad_\varpi (x) = \sum_{n=1} ^{p-1} (p-1-n)! \, x^n y^{n-1} \quad \hbox{\rm and} \quad
\ad_\varpi(y)  =  -\sum_{n=1}^{p-1} (p-1-n)! \, x^{n-1} y^n.$$
Thus, $[E_x,E_y] = \ad_\varpi$,  as desired.   \end{proof} 
\smallskip

Products  in $\der(\A_1)$  can now be described using this result. 

\begin{lemma}\label{lem:derA1prod} Assume $\chara(\FF) = p > 0$.  For all $D,E \in \der(\A_1)$, $a \in \A_1$, $w,z \in \cent1$, we have 
\begin{itemize}
\item  $[D, \ad_a] = \ad_{D(a)}$, 
\item   $z \ad_a = \ad_{za}$, 
\item $ [wD, zE]  =  wD(z)E - zE(w)D  + wz[D,E]$,
\item  $[wE_x, zE_y]  = wE_x(z)E_y - zE_y(w)E_x + wz\,\ad_{\varpi}$, with $\varpi$ as in \eqref{eq:varpi}.
\end{itemize}
\end{lemma}
   
\begin{section} {Generalities on Derivations of $\A_h$} \end{section} 
We turn our attention now to the Lie algebra $\der(\A_h)$ of 
$\FF$-linear derivations of $\A_h$  for arbitrary $0 \neq h\in \DD =\FF[x]$ and arbitrary $\FF$. 
Throughout, we view $\A_h$ as a subalgebra of $\A_1$ as in Conventions \ref{con:gens},  and  apply  the results we have just established in Sections 3.1 and 3.2 on $\der(\A_1)$ to derive information about $\der(\A_h)$.

We begin by determining when a derivation of $\A_h$ extends to one of $\A_1$. We then define the derivations $D_e$,  $e \in \mathsf{C}_{\A_h}(x)$, and introduce the element $a_0$,  which belongs to a localization of $\A_1$ and is a natural extension of the elements $a_n = \pi_h h^{n-1} \in \norm$ for $n \geq 1$.  The main results of this section are  Theorem \ref{thm:derdecomp}, which describes  a decomposition of $\der(\A_h)$ into a sum of Lie subalgebras  for arbitrary $\FF$, and Theorem \ref{L:brackets},  which gives expressions for various products involving the derivations $D_g$, $g \in \DD$,  and $\ad_{ra_n}$ for $n \geq 0$ and $r \in \DD$.    This sets the stage for Section \ref{sec:char0Der}, where we show that these derivations along with the inner derivations generate $\der(\A_h)$ when $\chara(\FF) = 0$.   
\medskip
 
\subsection{Extensions of  derivations} \hfil
\smallskip

To determine a necessary and sufficient condition for
a derivation of $\A_h$ to extend to a derivation of $\A_1$, 
we require a basic result about derivations of $\A_h$,  which can be shown using
\cite[Exer.~2ZC]{GW04}.  \smallskip 

\begin{lemma}\label{lem:derext}    Fix  $u,v \in \A_h$.     Let
$d: \FF[x] \rightarrow \A_h$ be the unique derivation such that $d(x) = u$.      There is 
a derivation $D \in \der(\A_h)$ such that $D(x) = d(x) = u$  and $D(\hat y) = v$ if and only
if $[v,x] + [\hat y, u] = d(h)$.   If such a derivation exists, it is unique.  \end{lemma} 

In the next result, we will use the fact that $D(h) \in \A_h h = h \A_h$ for every $D \in \der(\A_h)$.  This follows from the
computation $D(h) = [D(\hat y),x] + [\hat y,D(x)]$ and the fact \cite[Lem.\ 6.1]{BLO1} that $[\A_h,\A_h] \subseteq h\A_h$.  
 \smallskip 
 
\begin{thm}\label{thm:Dext}   Regard $\A_h \subseteq \A_1$ as in Conventions \ref{con:gens}.  
\begin{enumerate}
\item[{\rm (i)}]  A derivation $D \in \der(\A_h)$  extends to a derivation $\widetilde D$ of $\A_1$ if and only if $D(\hat y) \in \A_1 h$.
In particular, if $D(\hat y) = ah$ and $D(h) = bh$ for $a \in \A_1$ and $b \in \A_h$,
then $\widetilde D$ is determined by 
$$\widetilde D(x) = D(x),  \qquad  \widetilde D(y) = a - yb.$$
\item[{\rm (ii)}]  Suppose that $D, E \in \der (\A_1)$ restrict to derivations of $\A_h$ and $D = E$ as derivations of $\A_h$.  Then $D = E$ as derivations of $\A_1$.
\end{enumerate}
\end{thm}

\begin{proof} (i)  Assume  $D \in \der(\A_h).$    If $D$ extends to a derivation $\widetilde D$ of $\A_1$, then
$$D(\hat y) = \widetilde D(\hat y) =  \widetilde D(yh) =  \widetilde D(y)h + yD(h) \in \A_1 h.$$
Conversely,    suppose $D(\hat y) = a h$ where $a \in \A_1$.   We may assume
$D(h) = b h$ where $b \in \A_h$.    By Lemma \ref{lem:derext}  applied to $\A_1$ 
(and so with $ \widetilde D$ replacing $D$ and  $y$ replacing $\hat y$ in quoting that result)  there
is a unique derivation $\widetilde D$ of $\A_1$ with 
$$ \widetilde D(x) = D(x),  \qquad   \widetilde D(y) = a - yb$$
if and only if  $[a-yb, x] + [y, D(x)] = D(1) = 0$.    Since $\A_1$ is a domain,
it suffices to show that 
$\big([a-yb,x] + [y,D(x)] \big)h = 0$.     For this, we have 
\begin{align*}  [a-yb,x] h +& [y, D(x)] h = [ah,x] - [ybh,x] +   [y, D(x)] h \\
&=   [D(\hat y), x] -[yD(h),x]  + [\hat y,D(x)]  -  y[h, D(x)] \\
&=  [D(\hat y), x] + [\hat y, D(x)] - [y,x]D(h)   - y[D(h),x] - y[h,D(x)]  \\ 
&=  D([\hat y,x]) - D(h) - yD([h,x])   
=  0. \end{align*} 
Note that $\widetilde D$ thus defined restricts  to $D$ on $\A_h$.

(ii) Now assume that $D, E \in \der (\A_1)$ both restrict to derivations of $\A_h$ and  $D =E$ as derivations of $\A_h$.  The assumptions imply that 
$D(r) = E(r)$ for all $r \in \DD$, and $D(yh) = D(\hat y) = E(\hat y) = E(yh)$.  Therefore, 
$$D(y)h + yD(h) = E(y)h + yE(h),$$
and so $D(y) h = E(y) h$.  Since $h \neq 0$, we have $D(y) = E(y)$.  \end{proof} 
\medskip

For any $a\in \norm$, $\ad_a$ is a derivation of $\A_h$,  and if  $a$ happens
to belong to $\A_h$,  then 
 $[D,\ad_a] = 
\ad_{D(a)}$ for any derivation $D \in \der(\A_h)$.   However, if   $a \in \norm \setminus \A_h$,   then $D(a)$ may not be defined.  
This can be remedied in the following way.      

Recall from  \cite[Cor.\ 4.3]{BLO1}  that  
\begin{equation} \Sigma = \{h^m \mid m \geq 0\}  \end{equation}  is a left and a right Ore set in 
both $\A_1$ and $\A_h \subseteq \A_1$, and the corresponding localizations
$\A_1 \Sigma^{-1} = \A_h \Sigma^{-1}$   are equal.   It is well known that derivations extend under localization.
In particular,  if $D \in \der(\A_h)$, then $D$ extends uniquely to a derivation $\widetilde D$ of 
$\A_h \Sigma^{-1} = \A_1 \Sigma^{-1}$,  with $\widetilde D(h^{-1}) = -h^{-1}D(h)h^{-1}$.
\smallskip

\begin{lemma}\label{lem:extloc}  Suppose $D \in \der(\A_h)$, and let $\widetilde D$ be the extension of 
$D$ to a derivation  of $\A_1 \Sigma^{-1}$.    Then 
$[D, \ad_a] = \ad_{\widetilde D(a)}$ for all $a \in \mathsf{N}_{\A_1 \Sigma^{-1}}(\A_h)$,
and $\widetilde D(a) \in  \mathsf{N}_{\A_1 \Sigma^{-1}}(\A_h)$.   In particular,
$\widetilde D(a) \in  \mathsf{N}_{\A_1 \Sigma^{-1}}(\A_h)$ for all $a \in \norm$. 
\end{lemma} 

\begin{proof}  Assume $b \in \A_h \subseteq \A_1$ and $a  \in
 \mathsf{N}_{\A_1\Sigma^{-1}}(\A_h)$.  Then $[a,b] \in \A_h$
and $D([a,b])  = \widetilde D([a,b]) = [\widetilde D(a), b] + [a,D(b)]$ so that 
\begin{equation}\label{eq:Dad} 
[D, \ad_a ] (b) =  D([a,b]) - [a,D(b)] = [ \widetilde D (a), b] = \ad_{\widetilde D(a)}(b).  
\end{equation}
Since $ [ \widetilde D(a), b] =  [D, \ad_a ] (b) \in \A_h$, it is clear that $\widetilde D(a) \in \mathsf{N}_{\A_1 \Sigma^{-1}}(\A_h)$.
\end{proof} 
\smallskip

\subsection{The derivations\, $D_e$}\label{sec:Dders}  \hfil
\smallskip

Lemma \ref{lem:derext}  implies that for each $e \in  \mathsf{C}_{\A_h}(x)$ 
there is a unique derivation $D_e$ of $\A_h$
with $D_e(x) = 0$ and $D_e(\hat y) = e$.  
Such a derivation satisfies $D_e(f) \in \mathsf{C}_{\A_h}(x)$ for all $f \in \mathsf{C}_{\A_h}(x)$,
since $0 = D_e([x,f]) = [x,D_e(f)]$.   These derivations play a prominent role in our  investigations and also can be used to construct automorphisms of $\A_h$.  \medskip
 
\begin{prop}\label{prop:autg}  Assume $e,f  \in  \mathsf{C}_{\A_h}(x) = \centh \DD$.  Then 
\begin{itemize}   
\item [{\rm (i)}]  $[D_e,D_f] = D_c \,$, where $c = D_e(f) - D_f(e) \in \mathsf{C}_{\A_h}(x)$, so that
$\mathcal D_{\mathsf C} = \{ D_e \mid e \in \mathsf{C}_{\A_h}(x)\}$ is a Lie subalgebra of $\der(\A_h)$. 
\item[{\rm(ii)}]  $D_{\delta(g)} = - \ad_g$ for all $g \in \DD$, where $\delta(g) = g'h$.  In particular, $D_h = - \ad_x$. 
\item[{\rm (iii)}]  When $\chara(\FF) = 0$, then   $\mathcal D_{\mathsf C}  =  \{D_g \mid g \in \DD \}$.  Moreover,      
\begin{enumerate} \item[{\rm (a)}]  $\mathcal D_{\mathsf C}$ is abelian, and $D_g$ is locally nilpotent for all $g \in \DD$.  
\item[{\rm (b)}]   For any $g \in \DD$,   $\phi_g =  \exp(D_g) = \displaystyle{\sum_{n=0}^\infty \frac{(D_g)^n}{n!}}$  
is an automorphism of $\A_h$ with inverse $\phi_{-g} = \exp(-D_g)$,
and $\{ \phi_g  \mid  g \in \DD\}$ is an abelian subgroup
of $\mathsf{Aut}_\FF(\A_h)$ isomorphic to $(\DD,+)$. 
\end{enumerate} 
\end{itemize}  \end{prop} 

\begin{remark}\label{rem:autphi}  
The automorphism $\phi_g$ satisfies  $\phi_g(x) = x$
and $\phi_g(\hat y) = \hat y + g$, and $\phi_f \circ \phi_g = \phi_{f+g}$ holds for all $f,g \in \DD$. 
 In  \cite[Thm.~8.3\,(iv)]{BLO1}\, it is shown that if $\phi_g$ is
 defined by these expressions for the algebra $\A_h$ over any field,
 then  $\{\phi_g \mid g\in \DD\}$ forms a normal subgroup
 of $\mathsf{Aut}_\FF(\A_h)$ isomorphic to $(\DD,+)$.   
\end{remark}   
 
Every  derivation $\ad_c$,  with $c \in \norm_{\equiv 0}$  as  in  \eqref{eq:normsplit}, 
can be realized as a derivation in $\mathcal D_{\mathsf C}$ as follows. 
 
\begin{lemma}\label{L:adc=Df} Assume $\chara(\FF) = p > 0$  and  $c\in\norm_{\equiv 0}$.  Then there is $f \in 
\mathsf{C}_{\A_h}(x)$ such that $\ad_c = D_f$.  
\end{lemma}  

\begin{proof}  
Set $f=\ad_c(\hat y)$. Then $f\in\A_h$ because $c\in\norm$. Moreover, as $c\in\mathsf{C}_{\A_1}(x)$, it follows that $[f, x]=[\ad_c(\hat y), x]=\ad_c([\hat y, x])=0$, so $f\in\mathsf{C}_{\A_h}(x)$. This implies $\ad_c = D_f$, as required.\end{proof}  

The derivations $D_g$ with $g \in \DD$ can be used to give a decomposition of $\der(\A_h)$,  as the next result shows.  
\smallskip

\begin{thm}\label{thm:derdecomp}   Assume $\FF$ is arbitrary, and regard $\A_h \subseteq \A_1$.  Then
\begin{equation}\label{eq:DRE}  \mathcal D_{\DD}  = \{ D_g \mid g \in \DD\} \quad \hbox{\rm and} \quad \mathcal E  = \{ F \in \der(\A_1) \mid F(\A_h) \subseteq \A_h \} \end{equation}  are Lie subalgebras of $\der (\A_h)$,  $\mathcal D_{\DD}$ is abelian,  and $\der (\A_h) = \mathcal D_{\DD}  + \mathcal E$.
  \end{thm} 

\begin{proof}   
It is clear that $\mathcal D_\DD$ and $\mathcal E$ are Lie subalgebras of $\der (\A_h)$,  and 
 $\mathcal D_\DD$ is abelian (compare Proposition \ref{prop:autg} (i)). 
Assume  $D \in \der(\A_h)$.   Then $D(\hat y) = \sum_{j \geq 0}  r_j \hat y^j$, where
$r_j \in \DD$ for each $j$.  Now 
$D-D_{r_0} \in \der(\A_h)$,  and 
$$(D-D_{r_0})(\hat y) = \sum_{j \geq 1} r_j \hat y^j = \sum_{j \geq 1} r_j \hat y^{j-1}yh \in \A_1 h.$$
 Thus by Theorem \ref{thm:Dext}, the derivation $D-D_{r_0} \in \der(\A_h)$  extends to a derivation $E  \in \der(\A_1)$ such that  $D = D_{r_0}+E$, where $E$ belongs to $\mathcal E$.       \end{proof}     
 
 The derivations $D_g$ extend to derivations of  $\A_1 \Sigma^{-1}$, as the next result shows.

\begin{lemma}\label{lem:extendD_g}
For $g \in \DD$, the derivation $D_g \in \der (\A_h)$ extends uniquely to a derivation $\widetilde D_g$ of $\A_1 \Sigma^{-1}$ with $\widetilde D_g(\DD\Sigma^{-1}) = 0$,   $\widetilde D_g (y) = gh^{-1}$,  and $[D_g, \ad_a] = \ad_{\widetilde  D_g(a)}$,  for
all $a \in \norm$,  where 
$\widetilde D_{g}(a) \in \mathsf{N}_{\A_1 \Sigma^{-1}}(\A_h)$.  
\end{lemma}

\begin{proof}
It is clear that $D_g$ extends uniquely to a derivation $\widetilde D_g$ of $\A_1 \Sigma^{-1}$, and $\widetilde D_g(h^{-1})=-h^{-1}D_{g}(h)h^{-1}=0$.  Then it follows that  
\begin{equation}\label{eq:Dgy} \widetilde D_g (y) = \widetilde D_g (\hat y h^{-1}) = \widetilde D_g (\hat y)h^{-1} = D_g (\hat y)h^{-1} = gh^{-1}. \end{equation}
The final assertion is a direct 
consequence of Lemma \ref{lem:extloc}. 
\end{proof} 

\subsection{The element $a_0 = \pi_h h^{-1}$  in $\mathsf{N}_{\A_1 \Sigma^{-1}}(\A_h)$} \hfil 
\smallskip
 
Let  $\widetilde D_1$ be the extension of  the derivation $D_1$ to $\A_1 \Sigma^{-1}$,  and let 
$a_0 = \widetilde D_1(a_1)  = \pi_h h^{-1}  \in\mathsf{N}_{\A_1 \Sigma^{-1}}(\A_h).$    This definition fits naturally  
with the definition of the elements  $a_n = \pi_h h^{n-1} y^n \in \norm$  for   $n \geq 1$.  
Observe that in general  $\ad_{ra_0} \notin\mathcal E = \{F \in \der(\A_1) \mid
F(\A_h) \subseteq \A_h\}$.   Now since $\delta(r) = r' h$ for all $r \in \DD$,  the derivation $\delta$ extends to a derivation (again denoted
by $\delta$)  on $\DD \Sigma^{-1}$ with $\delta (h^{-1}) = -h' h^{-1}$.    
The linear transformation   given by
\begin{equation}\label{eq:del0}\delta_0: \DD \rightarrow \DD, \quad   r \mapsto \delta(r a_0)  =  (ra_0)' h  =   (r \pi_h h^{-1})'h  =  (r\pi_h)' - r \frac{\pi_h h'}{h} \end{equation}

\noindent will play a  special  role in what follows.    Since $h$ divides $\pi_h h'$ by  Lemma \ref{L:defpi},  it is evident that $\delta_0(\DD) = \delta(\DD a_0) \subseteq \DD$.   
\medskip
\begin{lemma}\label{L:del0}     For all $r \in \DD$, let   
$\delta_0(r)= \delta(r a_0)$ as in \eqref{eq:del0},  where $a_{0} = \pi_h h^{-1}  \in\mathsf{N}_{\A_1 \Sigma^{-1}}(\A_h).$ 
\begin{itemize} 
\item[{\rm (a)}] Then $\ad_{ra_{0}}=-D_{\delta(ra_0)} = -D_{\delta_0(r)} \in \mathcal D_{\DD}$  for all $r \in \DD$.  In particular,
$\ad_{a_0} =  -D_{\delta(a_0)}= -D_{\delta_0(1)}$  and $\degg(\delta(a_0)) < \degg h$.
\item[{\rm (b)}]  $\delta_0(rs) = \delta(rs a_0) = r \delta_0(s)  + r's \pi_h$.  In particular, $\delta_0(r) = r \delta_0(1) + r' \pi_h$,
where $\delta_0(1)  = \pi_h' - \frac{\pi_h h'}{h}.$
\end{itemize}
\end{lemma}
 
\begin{proof} For any $r \in \DD$, \ $\ad_{ra_{0}}(x)=0$ and 
\begin{equation*}
\ad_{ra_{0}}(\hat y)=[ra_0,y] h = -(ra_0)'h = -\delta(ra_0) = -\delta_0(r)  \in\DD.
\end{equation*}
Thus, $\ad_{ra_{0}}=-D_{\delta(ra_0)}=-D_{\delta_0(r)}  \in \mathcal D_\DD$, as these two derivations agree on a generating set of $\A_{h}$.  It can be seen from \eqref{eq:del0} that $\degg(\delta(a_0)) = \degg (\delta_0(1)) < \degg \pi_h\leq \degg h$.  Part\,(b) follows directly from the definitions.   \end{proof} 
\smallskip 

\subsection{Main result on products} \hfil
\smallskip

We can now state our main result on the Lie brackets  in $\hoch (\A_h)$.   Since $\mathsf{C}_{\A_h}(x) = \centh \DD$,  and $D_{z g}  =  z D_g$ for $z \in \centh, g \in \DD$,   we will  focus on products involving the derivations $D_g$ for $g \in \DD$.  This suffices when $\chara(\FF) = 0$,  since  $\centh = \FF1$ in that case.  When $\chara(\FF) = p > 0$, more general products will be considered in Section \ref{sec:prodsDerChar-p}.

\begin{thm}\label{L:brackets} Set $a_{-1} = 0$ and let  $a_{0}=\pi_{h}h^{-1}$.  For all $r \in \DD$, let   
$\delta_0(r)= \delta(r a_0) = (r\pi_h h^{-1})' h$ as in \eqref{eq:del0}. \begin{itemize}
\item[{\rm (a)}] For all $g,r \in\DD$ and $n\geq 0$,  we have  $[D_{g}, \ad_{ra_{n}}]=n\ad_{gra_{n-1}}
  = n \ad_{c a_{n-1}}$ in $\hoch(\A_h) =\der (\A_h) / \inder (\A_h)$, where $c$ is the remainder  of the 
division in $\DD$ of  $gr$ by $\frac{h}{\pi_h}$.  \item[{\rm (b)}]  For all $r, s\in\DD$ and all $m, n\geq 0$,  
$[ \ad_{ra_{m}}, \ad_{sa_{n}} ] = \ad_{qa_{m+n-1}} =  \ad_{da_{m+n-1}}$ in $\hoch(\A_h)$, where  
$q = mr \delta_0(s) -ns\delta_0(r)$,  and  $d$ is  the remainder
of the division in $\DD$ of $q$ by $\frac{h}{\pi_h}$.   \end{itemize}
\end{thm}

Our proof of this theorem, which we complete in Section \ref{sec:proofSec4MainThm}, will be the culmination of  a series of computational results. 

\subsection {The product $[D_g, \ad_a]$ for  $g \in \DD$ and $a \in \norm$}  \hfil 
\smallskip

\begin{lemma}\label{lem:D_g(y^n)-local}  Assume $D \in \der (\A_1 \Sigma^{-1})$ has the property that $D(x) = 0$ and $D( y) = f$, where $f \in \DD \Sigma^{-1}$.  Then 
$$D(y^n) = \sum_{k=1}^n {n \choose k} f^{(k-1)} y^{n-k}$$
for all $n \ge 1$, where $f^{(k-1)}$ denotes 
$(\frac{d}{dx})^{k-1}(f)$ and $f^{(0)} = f$.
\end{lemma}
\begin{proof} The assertion holds for $n = 1$ since $D(y) = f$.   
For larger $n$,  it follows by induction using the fact that
$ys=sy + s'$ for $s \in \DD \Sigma^{-1}$. 
\end{proof} 

Next we compute $\widetilde D_g$ on certain elements.  Ultimately, this will enable us to calculate
$[D_g, \ad_{ra_n}]$.
\begin{cor}\label{cor:D_g(a_n)}
Let $g,r \in \DD$ and  assume  $a_n = \pi_h h^{n-1} y^n$ for $n \ge 1$. Let $\widetilde D_g$ be the extension of $D_g$ to $\A_1 \Sigma^{-1}$ as in Lemma \ref{lem:extendD_g}.   Then 
\begin{itemize} 
\item[{\rm (a)}]  $\widetilde D_g(ry^n) = r \sum_{k=1}^n{n \choose k} (gh^{-1})^{(k-1)} y^{n-k}$.  
\item[{\rm (b)}]  $\widetilde D_g (r a_n) =    r\pi_h (gh^{-1})^{(n-1)}h^{n-1}   +  \sum_{k=1}^{n-1} {n \choose k} (gh^{-1})^{(k-1)} h^k r a_{n-k}.$
\item[{\rm (c)}]  Assume $\chara(\FF) = p > 0$. Then  $D_g(\ze) = \left(g h^{p-1}\right)^{(p-1)},$  where $\ze= h^p y^p \in \centh$.  
\end{itemize} 
\end{cor}
\noindent \emph{Proof.} \ Part (a) is immediate from  Lemma \ref{lem:D_g(y^n)-local},  since $\widetilde D_g(x) = 0$  and $\widetilde D_g(y) = gh^{-1}$ by \eqref{eq:Dgy}.    For (b),  we have from  part (a)

\begin{align*}
\widetilde D_g (r a_n) & =   r\pi_h h^{n-1} \sum_{k=1}^n {n \choose k} (gh^{-1})^{(k-1)} y^{n-k} \\ 
&=  r\pi_h (gh^{-1})^{(n-1)}  h^{n-1}   + \sum_{k=1}^{n-1} {n \choose k} (gh^{-1})^{(k-1)} h^k    r  \pi_h  h^{n-k-1} y^{n-k} \\
&=   r  \pi_h (gh^{-1})^{(n-1)} h^{n-1}  +  \sum_{k=1}^{n-1} {n \choose k} (gh^{-1})^{(k-1)} h^k r a_{n-k}.     
\end{align*}  
Item (c) is a consequence of the calculation
\begin{equation*} D_g(\ze) = h^p \sum_{k=1}^p {p \choose k} (gh^{-1})^{(k-1)} y^{p-k} 
= h^p(gh^{-1})^{(p-1)} = (gh^{p-1})^{(p-1)}. \hspace{.4cm} \square \end{equation*}
 
\begin{lemma}\label{lem:kDeriv_gh-inverse}
Let $g\in\DD$ and $k\geq 0$.   Then, there exist $r_{1}, \ldots, r_{k+1}\in\DD$ such that 
$\left( gh^{-1}\right)^{(k)}=\sum_{i=1}^{k+1}r_{i}h^{-i}$, with $r_{1}=g^{(k)}$ and $r_{k+1}=(-1)^{k}k!g(h')^{k}$. 
In particular, for every $k\geq 0$, there exists $s_{k}\in\DD$ such that 
\begin{equation}\label{E:xik}
\left( gh^{-1}\right)^{(k)}h^{k}=s_{k}+(-1)^{k}k!g(h')^{k}h^{-1}.
\end{equation}
\end{lemma}
\begin{proof}
This follows from the identity $(gh^{-1})^{(k)} = \sum_{j=0}^k {k \choose j} g^{(k-j)}(h^{-1})^{(j)}.$
\end{proof}

\begin{prop}\label{P:Dgadrai}
Assume $g,r \in \DD$.  Then for  $a_n=  \pi_h h^{n-1} y^n$ the following hold: 
\begin{itemize}
\item[{\rm (a)}]   If $n\geq 2$,  there exists $s\in\A_{h}$ so that $\widetilde D_g (r a_n)= s+ngra_{n-1}\in\mathsf{N}_{\A_1}(\A_h)$.   Thus, $[D_g, \ad_{r a_n}] = 
\ad_{\widetilde D_g(r a_n)} 
\in \{ \ad_b \mid b \in \norm\}$ and 
$$[D_g, \ad_{r a_n}]  =  n\ad_{gr a_{n-1}} \ \ \modu \, \inder(\A_h).$$
\item[{\rm (b)}] $[D_g, \ad_{r a_1}]=\ad_{gra_0}=-D_{\delta_0(gr)}$ where 
$\delta_0(gr)= \delta(gra_0) =  \big(gr\pi_hh^{-1}\big)' h.$ 
\item[{\rm (c)}]  $[D_g, \ad_r]=0$.
\end{itemize}
\end{prop}

\begin{proof}
For every $k\geq 0$, let $s_k \in\DD$ be given by~(\ref{E:xik}).  Assume $k, n\geq 2$. Then
\begin{align}
(gh^{-1})^{(n-1)}h^{n-1}r \pi_h&= s_{n-1}r\pi_{h}+(-1)^{n-1}(n-1)! g(h')^{n-1}h^{-1} r\pi_{h},\label{eqn:quotientDeriv1} \\
(gh^{-1})^{(k-1)} r h^k&=s_{k-1}rh+(-1)^{k-1}(k-1)!g(h')^{k-1}r \label{eqn:quotientDeriv2}.
\end{align}
The expression in \eqref{eqn:quotientDeriv1} is in  $\DD$  since $h$ divides $\pi_h h'$.  Now
if (\ref{eqn:quotientDeriv2}) is multiplied by $a_{n-k}$ (where $2 \le k \le n-1$), the right side is 
$$s_{k-1}rh a_{n-k} +  (-1)^{k-1}(k-1)! g (h')^{k-1}r \pi_h h^{n-k-1}y^{n-k},$$
 which is in $\A_h$ by (b) of Theorem \ref{T:norm}.  Hence, by Corollary \ref{cor:D_g(a_n)},  we have (a).  
Part\,(b) follows from Corollary~\ref{cor:D_g(a_n)} and Lemma \ref{L:del0}\,(a).  Part\,(c) is clear.
\end{proof}   
\smallskip

\subsection{The product $[\ad_{ra_m}, \ad_{sa_n}]$ for $r,s \in \DD$}  \hfil
\smallskip

Here we focus on the commutators
$[\ad_{ra_m}, \ad_{sa_n}]$.   As before,   $f^{(k)}$ denotes $\left(\frac{d}{dx} \right)^k(f)$ for any $f \in \DD$.  Our starting point is a fact about the terms $(r \pi_h h^\ell)^{(k)}$ for $r \in \DD$. 
\smallskip 

\begin{lemma}\label{lem:deriv(h^k pi)}
Fix $\ell \ge 0$ and let $r \in \DD$.  If $k \ge 2$, then 
\begin{equation}\label{eq:derhpi} (r \pi_h h^\ell)^{(k)} \in \DD h^{\ell+2-k} + \DD h^{\ell+1-k} h'.
\end{equation}  
\end{lemma}

\noindent \emph{Proof.} \ Consider first  the case $k = 2$.    Then 
\begin{equation}\label{eq:sum}(r \pi_h h^\ell)^{(2)} = (r\pi_h)'' h^\ell + 2 \ell (r \pi_h)' h^{\ell-1} h' 
+ \ell(\ell-1)r\pi_h h^{\ell-2}( h')^2 + \ell r\pi_h h^{\ell-1} h''.\end{equation}
Since $h$ divides $\pi_h h'$, it follows that $\ell(\ell-1)r\pi_h h^{\ell-2}( h')^2  \in \DD h^{\ell-1}h'$.
 We may suppose $\pi_h h' = d h$ for  $d \in \DD$ and then take the derivative of both sides to get
$\pi_h h'' = d'h+d h'  - \pi_h' h'$.   From that we deduce $\ell r\pi_h h^{\ell-1} h''$ belongs
to $\DD h^{\ell} + \DD h^{\ell-1} h'$, which is the right-hand side of \eqref{eq:derhpi} when $k=2$.
The first two summands of  \eqref{eq:sum} also clearly belong to the right-hand side of \eqref{eq:derhpi}, so 
the result  holds when $k = 2$.  

The inductive step follows from the fact that for $r, s \in \DD$ 
\begin{gather*} ( r h^{\ell+2-k})'  \in \DD h^{\ell+2-(k+1)} \quad \hbox{\rm and} \\  
\hspace{2cm} (s h^{\ell+1-k} h')'  \in \DD h^{\ell+2-(k+1)}+\DD h^{\ell+1-(k+1)} h'. \hspace{2.9cm} \square \end{gather*}  

The proof of the next lemma will use the fact that $[ \DD, \DD] = 0$
and the relation $[y^m,  f]  = \sum_{k=1}^m  {m \choose k}  f^{(k)} y^{m-k}$ in $\A_1$ from Lemma \ref{lem:identity_y^nf}.  \smallskip 

\begin{lemma}\label{L:bracket-ra_m_sa_n}
Let $r, s \in \DD$, and let $m, n \ge 1$.  In the Lie algebra $\hoch (\A_h)$, 
$$[ \ad_{ra_m}, \ad_{sa_n} ] =  \ad_{[ra_m, sa_n]} = \ad_{qa_{m+n-1}}, \ \ \  \hbox{\rm where} 
\ \  q= mr \delta_0(s) -ns\delta_0(r).$$
\end{lemma}  

\begin{proof} We first compute $[ra_m, sa_n]$ in $\norm$ and then argue that certain elements are 0 in the factor Lie algebra $\norm / \A_h$.   For all $r,s \in \DD$,  
\begin{align*}
[ra_m, sa_n] &= r\pi_h h^{m-1} [y^m, s\pi_h h^{n-1}] y^n - s\pi_h h^{n-1} [ y^n, r\pi_h h^{m-1}] y^m \\
&= r\pi_h h^{m-1} \sum_{k=1}^m {m \choose k} (s\pi_hh^{n-1})^{(k)} y^{m+n-k}\\
& \qquad  - s\pi_h h^{n-1} \sum_{k=1}^n {n \choose k} (r\pi_h h^{m-1})^{(k)} y^{m+n-k}.
\end{align*}
For $k \geq 2$,  Lemma \ref{lem:deriv(h^k pi)} implies that  ${m \choose k} (s \pi_h  h^{n-1})^{(k)} = u h^{n-1+2-k} + v h^{n-1+1-k} h'$ for some $u, v \in \DD$ (which depend on  $k$ and $m$).  Observe that 
\begin{align*} r\pi_hh^{m-1} uh^{n+1-k} y^{m+n-k} &= ru \pi_h h^{m+n-k} y^{m+n-k} \in \A_h, \quad \hbox{\rm and also} \\
r \pi_h h^{m-1}vh^{n-k}h' y^{m+n-k}\, &= rv \pi_h h' h^{m+n-1-k} y^{m+n-k} \in \A_h \end{align*}
because  $\pi_h h'$ is divisible by $h$.
Similar reasoning applies to the terms in the second summation.  It follows that the terms coming from the above sums can be nonzero in $\norm / \A_h$ only when $k=1$.  Thus, modulo $\A_h$, 
\begin{eqnarray*} [ra_m, sa_n] &=& m r \pi_h h^{m-1}(s\pi_h h^{n-1})' y^{m+n-1} - n s \pi_h h^{n-1}(r \pi_h h^{m-1})' y^{m+n-1}
\\   &=&  \left( mrh^{m-1} (s\pi_h h^{-1} h^n )'  - n sh^{n-1} (r\pi_h h^{-1}h^m)' \right) \pi_h y^{m+n-1} \\
&=&\left(mr\delta_0(s)h^{m+n-2}  - ns \delta_0(r) h^{m+n-2} \right) \pi_h y^{m+n-1} \\
 &=&\left(mr \delta_0(s)-ns \delta_0(r) \right) a_{m+n-1},
\end{eqnarray*} 
where  $\delta_0 : \DD \to \DD$ is as in \eqref{eq:del0}.  Hence, in $\hoch(\A_h)$ we have  $[\ad_{ra_m}, \ad_{sa_n}] = \ad_{[ra_m, sa_n]} = \ad_{qa_{m+n-1}}$, where  $q = mr \delta_0(s)-ns \delta_0(r)$, as desired.  \end{proof}
 
\subsection{Proof of Theorem \ref{L:brackets}}\label{sec:proofSec4MainThm} \hfil

\noindent Take $g\in\DD$. By Proposition~\ref{P:Dgadrai}, we have the following products in $\hoch(\A_1)$: \  
 $[D_{g}, \ad_{ra_{n}}]=\ad_{\tilde D_{g}(ra_{n})}=n\ad_{gra_{n-1}}$ if $n \geq 2$,  and $[D_{g}, \ad_{ra_{1}}]= -D_{\delta_0(gr)} =\ad_{gra_{0}}$. By  Lemma \ref{L:del0}\,(a) and Theorem \ref{thm:derdecomp}, \, $[D_{g}, \ad_{ra_{0}}] \,= -[D_g,D_{\delta_0(r)}] = 0$,  which shows that (a) holds for $n = 0$ as well.  
 Since $\frac{h}{\pi_h}a_n \in \A_h$ for all $n$, the rest of part (a) follows
 from applying the division algorithm.     
 
For $m, n\geq 1$,  part (b) is a consequence of  Lemma \ref{L:bracket-ra_m_sa_n}.  Given the skew-symmetry of the formula in (b), it suffices to consider the case $m=0$.  By Lemma \ref{L:del0}\,(a) and Proposition \ref{P:Dgadrai}, we have in $\hoch (\A_h)$,  
\begin{equation*}
[ \ad_{ra_0}, \ad_{sa_{n}} ] = -[ D_{\delta_0(r)}, \ad_{sa_{n}} ]  =  -n\, \ad_{\delta_0(r)sa_{n-1}} = -\ad_{ns\delta_0(r) a_{n-1} },\end{equation*}
which implies (b).   \hspace{9.3cm} $\square$  

 \subsection{Properties of $\delta_0$} \label{subsec:delta0}  \hfil  \smallskip

We conclude this section with a few results  on  the map
$\delta_0$ that will be used in the next two sections. Their statements require the
element $\varrho_h$ in \eqref{eq:vrhodef}.

\begin{lemma}\label{lem:divisor-del0} 
Assume $\FF$ is arbitrary, and let $\delta_0: \DD \rightarrow \DD$, $\delta_0(r) = \delta(ra_0)$, be as in (\ref{eq:del0}).   For all $r \in \DD$, $\frac{h}{\pi_h \varrho_h}$ divides $\delta_0(r)$ if and only if $\frac{h}{\pi_h \varrho_h}$ divides $r$.
\end{lemma}
\begin{proof}
Let $\hat h = \frac{h}{\varrho_h}$. Then $\pi_{\hat h}=\pi_{h}$ and $\varrho_{\hat h}=1$.  Let $\hat \delta(r) = r'\hat h$,  and let
$\hat a_0 = \pi_{\hat h} \hat h^{-1} =  \varrho_h a_0$.     Then  $\frac{h}{\pi_h \varrho_h}=\frac{\hat h}{\pi_{\hat h}}$ and
\begin{equation*}
\hat \delta (r\hat a_0) =  \left (r\hat a_0\right)' {\hat h}= \left (ra_0\right)'  \varrho_h{\hat h}=
\left (ra_0\right)' h=\delta(ra_0).
\end{equation*}
Thus, it is no loss of generality to assume that $\varrho_h=1$.

For $r \in \DD$, $\delta \left( r \frac{h}{\pi_h} a_0 \right) = \delta(r) = r' h$ is divisible by $h$, and therefore by $\frac{h}{\pi_h}$, and this establishes one of the implications.  For the direct implication, let $u$ be a prime divisor of $h$, and write $h = u^\alpha v$, where $\alpha\geq 1$ and $\gcd (u,v) = 1$. Since $\varrho_h = 1$, we may  also assume that $\alpha < p$ when $\chara(\FF) = p > 0$.  It follows that $\pi_h = u \pi_{v}$. Write $r = u^k s$, where $k\geq 0$ and 
$\gcd (u,s) = 1$. We will 
show that if $u^{\alpha-1}$ divides $\delta(ra_0)$, then $u^{\alpha-1}$ divides $r$.  Since $u$ is an arbitrary prime divisor of $h$, it will  follow from this that $\frac{h}{\pi_h}$ divides $r$, provided it divides $\delta(ra_0)$. 

With this notation,  we have  
\begin{align*} \delta_0(r) & = \delta(ra_0) =  \left(r \pi_hh^{-1} \right)' h= \left( u^{k+1-\alpha}s\pi_{v}v^{-1} \right)' u^\alpha v \\
& = (k+1-\alpha) u^k u's\pi_{v} + u^{k+1}v\, \left(s\pi_{v}v^{-1} \right)'. \end{align*}

Assume $u^{\alpha-1}$ divides $\delta_0(r)$. It is enough to argue that $k\geq \alpha-1$. Supposing the contrary, we have $k< \alpha-1$, so $k+1\leq \alpha-1$, which implies that $u^{k+1}$  divides $\delta_0(r)$. 
Now  $v  \left(s\pi_{v}v^{-1} \right)' \in \DD$, so  $u$ divides $(k+1-\alpha) u's\pi_{v}$.  Note that $u'\neq 0$, because we are assuming $\varrho_h=1$. As $u', s$, and $v$ are coprime to $u$, this implies $k = \alpha-1$ when $\chara(\FF) = 0$, which is a contradiction.     When $\chara(\FF) = p > 0$,  then  $k \equiv \alpha-1 \modd p$, but since  $1 \le \alpha < p$,  we again have the contradiction  $k = \alpha-1$.  Thus,  indeed $k\geq \alpha-1$.
\end{proof}
 
\begin{lemma}\label{L:del0ker}
Assume $\FF$ is arbitrary.   Then the following hold.
\begin{itemize}
\item[{\rm (a)}]  $\ker \delta_0 = \left(\DD \cap \centh\right)\frac{h}{\pi_h \varrho_h}$. 
\item[{\rm (b)}]  
$\mathsf{dim}\left\{ \delta_0(r)  \ \big | \  r \in \DD, \ \degg r  <\degg\frac{h}{\pi_h\varrho_h} \right\}  = 
\degg \frac{h}{\pi_h \varrho_h}$. 
\item[{\rm (c)}]  When $\chara(\FF) =0$,  then $\ker \delta_0 = \FF \frac{h}{\pi_h}$ and \\  
$\hspace{2.2cm} \mathsf{dim}\left\{ \delta_0(r)  \ \big | \  r \in \DD, \ \degg r  <\degg\frac{h}{\pi_h} \right\}= \degg  \frac{h}{\pi_h}$.
\item[{\rm (d)}] For $s \in \DD$, $\left( \frac{s}{h} \right)' = 0$ if and only if $s \in \left(\DD \cap \centh\right) \frac{h}{\varrho_h}$.
\end{itemize} 
\end{lemma}

\begin{proof}  (a)  Let $c \in \DD \cap \centh$ and note that 
$$\delta_0\left(c \frac{h}{\pi_h \varrho_h} \right)  =   \left( c \frac{h}{\pi_h \varrho_h} \pi_h h^{-1} \right)' h=  ( c\varrho_h^{-1})' h= 0.$$  
Therefore, $\left(\DD \cap \centh\right)\frac{h}{\pi_h \varrho_h} \subseteq \ker \delta_0$.

For the other containment, suppose that $\delta_0(r) = 0$.  Then Lemma \ref{lem:divisor-del0} implies that we may write $r = \tilde r \frac{h}{\pi_h \varrho_h}$ for $\tilde r \in \DD$.  Then applying Lemma \ref{L:del0}\,(b) we have
$$0  = \delta_0 \left( \tilde r \frac{h}{\pi_h \varrho_h} \right) = \tilde r \delta_0 \left( \frac{h}{\pi_h \varrho_h} \right) + \tilde r ' \frac{h}{\pi_h \varrho_h} \pi_h =   \tilde r ' \frac{h}{\pi_h \varrho_h} \pi_h,$$
which forces $\tilde r' = 0$, and thus $r = \tilde r \frac{h}{\pi_h \varrho_h} \in \left(\DD \cap \centh\right) \frac{h}{\pi_h \varrho_h}$.

For (b),  every  $r \in \ker \delta_0= (\DD \cap \centh)\frac{h}{\pi_h\varrho_h}$ is divisible by $\frac{h}{\pi_h\varrho_h}$,  so $r$ must be  0 or have  degree greater than or equal to the degree of  $\frac{h}{\pi_h\varrho_h}$.  Thus, the linear map 
\begin{equation}\label{eq:map4varth}
\left\{ r\in \DD\ \big | \ \degg r <\textstyle{\degg\frac{h}{\pi_h\varrho_h}} \right\}\ \longrightarrow \ \left\{\delta_0(r)\ \big | \ \degg r <\degg\textstyle{\frac{h}{\pi_h\varrho_h}} \right\} 
\end{equation}  
is an isomorphism.  Part (c) is immediate from (b)  and the fact that $\centh = \FF1$
 and $\varrho_h = 1$  when $\chara(\FF) = 0$.

For (d), it is clear that $\left( \frac{s}{h} \right)' = 0$ if $s \in \left(\DD \cap \centh\right) \frac{h}{\varrho_h}$.  For the other direction, suppose that $\left( \frac{s}{h} \right)' = 0$. Then $s'h=sh'$, so $h$ divides $sh'$ and it follows that $\pi_{h}$ divides $s$. Moreover, 
$$
\delta_0 \left( \frac{s}{\pi_h} \right) =  h \left (\frac{s}{h}\right)' =0,
$$
and this implies that $\frac{s}{\pi_h}\in\ker \delta_0= \left(\DD \cap \centh\right)\frac{h}{\pi_h \varrho_h}$, thus establishing the claim that $s\in \left(\DD \cap \centh\right)\frac{h}{\varrho_h}$.
\end{proof}  
 
\begin{section}{$\der(\A_h)$  when $\chara(\FF) = 0$}\label{sec:char0Der}  \end{section}   
The one-variable \emph{Witt algebra}  (also
known as the centerless Virasoro algebra) is the derivation algebra  $\mathsf{W} =  \der(\FF[t]) = \spann_\FF\{w_n  = t^{n+1}\frac{d}{dt}  \mid n \geq -1\}$,  where  $[w_m, w_n] = (n-m)w_{m+n}$ for $m,n \geq -1$,   ($w_{-2} = 0$).   When $\FF$ is the complex field, $\mathsf{W}$ is the Lie algebra of vector fields
on the unit circle, so  it has played an important role in many areas of mathematics and physics.   Our aim in this section is to show the following result  about $\hoch(\A_h) = \der(\A_h)/\inder(\A_h)$ for fields of characteristic 0, which we prove in Section \ref{sec:proofSec5MainThm}.

\medskip

\begin{thm}\label{thm:idealJ}  
Let  $\chara(\FF) = 0$,  and assume $h \neq 0$ and  $a_n = \pi_h h^{n-1}$ for all $n \geq 0$.   Then  
$\hoch(\A_h) = \mathsf{Z}(\hoch(\A_h))  \oplus [\hoch(\A_h),\hoch(\A_h)]$;
\begin{equation}\label{eq:Jideal} 
\mathcal N=\spann_\FF\{\ad_{r a_n} \mid  r \in \DD \pie, \ n \geq 0\}
\end{equation}
is the unique maximal  nilpotent ideal  of 
$ [\hoch(\A_{h}), \hoch(\A_{h})]$;    and 
\begin{equation*} \hoch(\A_h)/\mathcal N =  \mathsf{Z}(\hoch(\A_h)) \oplus [\hoch(\A_{h}), \hoch(\A_{h})]/\mathcal N, \ \ \  \hbox{\rm where}  \end{equation*}
\begin{itemize}
\item[{\rm (i)}] $\mathsf{Z}(\hoch(\A_h)) \cong \big\{ D_{r\frac{h}{\pi_{h}}}\, \big | \, \degg r <\degg \pi_{h} \big\}.$
\item[{\rm (ii)}] $ [\hoch(\A_{h}), \hoch(\A_{h})]/\mathcal N \cong  \big( (\DD/\DD \pie) \otimes \mathsf{W} \big)$,  and  $\mathsf{W} = \mathsf{span}_\FF\{w_i \mid i \geq -1\}$ is  the Witt algebra. 
\item[{\rm (iii)}]  $\big(\DD/\DD \pie\big) \otimes \mathsf{W}\ \cong\ \big( (\DD/\DD \pr_{1}) \otimes \mathsf{W} \big)\oplus \cdots\oplus \big( (\DD/\DD \pr_{k}) \otimes \mathsf{W} \big),$   a direct sum of simple Lie algebras, where $\pr_1, \dots, \pr_k$  are the monic prime factors of $h$ with multiplicity $>1$,
and each summand is a field extension of $\mathsf{W}$. \end{itemize}    \end{thm} 

We start by describing the decomposition $\der(\A_h) = \mathcal D_\DD + \mathcal E$ in Theorem \ref{thm:derdecomp} more explicitly and prove Theorem \ref{thm:idealJ} 
in a series of results.  We conclude the section by interpreting Theorem  \ref{thm:idealJ}  in
some cases of special interest.  \smallskip

\begin{thm}\label{T:derdec0}   Assume $\chara(\FF) = 0$, and regard $\A_h \subseteq \A_1$.  
Then $\der(\A_h) = \mathcal D \oplus \mathcal E$ where $\mathcal D = \{D_g \mid g \in \DD, \ \degg g < \degg h\}$ and $\mathcal E = \{\ad_a \mid a \in \norm\}$.  \end{thm}

\begin{proof}  We know from Theorem \ref{thm:derdecomp} that 
$\der(\A_h) = \mathcal D_\DD + \mathcal E$,  where $\mathcal D_\DD = \{D_g \mid g \in \DD \}$
and $\mathcal E = \{F \in \der(\A_1) \mid F(\A_h) \subseteq \A_h\}$.   Since every derivation of
$\A_1$ is inner (see Proposition \ref{P:Weyl0}),  $\mathcal E =  \{\ad_a \mid a \in \norm\}$.
Assume $D_f  \in \mathcal D_\DD$ and write  
$f = qh + g$,  where $\degg g< \degg h$.  
When  $\chara(\FF) = 0$, there exists $r \in \DD$ so that $r' = -q$.  Then
$(D_f - \ad_r)(x) = 0$,  and $(D_f - \ad_r)(\hat y) = 
f + [\hat y, r]  = f + r'h = f-qh = g$.     Therefore $D_f - \ad_r = D_g$  
and  $\der(\A_h) = \mathcal D + \mathcal E$,  where  $\mathcal E = \{\ad_a \mid a \in \norm\}$ 
and $\mathcal D = \{D_g \mid g \in \DD, \, \degg g < \degg h\}$.

Suppose now that  $D  \in \mathcal D  \cap \mathcal E$.   Then  $D(\DD) = 0$ and $D(\hat y) = g$ for some $g \in \DD$ with $\degg g < \degg h$  since $D \in \mathcal D$.  But then  $D(y)h = D(\hat y) =  g  \in \DD \subset \A_h$.  This implies $D(y) \in \DD$, and since  
 $\degg g < \degg h$, it must be that  $g = 0$,  and hence  $D = 0$.     \end{proof} 

\begin{exam}\label{ex:h=pih}  When $\chara(\FF) = 0$ and there are no repeated prime factors in $h$,  we have
$\frac{h}{\pi_h} \in \FF^*$.   In this situation,    $\norm = \A_h$  (compare Remark \ref{R:norm}).    Then 
$\mathcal E = \inder(\A_h)$,  and 
  $\hoch(\A_{h})\cong \mathcal D =\{ D_g \mid g \in \DD,  \degg g < \degg h\}$ is an abelian Lie algebra of dimension $\degg h$. 
\end{exam} 

In light of this result, it is tempting to think that the subalgebra  $\mathcal E$ might be an ideal of $\der(\A_h)$.  However,  that is not true in general as the next example illustrates.  

\begin{exam} Let  $\chara(\FF) = 0$ and $h = x^m$ for $m \geq 2$.  Then $\pi_h = x$,  and according to Proposition \ref{P:Dgadrai}\,(b),  
$[D_1, \ad_{a_1}]= \ad_{a_0} = -D_{\delta(a_0)},$ 
where  $\delta(a_0)= \left(\pi_h h^{-1}\right)' h= 1-m$.  Thus,  $[D_1, \ad_{a_1} ]   = (m-1) D_1  \notin \mathcal E$.
\end{exam} 
 
\begin{lemma}\label{lem:indepofads}  
Let  $\chara (\FF) = 0$  and $h \neq 0$ be arbitrary. Assume $g \in \DD$ with $\degg g < \degg h$,  and $r_n \in \DD$  with  $\degg r_n < \degg \frac{h}{\pi_h}$  for all $n \geq 0$. 
\begin{itemize}
\item[{\rm (i)}]  If  $D_g +\sum_{n\geq 1}\ad_{r_{n}a_{n}} \in \inder (\A_{h})$, then   $g=0 =r_{n}$ for all $n\geq 1$.
\item[{\rm(ii)}]  If $\sum_{n\geq 0}\ad_{r_{n}a_{n}} \in \inder (\A_{h})$, then  $r_n  = 0$ for all $n \geq 0$.
\end{itemize}
\end{lemma}

\begin{proof}   (i)  Write $D_g+\sum_{n\geq 1}\ad_{r_{n}a_{n}}=\ad_{a}$  for some $a\in\A_{h}$. Then
\begin{equation*}
D_g=\ad_{a}-\sum_{n\geq 1}\ad_{r_{n}a_{n}}\in\mathcal D\cap\mathcal E=0,
\end{equation*}
by Theorem \ref{T:derdec0}. It follows that $g=0$ and $\ad_{b}=0$, where $b=a-\sum_{n\geq 1}r_{n}a_{n}$. Thus, $b\in\A_{1}$ centralizes $\A_{h}$.  By Lemma \ref{L:clizer}, $b \in \DD  \subset \A_h$, so in fact $b \in \FF$,  as it commutes with $\hat y$.   In particular, we have $\sum_{n\geq 1}r_{n}a_{n}\in\A_{h}$. Since $a_n =  \pi_h h^{n-1} y^n$, we conclude from part (c) of Lemma \ref{L:embed} that $h$ divides $r_{n}\pi_{h}$ for all $n\geq 1$; that is,  $r_n \in \DD \frac{h}{\pi_h}$
for all $n \geq 1$.  But  since $\degg r_n < \degg\frac{h}{\pi_h}$, it must be that $r_n = 0$ for all $n \geq 1$.  
 
(ii)  Assume $\sum_{n \geq 0}\ad_{r_{n}a_{n}} \in \inder(\A_h)$.    By Proposition \ref{L:del0}\,(a),  $\ad_{r_0 a_0} =  -D_{\delta_0(r_0)}$.   As $\degg r_0 <  \degg\frac{h}{\pi_h}$,  we have that  $\degg \delta_0(r_0) < \degg h$.  
Therefore,  by (i)  we know that $r_n=0$ for all $n \geq 1$,  and $\delta_0(r_0) = 0$.    This implies
$r_0 \in \ker \delta_0 = (\DD \cap \centh) \frac{h}{\pi_h} = \FF \frac{h}{\pi_h}$ by Lemma \ref{L:del0ker}.   But then
 $\degg r_0 < \degg \frac{h}{\pi_h}$ forces  $r_0 = 0$ to hold.    \end{proof} 
\subsection{The structure of $\mathcal E$}  \hfil
\smallskip

Recall from Theorem \ref{T:derdec0} that $\mathcal E = \{\ad_a \mid a \in \norm\}$ when 
$\chara(\FF) = 0$.   
The next theorem,  a key result in our paper,  clarifies the relationship between $\mathcal E$ and $\inder(\A_h)$
and provides more detailed information about $\der(\A_h)$ and $\hoch(\A_h) = \der(\A_h)/\inder(\A_h)$.  
\medskip

\begin{thm}\label{T:hochstruct}   Assume $\chara (\FF)=0$. Then as vector spaces over $\FF$,  
\begin{itemize}
\item[{\rm (i)}] $ \mathcal E= \spann_\FF\{\ad_{ra_n}  \mid  r \in \DD, \degg r < \deghpi, \ n \geq 1\}\oplus \inder(\A_h).$
\item[{\rm(ii)}] $\der(\A_h) = \mathcal D \oplus  \spann_\FF\{\ad_{r a_n} \mid r \in \DD,  \degg r < \deghpi, \ n \geq 1\}\oplus \inder(\A_h),$ 
where $\mathcal D = \{D_g \mid  g \in \DD, \  \degg g< \degg h\}.$   
 \item[{\rm(iii)}] ${\hoch(\A_{h}) \  \cong \  \mathcal D \oplus \spann_\FF\{\ad_{r a_n} \mid r \in \DD,  \degg r <\deghpi,  \ n \geq 1\}.}$ 
  \end{itemize}
\end{thm}   

\begin{remark} In the statement of Theorem  \ref{T:hochstruct}\,(iii)  and in what follows, we identify the derivations $D_g$  $(\degg g< \degg h)$ and the derivations $\ad_{r a_n}$ $(\degg r < \deghpi,  \ n\geq 1)$  with  their image in $\hoch(\A_h)= \der(\A_h)/\inder(\A_h)$ and use
the same notation for both.  \end{remark} 

 \noindent \emph{Proof of Theorem \ref{T:hochstruct}}. \  Clearly   $\inder(\A_{h})\subseteq \mathcal E= \{\ad_a \mid a \in \norm\}$.   Moreover, the sum $\spann_\FF\{\ad_{r a_n} \mid r \in \DD,  \degg r < \deghpi, \ n \geq 1\}+\inder(\A_{h})$ is direct by Lemma~\ref{lem:indepofads}\,(ii).

To show $\mathcal E$ equals this direct sum,  assume  $b \in\mathsf{N}_{\A_1}(\A_h)$.  By 
Theorem \ref{T:norm}(a)(i),  we may suppose  $b=r_{0}+\sum_{n \geq 1}r_{n}a_{n}$, where  $r_{n}\in\DD$ for all $n$.  For $n\geq 1$, write $r_{n}=q_{n}\frac{h}{\pi_{h}}+\tilde r_{n}$, with $q_{n}, \tilde r_{n}\in\DD$ and $\degg \tilde r_{n}<\deghpi$. Then,
\begin{equation*}
b=r_{0}+\sum_{n\geq 1}q_{n}{\textstyle \frac{h}{\pi_{h}}}a_{n}+\sum_{n\geq 1}\tilde r_{n}a_{n} .
\end{equation*}
Since $\frac{h}{\pi_h} a_n = h^n y^n \in \A_h$ for all $n \geq 1$,  we have  $a=r_{0}+\sum_{n\geq 1}q_{n}\frac{h}{\pi_{h}}a_{n}\in\A_{h}$. Thus, 
$\ad_{b}=\sum_{n \geq 1}\ad_{\tilde r_{n}a_{n}} + \ad_a$  is an element of $\spann_{\FF}\{\ad_{ra_n}  \mid  r \in \DD, \degg r < \textstyle{\deghpi}, \ n \geq 1\}\oplus \inder(\A_h).$  Combining that  with  Theorem \ref{T:derdec0} gives (ii),  and hence (iii).   \qed 

\subsection{The commutator ideal  $[\hoch (\A_h), \hoch (\A_h)]$} \hfil
\smallskip

\begin{prop}\label{prop:derived}  Assume $\chara (\FF)=0$.  Then
\begin{equation}\label{eq:hochcomm} [\hoch (\A_h), \hoch (\A_h)]  =  \spann_\FF\{\ad_{ra_n}\,|\, r \in\DD, \, \degg r <\degg\textstyle{\frac{h}{\pi_h}}, \,  n \geq 0\}. \end{equation}
Moreover, $\hoch (\A_h)/[\hoch (\A_h), \hoch (\A_h)]$ is an abelian Lie algebra of dimension $\degg \pi_{h}$. 
\end{prop}

\noindent \emph{Proof.} \ 
Assume  $r\in\DD$, $\degg r < \degg \frac{h}{\pi_h}$,  and $n\geq 0$. Then by Lemma \ref{L:brackets}\,(a), $$\ad_{ra_{n}}=\textstyle{\frac{1}{n+1}}[D_{1}, \ad_{ra_{n+1}}]$$ in $\hoch(\A_h)$, which proves the right side of 
\eqref{eq:hochcomm} is contained in the left.   The reverse containment follows from Theorem \ref{T:hochstruct}\,(iii),
Lemma \ref{L:brackets}, and the fact that $\mathcal D$ is abelian (Theorem \ref{thm:derdecomp}).

Consider the linear map 
\begin{equation}\label{eq:rho} \rho \, :\, \{ g\in\DD\mid \degg g<\degg h\}\longrightarrow \hoch (\A_h)/[\hoch (\A_h), \hoch (\A_h)],\end{equation} with $\rho (g)=D_{g}+[\hoch (\A_h), \hoch (\A_h)]$.   By  Theorem \ref{T:hochstruct}\,(iii)
and  \eqref{eq:hochcomm}, $\rho$ is surjective. 

Now suppose $g\in\DD$ with $\degg g<\degg h$, and $\rho(g)=0$. Then there exist $r_{n}\in\DD$ with $\degg r_n<\deghpi$, so that $D_{g}=\sum_{n\geq 0}\ad_{r_{n}a_{n}}=\ad_{r_{0}a_{0}}+\sum_{n\geq 1}\ad_{r_n a_n}$. 
Hence, by Lemma~ \ref{L:del0}\,(a),
$D_{g+\delta_0(r_0)}-\sum_{n\geq 1}\ad_{r_na_n}=0$. Thus, $g= -\delta_0(r_0)$ by Lemma~\ref{lem:indepofads}\,(i).  Conversely, if $g=-\delta_0(r_0)$ for some $r_{0}\in\DD$ with $\degg r_{0}<\deghpi$, then $\rho(g)=0$. Therefore, 
\begin{equation}\label{eq:kerrho}
\ker \rho =\left\{\delta_0(q)  \mid  \degg q<\textstyle{\mathsf{deg}\frac{h}{\pi_{h}}}\right\},
\end{equation}
and $\dimm \ker \rho=\deghpi$, by Lemma \ref{L:del0ker}\,(c). Consequently, $$\hspace{.9cm} \mathsf{dim}\left(\hoch (\A_h)/[\hoch (\A_h), \hoch (\A_h)]\right)=\degg h-\textstyle{\mathsf{deg}\frac{h}{\pi_{h}}}=\degg\pi_{h}. \hspace{.9cm} \square$$   
 
\subsection{The center of  $\hoch(\A_h)$} \hfil 
 
\begin{thm}\label{thm:hochdec}
Assume $\chara (\FF)=0$.  Then 
\begin{equation}\label{eq:hochdecomp} 
\hoch(\A_{h})=\mathsf{Z}(\hoch (\A_h))\oplus [\hoch (\A_h), \hoch (\A_h)], \ \ \hbox{\rm where}  \end{equation}
\begin{equation}\label{eq:hochcent} \mathsf{Z}(\hoch (\A_h))=\bigg\{ D_{r\frac{h}{\pi_h}}\,\bigg | \,\degg r <\degg \pi_{h} \bigg\}  \, \hbox{\rm  and} \,
\dimm \mathsf{Z}\left(\hoch (\A_h)\right) = \degg \pi_{h}.\end{equation}
\end{thm}
\begin{proof} 
Let $z\in \mathsf{Z}(\hoch (\A_h))$.  By Theorem \ref{T:hochstruct}\,(iii), we may write  $z=D_g+\sum_{n=1}^{\ell}\ad_{r_{n}a_{n}}$, with $g, r_{n}\in \DD$, $\degg g <\degg h$ and $\degg r_{n} < \deghpi$ for all $n$. Then by Lemma~\ref{L:brackets}\,(a),\  
 $0=[D_{1}, z]=\sum_{n=1}^{\ell}n\,\ad_{r_{n}a_{n-1}}.$
By Lemma \ref{lem:indepofads}\,(ii), \  $r_{n}=0$ for all $1\leq n\leq \ell$ and $z=D_g$.  But then\
$0=[D_g, \ad_{a_{1}}]=\ad_{ga_{0}},$ \ so $\frac{h}{\pi_h}$ divides $g$. This proves one direction of the  inclusion in 
\eqref{eq:hochcent}.

Conversely, for all $g, r, s\in\DD$ and $n\geq 1$,  we have in $\hoch(\A_h)$,
\begin{center}
$ \left[D_{r\frac{h}{\pi_h}}, \ad_{sa_{n}} \right]=n\ad_{\frac{h}{\pi_h}rsa_{n-1}}=0= 
\left[ D_{r\frac{h}{\pi_h}}, D_g \right],$
\end{center}
showing that $D_{r\frac{h}{\pi_h}} \in \mathsf{Z}(\hoch (\A_h))$ and implying that 
\eqref{eq:hochcent}  holds.
 
To  verify  the sum in \eqref{eq:hochdecomp} is direct, suppose 
$$z\in \mathsf{Z}(\hoch (\A_h))\cap [\hoch (\A_h), \hoch (\A_h)].$$  By \eqref{eq:hochcent}, there  is a  $g\in \DD \frac{h}{\pi_{h}}$ with $\degg g <\degg h$ such that $z=D_{g}$.  But then $g \in \ker \rho$, where $\rho$  is as in \eqref{eq:rho},  and hence $g = \delta_0(q)$ for some $q$ with  $\degg q< \mathsf{deg}\frac{h}{\pi_{h}}$ by \eqref{eq:kerrho}.    Hence, $\frac{h}{\pi_{h}}$ divides $\delta_0(q)$.   But when $\chara(\FF) = 0$, Lemma \ref{lem:divisor-del0} implies that $\frac{h}{\pi_h}$ divides $q$.  Since $\degg q < \degg \frac{h}{\pi_h}$, it follows that $q = 0$, so that $z = 0$.

We know now that  the map $$\iota : \mathsf{Z}(\hoch (\A_h))\rightarrow \hoch (\A_h)/[\hoch (\A_h), \hoch (\A_h)],$$ given by restriction of the canonical epimorphism  is injective.  By  Proposition~\ref{prop:derived} and \eqref{eq:hochcent}, both algebras have dimension $\degg \pi_{h}$, so $\iota$ is in fact an isomorphism. In particular, 
\begin{equation*}
\hoch(\A_{h})=\mathsf{Z}(\hoch (\A_h)) + [\hoch (\A_h), \hoch (\A_h)],
\end{equation*}
which finishes the proof.
\end{proof} 
  
\subsection{The structure of $[\hoch(\A_{h}), \hoch(\A_{h})]$} \label{S:hochstr}\hfil
\smallskip   

Let $\chara(\FF) = 0$,  and assume as before 
$h= \lambda \pr_1^{\alpha_1} \cdots \pr_t^{\alpha_t},  \ \  \pi_h = \pr_1 \cdots \pr_t,$
where the $\pr_i$ are the distinct monic prime factors of $h$ and $\lambda \in \FF^*$.     Let
\begin{equation}\vth = \delta_0(1) = \pi_h' - \frac{\pi_h h'}{h}  =  \sum_{i=1}^t  (1-\alpha_i)  \pr_1 \cdots \widehat \pr_i \cdots \pr_t \pr_i'.\end{equation}
Observe that $\frac{h}{\pi_h} =  \lambda\displaystyle \prod_{i, \ \alpha_i \geq 2} \pr_i^{\alpha_i-1}$,  so that  \  $\pie= \prod_{i,\ \alpha_{i}\geq 2}\pr_{i}$ \ is the product of the
distinct prime factors of $h$ having multiplicity  $>1$,  and $\mathsf{gcd}(\vth, \pie) =1$.  

Recall from Proposition \ref{prop:derived} that  $$[\hoch (\A_h), \hoch (\A_h)]=  \spann_\FF\{\ad_{ra_n}\mid r \in \DD, \degg r < \degg \textstyle{\frac{h}{\pi_h}},   \ n \geq 0\},$$  where   $a_n =  \pi_h h^{n-1}y^n$ for all $n \geq 0$, and $a_n\in \norm$ for all $n\geq 1$.  
For $m,n\geq 0$ and $r, s\in\DD$, by Lemma \ref{L:brackets}\,(b)  we have $[\ad_{ra_{m}}, \ad_{sa_{n}}]=\ad_{qa_{m+n-1}}
 = \ad_{da_{m+n-1}}$ in $\hoch(\A_h)$, where   
$q =  mr\delta_0(s)-ns\delta_0(r)$ and $d$ is the remainder when $q$ is divided by $\frac{h}{\pi_h}$ in $\DD$.
 
Using  \eqref{eq:hochdecomp} and the fact that  $\delta_0(r) = r\delta_0(1) + r'\pi_h$  and $\pi_h$ is divisible by $\pie$, we have  that   
$$\mathcal N=\spann_\FF\{\ad_{r a_n}\mid r  \in \DD\pie, \ n \geq 0\}$$ 
is an ideal of $\hoch(\A_{h})$  contained in $[\hoch(\A_{h}), \hoch(\A_{h})].$    Our immediate goal is to demonstrate several important properties of the ideal $\mathcal N$ and to understand the Lie algebra  
$$\mathcal L = [\hoch(\A_{h}), \hoch(\A_{h})]/\mathcal N.$$

For $g \in \DD$ and $m \geq -1$,   set  \begin{equation}\label{eq:egmdef} e_{g,m}  = -\ad_{g a_{m+1}} + \mathcal N. \end{equation}   
Then for $r \in \DD$,  we have 
\begin{equation} \label{eq:coset} g  =  r \  \modd \DD\pie \ \   \Longrightarrow\ \  e_{g,m}  =  e_{r,m}. \end{equation}  

Theorem \ref{L:brackets}\,(b)  shows that the elements $\ad_{ra_m}$ have a multiplication very similar
to that of $\DD \otimes \mathsf{W}$, where $\mathsf{W}$ is the Witt algebra.   This motivates the next result. 

\begin{lemma}\label{lem:idealJ}  Assume $\chara(\FF) = 0$, and let $\mathsf{W} = \mathsf{span}_\FF\{w_n \mid n \geq -1\}$ be the Witt algebra  so that $[w_m,w_n] = (n-m)w_{m+n}$ for $m,n \ge -1$ \, ($w_{-2} = 0$).    Then   $\mathcal L = [\hoch(\A_{h}), \hoch(\A_{h})]/\mathcal N   \cong   \big(\DD/\DD\pie\big) \otimes \mathsf{W}$,  
and $\mathcal L$ is simple if $\pi_{(h/\pi_{h})}$ is a prime polynomial.   \end{lemma} 

\begin{proof}   In proving this lemma, we will  use $r$ to denote both an element of $\DD$ and the coset it determines in  $\DD/\DD \pie$, which is permissible to do by \eqref{eq:coset}.       

The elements $e_{x^{j},m}$, with  $0\leq j < \degg \pie$ and $m\geq -1$, generate $\mathcal L$ by \eqref{eq:coset}.  To show they form a basis of $\mathcal L$, suppose
$\sum_{j, m}\gamma_{j, m}e_{x^{j},m}=0$,  for scalars $\gamma_{j, m}$, $0\leq j < \degg \pie$ and $m\geq -1$. 
Let $r_{m}=\sum_{j}\gamma_{j, m}x^{j}$. Thus, $\sum_{m\geq -1}\ad_{r_{m}a_{m+1}}\in\mathcal N$, which by Lemma~\ref{lem:indepofads}\,(ii)  implies that $r_{m}\in \DD \pie$ for all $m\geq -1$,  since by construction, $\degg r_{m}<\degg \pie \leq \deghpi$. Hence, it must be that $r_{m}=0$ and $\gamma_{j, m}=0$, for all $0\leq j < \degg \pie$ and $m\geq -1$. 

Assume $\upsilon  \in \DD$ satisfies $\upsilon \vth = 1 \; \modd \DD \pie$, and 
consider the linear map
$$(\DD/\DD \pie) \otimes \mathsf{W}  \rightarrow \mathcal L,  \ \ \  r \otimes w_m \mapsto  e_{r\upsilon, m}.$$
Now 
$$[r \otimes w_m, s \otimes w_n] = (n-m)(rs \otimes w_{m+n}) \mapsto  (n-m) e_{rs\upsilon,  m+n}.$$
However,  in  $\mathcal L$  we have by Lemma \ref{L:brackets}\,(b) (as $\pie$ divides $\pi_{h}$)  that
\begin{eqnarray*}  [e_{r\upsilon,m},  e_{s\upsilon, n}] &=&   (m-n)\ad_{ rs\upsilon^2 \vth  a_{m+n+1} } + \mathcal N \\
& = & 
(m-n)\ad_{ rs\upsilon a_{m+n+1} } + \mathcal N   =  (n-m) e_{rs\upsilon, m+n}.  \end{eqnarray*} 
Thus, this map is a Lie homomorphism with inverse map given by  $e_{r,m}   \mapsto   r\vth \otimes w_m$  
for $r \in \DD$,   $\degg r < \degg \pie$, so that $\mathcal L \cong   \big(\DD/\DD\pie\big) \otimes \mathsf{W}$.

Suppose now that $\pie$ is a prime polynomial.   We argue  that $\mathbb K \otimes \mathsf W$ is simple, where $\mathbb K$ denotes the field $\DD/\DD \pie$.  Let $\Omega$ denote a nonzero ideal of $\mathbb K \otimes \mathsf W$, and let $0 \neq  \omega = \sum_{n=-1}^\ell  \xi_n \otimes w_n \in \Omega$,  where $\omega$ is chosen so that $\ell \geq -1$ is minimal. 
Then 
$$0 \neq [ 1 \otimes w_{-1},\omega] = \sum_{n=0}^\ell  [ 1 \otimes w_{-1},  \xi_n \otimes w_n ] =  \sum_{n=0}^\ell (n+1) \xi_n \otimes w_{n-1} \in \Omega.$$
This contradicts the minimality of $\ell$, unless $\ell= -1$.    Hence,  we may suppose $0 \neq \xi \otimes w_{-1} \in \Omega$ for some $0 \neq \xi \in \mathbb K$.  From this it follows that $\Omega$ contains 
$$[ \xi \otimes w_{-1},   \kappa \otimes w_{m+1}] = (m+2) \xi  \kappa \otimes w_{m}$$
for every $\kappa \in \mathbb K$ and $m \ge -1$, and consequently $\mathbb K \otimes \mathsf W \subseteq \Omega$.    \end{proof}

Assume there are  $k \geq 0$ distinct monic prime factors of $h$ with multiplicity  $>1$.  If $k=0$,  then $\frac{h}{\pi_h} \in \FF^*$ and $\pie=1$.  In this case,  $\DD/\DD \pie=0$  and  $\mathcal L = [\hoch(\A_{h}), \hoch(\A_{h})]/\mathcal N=0$. 
If  $k \geq 1$, then after possibly renumbering the factors,  we may suppose that $\pr_1, \ldots \pr_k$ are the distinct monic primes  occurring with multiplicity $>1$ in $h$.   In other words, $\pie= \pr_{1}\cdots \pr_{k}$.   Then 
\begin{equation}\label{eq:chinese}
\DD/\DD \pie=\DD/\DD \pr_{1}\cdots \pr_{k}\cong \DD/\DD \pr_{1}\oplus \cdots\oplus \DD/\DD \pr_{k},
\end{equation}
so it follows that 
\begin{equation}\label{eq:Lsemisimple}
\big(\DD/\DD \pie\big) \otimes \mathsf{W}\cong \left( (\DD/\DD \pr_{1}) \otimes \mathsf{W} \right)\oplus \cdots\oplus \left( (\DD/\DD \pr_{k}) \otimes \mathsf{W} \right).
\end{equation}
By Lemma~\ref{lem:idealJ}, each of the summands  $(\DD/\DD \pr_{i}) \otimes \mathsf{W}$ corresponds to a simple ideal 
of $\mathcal L$, so $\mathcal L$ is semisimple in this case.   

\begin{cor}\label{cor:structureL}
 Assume $\chara(\FF) = 0$ and $h= \lambda \pr_1^{\alpha_1} \cdots \pr_t^{\alpha_t}$,   where
$\lambda \in \FF^*$,   the  $\pr_i$ are the distinct monic prime factors of $h$,  and for $k \geq 0$, $\pr_1, \ldots, \pr_k$ are the ones which occur with multiplicity $>1$. (When $k = 0$, no factor has multiplicity $>1$.)   Let  $\mathcal N=   \spann_\FF\{ \ad_{ra_n} \mid  r\in \DD \pie, \ n \geq 0\}  \subseteq [\hoch(\A_{h}), \hoch(\A_{h})].$  Then the following hold:
\begin{itemize}
\item[{\rm (i)}]  $\mathcal N$ is the unique maximal nilpotent ideal of $[\hoch(\A_{h}), \hoch(\A_{h})]$ and the quotient   
$[\hoch(\A_{h}), \hoch(\A_{h})]/\mathcal N$ is the direct sum of $k$ simple Lie algebras  
\begin{equation}\label{eq:hochquo}\hspace{-.3cm}[\hoch(\A_{h}), \hoch(\A_{h})]/\mathcal N \  \cong \  \left( (\DD/\DD \pr_{1}) \otimes \mathsf{W} \right)\oplus \cdots\oplus \left( (\DD/\DD \pr_{k}) \otimes \mathsf{W} \right),\end{equation} where $\mathsf{W}$ is the Witt algebra. 
\item[{\rm (ii)}]  If $\alpha_{i}\leq 2$ for all $1\leq i\leq t$, then $\mathcal N=0$. 
\begin{enumerate}
\item[{\rm (a)}] If  $\alpha_{i}=1$ for all $i$, then $[\hoch(\A_{h}), \hoch(\A_{h})] = 0$.   
\item[{\rm (b)}] If some $\alpha_i = 2$,
then $[\hoch(\A_{h}), \hoch(\A_{h})] $ is the direct sum of simple Lie algebras (compare \eqref{eq:hochquo}). 
\end{enumerate}
\item[{\rm (iii)}] If there is $i$ such that $\alpha_{i}\geq 3$, then $\mathcal N\neq 0$,  and $[\hoch(\A_{h}), \hoch(\A_{h})]$ is neither nilpotent nor semisimple.
\end{itemize}
\end{cor} 
\begin{proof}
By Lemma~\ref{lem:idealJ} and the above, $\mathcal L = [\hoch(\A_{h}), \hoch(\A_{h})]/\mathcal N$ is a direct sum of $k\geq 0$ simple Lie algebras of the form $\big(\DD/\DD \pr_{i}\big) \otimes \mathsf{W}$, where $i\leq k$ and $\mathsf{W}$ is the Witt algebra.

To show that $\mathcal N$ is nilpotent,  let  $\mathcal N_{j} \subseteq \mathcal N$  for $j \geq 1$ be defined by
\begin{equation}\label{eq:Jj}
\mathcal N_{j}= \spann_\FF\{\ad_{ra_n} \mid  r \in \DD \big(\pie\big)^{j}, \ n \geq 0\}.
\end{equation} 
Then it is easy to  see, using Lemma~\ref{L:brackets} and the fact that $\pi_{(h/\pi_{h})}$ divides $\pi_{h}$, that $\mathcal N_{j}$ is an ideal of $[\hoch(\A_{h}), \hoch(\A_{h})]$ and $[\mathcal N, \mathcal N_{j}]\subseteq \mathcal N_{j+1}$.  As $\frac{h}{\pi_h}$ divides  $(\pi_{(h/ \pi_h)})^n$ for some $n$,   it follows that $\mathcal N_n=0$ and $\mathcal N$ is nilpotent.  

For any nilpotent ideal  $\mathcal J$  of $[\hoch(\A_{h}), \hoch(\A_{h})]$, \  $(\mathcal J+\mathcal N)/\mathcal N$ is a nilpotent ideal of $\mathcal L$. Since $\mathcal L$ is either $0$ or a direct sum of simple ideals, it has no nonzero nilpotent ideals. Hence, $\mathcal J\subseteq \mathcal N$, which proves the claim that  $\mathcal N$ is the unique maximal nilpotent ideal of $[\hoch(\A_{h}), \hoch(\A_{h})]$.

If all prime factors of $h$ have multiplicity at most $2$, then $\pie=\frac{h}{\lambda\pi_{h}}$ and $\mathcal N=0$. Thus, $[\hoch (\A_h), \hoch (\A_h)]=\mathcal L$ and part {\rm (ii)} follows. If there is a prime factor of $h$ with multiplicity greater than $2$,  then $\frac{h}{\pi_h}$ does not divide $\pie$, so $\mathcal N\neq0$.   In particular, $[\hoch(\A_{h}), \hoch(\A_{h})]$ is not semisimple, as it has a nonzero nilpotent ideal. However, if $[\hoch(\A_{h}), \hoch(\A_{h})]$ were nilpotent, then $\mathcal N=[\hoch(\A_{h}), \hoch(\A_{h})]$ and thus $\pie=1$, so $\frac{h}{\pi_h}\in\FF^{*}$, which contradicts our hypothesis. Therefore, $[\hoch(\A_{h}), \hoch(\A_{h})]$ is not nilpotent either.
\end{proof}

We now have all the pieces to assemble the proof of Theorem \ref {thm:idealJ}.  

\subsection{Proof of Theorem \ref {thm:idealJ}.}\label{sec:proofSec5MainThm} \hfil   

\noindent  By Theorem \ref{thm:hochdec}, 
$\hoch(\A_{h})=\mathsf{Z}(\hoch (\A_h))\oplus \allowbreak [\hoch (\A_h), \hoch (\A_h)]$ if  $\chara(\FF) = 0$,  where  $\mathsf{Z}(\hoch (\A_h)) = \big\{ D_{r\frac{h}{\pi_h}}\,\big | \,\degg r <\degg \pi_{h} \big\}$  and $\dimm \mathsf{Z}\left(\hoch (\A_h)\right) = \degg \pi_{h}$.  Then Corollary \ref{cor:structureL} tells us that    
$ \mathcal N= \spann_\FF\{\ad_{r a_n} \, | \,  r \in \DD \pie, \ n \geq 0\}$ \
is the unique maximal nilpotent ideal of \  $[\hoch(\A_{h}), \hoch(\A_{h})]$ \  and  \newline 
$[\hoch(\A_{h}), \hoch(\A_{h})]/\mathcal N \, \cong \left( (\DD/\DD \pr_{1}) \otimes \mathsf{W} \right)\oplus \cdots\oplus \left( (\DD/\DD \pr_{k}) \otimes \mathsf{W} \right)$, a direct sum of simple Lie algebras,  where $\mathsf{W}$ is the Witt algebra;  $\pr_1, \dots, \pr_k$  are the monic prime factors of $h$ with multiplicity $>1$;  
and each summand is a field extension of $\mathsf{W}$.   This establishes all the assertions in Theorem \ref {thm:idealJ}  and concludes the proof. \hspace{.8cm}  $ \square$

 \begin{cor}\label{cor:HH1nilpotent} Assume $\chara(\FF) = 0$.  Then
\begin{itemize}
\item[{\rm (a)}] $\mathsf{Z}(\hoch(\A_h)) \oplus\mathcal N$ is the unique maximal nilpotent ideal of $\hoch(\A_{h})$.
\item[{\rm (b)}] $\hoch(\A_{h})$ is a nilpotent Lie algebra if and only if $\frac{h}{\pi_h} \in \FF^*$.
\item[{\rm (c)}] {\rm [Example \ref{ex:h=pih} revisited]}
If  $\frac{h}{\pi_h} \in \FF^*$, then $\pie=1$, which implies
$[\hoch (\A_h), \hoch (\A_h)] = 0 = \mathcal N$  and \begin{equation*}
\hoch(\A_{h})\cong  \{D_g \mid \degg g < \degg \pi_h =  \degg h\},   \end{equation*}
an abelian Lie algebra of dimension $\degg h$.
\end{itemize}  \end{cor}

It is a consequence of Theorem \ref{thm:idealJ}  that $\hoch(\A_h)$   modulo its unique maximal nilpotent ideal $\mathsf{Z}(\hoch(\A_h)) \oplus\mathcal N$ is either 0 or  a direct sum of  ideals that are simple Lie algebras of the form $\DD_f \otimes \mathsf{W}$,  where $f \in \DD = \FF[x]$, $\DD_f  = \DD/\DD f$,   and $\mathsf{W}$ is the Witt algebra.   Proposition \ref{P:centroid}  below gives  a criterion for two such algebras $\DD_f$ and $\DD_g$  to be isomorphic.   

Recall that the  {\it centroid} of an $\FF$-algebra $\mathcal A$  is 
\begin{equation}\label{eq:centroid1} \mathsf{Ctd}_\FF(\mathcal A)  =  \{\chi \in \mathsf{End}_\FF(\mathcal A) \mid
 a \chi(b) =\chi(ab) = \chi(a)b  \ \hbox{\rm for all}  \ a,b \in \mathcal A\}. \end{equation} 
 If two algebras $\mathcal A_1$ and $\mathcal A_2$ are isomorphic via an
 isomorphism $\eta$,   then $\mathsf{Ctd}_\FF(\mathcal A_1)$ is isomorphic
 to $\mathsf{Ctd}_\FF(\mathcal A_2)$ via the isomorphism
 $\chi \mapsto  \eta \chi \eta^{-1}$.  
  
Now it follows from \cite[Cor. 2.23]{BN06} that if $\mathcal A$ and $\mathcal B$ are algebras over a field $\FF$,  $\mathcal B$ is perfect and finitely generated as a module over its algebra of multiplication operators,  and $\mathcal A$ is unital,  then 
\begin{equation}\label{eq:centroid2} \mathsf{Ctd}_\FF(\mathcal A \otimes \mathcal B) \cong \mathsf{Ctd}_\FF(\mathcal A) \otimes \mathsf{Ctd}_\FF(\mathcal B).\end{equation} (The roles of $\mathcal A$ and $\mathcal B$ are reversed here from what is in \cite{BN06}  to make this compatible with our expressions.)  We will apply this result to compute the centroid of  the Lie algebra $\DD_f \otimes \mathsf{W}$, which we can do since $\mathsf{W}$ is perfect and generated by $w_{-1}, w_2$,  and then use this to show 
     
\begin{prop}\label{P:centroid}   $\DD_f \otimes \mathsf W   \cong \DD_g \otimes \mathsf W$ 
 if and only if  $\DD_f =  \DD/\DD f$ and $\DD_g = \DD/\DD g$ are isomorphic.   \end{prop} 

\begin{proof}   If $
\chi \in \mathsf{Ctd}_\FF(\mathsf W)$, then $n\chi(w_n) = \chi([w_0, w_n]) = [w_0, \chi(w_n)]$,
which implies that $\chi(w_n)$ lives in the eigenspace $\FF w_n$  of  $\ad_{w_0}$  corresponding to $n$. Thus,  $\chi(w_n) = \lambda_n w_n$ for some $\lambda_n \in \FF$.
But then the above calculation says:   $n \lambda_n w_n = \chi([w_0,w_n]) = [\chi(w_0), w_n] = n \lambda_0 w_n$, which forces $\lambda_n = \lambda_0$ for all $n$.    Hence,  $\chi = \lambda_0\, \mathsf{id}_{\mathsf W}$ and $\mathsf{Ctd}_\FF(\mathsf{W}) = \FF \mathsf{id}_{\mathsf W}$. (Compare
\cite[Ex.~2.25]{BN06}.) 
 
Any $\chi \in \mathsf{Ctd}_\FF(\DD_f)$ satisfies $\chi(r) = \chi(1)r$ for all $r$.  Thus, if
$s_\chi = \chi(1)$, we have  $\chi(r) = s_\chi r$, and the map $\chi \mapsto s_{\chi}$ shows 
that  $\mathsf{Ctd}_\FF(\DD_f) \cong \DD_f$.  
 
Now if  $\DD_f \otimes \mathsf W   \cong \DD_g \otimes \mathsf W$, then their centroids are isomorphic.
Hence,
 \begin{align*} \mathsf{Ctd}_\FF(\DD_f \otimes \mathsf W) \   \cong \  \mathsf{Ctd}_\FF(\DD_g \otimes \mathsf W) \ \ &\iff \\
\mathsf{Ctd}_\FF(\DD_f) \otimes   \mathsf{Ctd}_\FF(\mathsf W) \ \cong \ \mathsf{Ctd}_\FF(\DD_g) \otimes   \mathsf{Ctd}_\FF(\mathsf W)   
& \iff    \\
\DD_f \otimes  \FF \mathsf{id}_{\mathsf W}\ \cong \ \DD_g\otimes  \FF \mathsf{id}_{\mathsf W} \ \, &\iff  \ \,   
\DD_f  \cong  \DD_g. \end{align*} 
Conversely, if $\psi: \DD_{f}\rightarrow \DD_{g}$ is an isomorphism, then $\psi\otimes \mathsf{id}_{\mathsf W}: \DD_f \otimes \mathsf W\rightarrow \DD_g \otimes \mathsf W$ is an isomorphism, with inverse $\psi^{-1}\otimes \mathsf{id}_{\mathsf W}$.  \end{proof}
\smallskip

 \subsection{Special cases}  \hfil
 \smallskip
 
In this concluding  subsection, we summarize the derivation results for the well-known examples
$\A_1$ (Weyl algebra), $\A_x$ (universal enveloping algebra of the two-dimensional non-abelian Lie algebra), and $\A_{x^2}$ (Jordan plane).  As mentioned earlier, the result for the
Weyl algebra goes back to Sridaran \cite{Sr61} and can be found in \cite[Sec.~4.6]{Dix96}  (see also Proposition \ref{P:Weyl0} above).   In Theorem 4.6 ($\chara(\FF) = 0$), Theorem 4.10 ($\chara(\FF) = p>2$), and Theorem 4.16 ($\chara(\FF) = 2$) of \cite{shirikov05}, Shirikov has computed the derivations of the Jordan plane $\A_{x^2}$.  The results for $\A_{x^2}$ in \cite{shirikov05}  (see also \cite{shirikov07-2}) are stated in  a different form  from what is given in Theorem \ref{thm:exams} below and in the next section for prime characteristics.
The assertions about $\hoch(\A_h)$  in 
the next theorem follow from Section \ref{S:hochstr}.  
 \smallskip

\begin{thm}\label{thm:exams}  Assume $\chara(\FF) = 0$, and for $g \in \DD$, let $D_g$
denote the derivation of $\A_h$ with $D_g(x) = 0$ and $D_g(\hat y) = g$.   Then
\begin{itemize}
\item[{\rm (i)}]   For  $\A_1$,   $\der(\A_1) = \inder(\A_1)$, so $\hoch(\A_1) = 0$. 
\item[{\rm (ii)}]  For $\A_x$,    $\der(\A_x) = \FF D_1 \oplus \inder(\A_x)$, so 
$\hoch(\A_x)$ is a one-dimensional Lie algebra  with basis  $\{D_1\}$.
\item[{\rm (iii)}]  For $\A_{x^m}$ with $m \geq 2$,  $\pi_h = x$,  and    
\begin{align*} \hoch(\A_{x^m})/\mathcal N  & =    \mathsf{Z}(\hoch(\A_{x^m})) \oplus [\hoch(\A_{x^m}), \hoch(\A_{x^m})]/\mathcal N  \\
& =   \FF D_{x^{m-1}} \oplus [\hoch(\A_{x^m}), \hoch(\A_{x^m})]/\mathcal N \\ & \cong \ 
  \FF D_{x^{m-1}} \oplus \mathsf{W} \end{align*}
where $\mathsf{W} = \mathsf{span}_\FF\{w_i \mid i \geq -1\}$ is  the Witt algebra.  The ideal $\mathcal N$
is nilpotent of index $\leq m-1$. In particular, $\mathcal N = 0$ when $m = 2$. 
\end{itemize}
\end{thm}
\medskip

\begin{section}{$\der(\A_h)$ when  $\chara(\FF) = p > 0$}\label{sec:derChar-p} \end{section} 
\smallskip
  
\noindent {\it Throughout we assume that the field $\FF$ has characteristic $p > 0$, $h\neq 0$, and $\varrho_h$ is as in  Definition \ref{D:varr}.  Our main results in this section are Theorem \ref{T:decomp2} and Corollary \ref {C:decomp2}, which give direct sum decompositions for $\der(\A_h)$ as a module over the center $\centh$ of $\A_h$,  and Theorem \ref{T:pfreeHH1}, which gives necessary and sufficient conditions for $\hoch(\A_h)$ to be a free $\centh$-module.  In the final subsection, we determine the Lie brackets  in $\der(\A_h)$.}   \medskip
 
\subsection{The derivations $D_g$ and the decomposition} \hfil  \smallskip

From Theorem \ref{thm:derdecomp}, we know that for every  $D \in \der(\A_h)$ there exist $E \in \mathcal E = \{F \in \der(\A_1) \mid F(\A_h) \subseteq \A_h\}$ and $g \in  \DD$  so that $D = D_g + E$,  where $D_g$ is the derivation of $\A_h$ given by  $D_g(x) = 0$ and $D_g(\hat y) = g$. The main problem is to determine conditions for $E \in \der(\A_1)$ to restrict to a derivation of $\A_h$. 
Theorem  \ref{thm:derA1} tells us that every derivation of $\A_1$ has  the form  $wE_x + zE_y + \ad_a$ where $w,z \in \cent1$, $a \in \A_1$ and $E_x$, $E_{y}$ are as in  \eqref{eq:ExEy}. However, it is not generally true that  $wE_x$ and $zE_{y}$ restrict to $\A_{h}$ for arbitrary elements $w,z$ of $\cent1 = \FF[x^p, y^p]$.   
 
\subsection{Derivations of the form $wE_x$}  \hfil  

\begin{lemma}\label{lem:wEx} Let $\chara(\FF) = p > 0$, and assume $E = wE_x + zE_y + \ad_a \in \der(\A_1)$ restricts to a derivation of $\A_h$, where $w,z\in \cent1$ and $a \in \A_1$.   Then $w \in \centh$.  \end{lemma}

\begin{proof}   Derivations map the center to itself,  so by Theorem \ref{L:center} and Lemma \ref{lem:ee} we know  that  $E(x^p) =  -w \in \cent1 \cap \A_h = \centh$.  \end{proof} 

We will provide necessary and sufficient conditions   on $w \in \mathsf Z(\A_h)$ for  $wE_x$ to restrict to a derivation of $\A_h$,  but this will require the next lemma.   \smallskip 

\begin{lemma}\label{lem:vh^{p-1}-in-F[x^p]}
Let $\varrho_h$ be as in \eqref{eq:vrhodef}, and assume $v \in \DD$.  Then $v h^{p-1} \in \FF[x^p]$ if and only if 
$v' h = v h'$ if and only if $v \in \FF[x^p] \frac{h}{\varrho_h}$. \end{lemma}

\noindent \emph{Proof.}  \begin{align*} vh^{p-1} \in \FF[x^p] &\iff  (vh^{p-1})' = 0 
\iff \, v'h = vh'   \,\iff \, \left (vh^{-1}\right)'  = 0 \\ 
&\iff v \in (\DD \cap \centh)\frac{h}{\varrho_h}   = \FF[x^p] \frac{h}{\varrho_h}  \, \, \hbox{\rm by
Lemma \ref{L:del0ker}\,(d)}.  \ \quad \square  \end{align*}

\begin{prop}\label{prop:wEx} 
Assume $\chara(\FF) = p > 0$ and let $w \in \centh$.   The following are equivalent. \begin{itemize}
\item[{\rm (i)}]   $w E_x$ restricts to a derivation of  $\A_h$;  
\item[{\rm (ii)}]  $w\in\centh  \frac{h^p}{\varrho_h}$;  
\item[{\rm (iii)}] $wE_x(x) \in \A_h$; 
\item[{\rm (iv)}] $wE_x \in \centh \be_x$, where $\be_x =  \frac{h^p}{\varrho_h} E_x$.
\end{itemize} 
 \end{prop}
 
 \begin{proof}
Since $w \in \centh$, we may assume  $w = \sum_{i \equiv 0 \modd p} s_i  h^{i} y^{i}$,  where
$s_i \in \FF[x^p]$ for all $i$.     Now
$w E_x(x) =  \textstyle{\sum_{i \equiv 0 \modd p}} \  s_i  h^{i} y^{i+p-1} \in \A_h  \iff   
h^{p-1}$ divides $s_i$ for each $i   \iff$
for each $ i, \ s_i  = w_i \frac{h}{\varrho_h}h^{p-1} = w_i \frac{h^p}{\varrho_h} \in \FF[x^p]$  for some $w_i\in \FF[x^p],$  
by Lemma \ref{lem:vh^{p-1}-in-F[x^p]}.  Therefore,  (ii) and (iii) are equivalent.

The implication (i) $\Longrightarrow$ (iii) is clear.  Now assume $wE_x(x) \in \A_h$.  Then by the equivalence of (ii) and (iii), we may suppose that $w =  u \frac{h^p}{\varrho_h}$ for some $u \in \centh$.   Now Lemma \ref{lem:ee}\,(f) implies that $E_x(\hat y) \in  h' y^p + \sum_{i=0}^{p-1} \DD y^i$, so $wE_x(\hat y) = u \frac{h^p}{\varrho_h} E_x (\hat y) \in u \frac{h^p}{\varrho_h} h' y^p + \sum_{i=0}^{p-1} \DD u \frac{h^p}{\varrho_h} y^i$, which belongs to $\A_h$ since $\varrho_h$ divides $h'$.  Thus, (ii) implies (i).

It is clear that  (ii)  and (iv) are equivalent, as $E_{x} \neq 0$ and $\A_{1}$ is a domain. \end{proof}

\begin{thm}\label{thm:z_1E_x-restricts}
Assume $\chara(\FF) = p > 0$, and let $E = wE_x + zE_y + \ad_a  \in \der(\A_1)$ with $w, z \in \cent1 = \FF[x^p,y^p]$,  and $a \in \A_1$.  If  $E \in \der (\A_h)$, then  $w E_x \in \der (\A_h)$ and $w\in\centh  \frac{h^p}{\varrho_h}$.
\end{thm}
\begin{proof}
Since  $E(x) \in \A_h$, we have $wy^{p-1} +  [a, x] \in \A_h$.  Observe that 
$$w y^{p-1}  \in \bigoplus_{i \equiv -1 \modd p} \DD y^i  \qquad \hbox{\rm and}  \qquad
[a, x]   \in \bigoplus_{i \not \equiv -1 \modd p} \DD y^i.$$
Thus  $wy^{p-1} \in \A_h$ and $[a, x] \in \A_h$.  This implies that $wE_x (x) = wy^{p-1} \in \A_h$, and the result now follows from Lemma~\ref{lem:wEx} and Proposition \ref{prop:wEx}.
\end{proof}

\subsection{Derivations of the form $D = zE_y + \ad_a$}\label{subsec:zEy+ad_a} \hfil
\smallskip

In view of Theorems \ref{thm:derA1}, \ref{thm:derdecomp}, and  \ref{thm:z_1E_x-restricts},  we know that every derivation of $\A_h$ has the form $D_g + u\be_x + zE_y + \ad_a$,  where $g \in \DD$,  $D_g$ and $u\be_x$ are derivations of $\A_h$, $u \in \centh$, $z \in \cent1$, $a\in\A_{1}$,  and $\be_x = \frac{h^p}{\varrho_h}E_x$.  Moreover, every $D_g + u\be_x$ with $g \in \DD$ and $u \in \centh$ gives a derivation of $\A_h$.  For that reason,  we may assume that $D = zE_y + \ad_a$ is a derivation of $\A_1$ that restricts  to a derivation of $\A_h$. 
\smallskip  

\begin{lemma}\label{lem:reduceTo-a=a_p}
Let  $D = z E_y + \ad_a \in \der(\A_1)$ for some  $z \in \mathsf Z(\A_1)$ and $a \in \A_1$,  and suppose $D \in \der(\A_h)$. Then  $a = b+c$, where  $b \in \norm_{\not \equiv 0}$ and $c  \in \mathsf{C}_{\A_1}(x) = \FF[x,y^p]$ as in Remark \ref{R:norm}, 
and both $\ad_b$ and $zE_y + \ad_c$  are derivations of $\A_1$ that restrict to derivations of $\A_h$.   Moreover, if $a = \sum_{i \geq 0} r_i y^i$ and $z = \sum_{i \equiv 0 \modd p} c_i y^i$, where $r_i \in \DD$ and $c_i \in \FF[x^p]$ for all $i$, then $zE_y + \ad_a = D_f +  \tilde z E_y + \ad_{\tilde c} + \ad_b$, where $\tilde z = \sum_{i \equiv 0 \modd p, i > 0} c_i y^i$,  $\tilde c = \sum_{i \equiv 0 \modd p, i > 0} r_i y^i$ and $f = c_0hx^{p-1} - \delta(r_0)\in \DD$,   and  $\tilde z E_y + \ad_{\tilde c} \in \der(\A_h)$.
\end{lemma}

\begin{proof}  
Let $a$ and $z$ be as in the statement of the lemma.  Since $zE_y(x) = 0$, we have  $D(x) \in \A_h$ if and only if $[a, x]  \in \A_h$.    
 As in \eqref{eq:ax},   $[a, x]  \in \A_h \iff r_i \in \DD h^{i-1}$ for all $i \not\equiv 0 \modd p$.  Thus, we write $r_i = s_i h^{i-1}$  for each such $i$,   where $s_i \in \DD$.    

Now $D(hy) = D(\hat y) - D(h') \in \A_h$,  and we reason as in \eqref{eq:ahy} that 
\begin{align} \label{eq:Dhy} \hspace{-.3cm} \ D(hy) \in \A_h  \hspace{-.1cm}
&\hspace{-.1cm}\iff  zhx^{p-1} + \sum_{i \not \equiv 0 \modd p} s_i h^{i-1}h' y^i 
-   \sum_{i \equiv 0 \modd p} r_i' h y^i \in \A_h  \nonumber  \\ 
 \hspace{-.55cm} \iff &\ \sum_{i \equiv 0 \modd p} \left(c_ix^{p-1}-r_i'\right)hy^i \in \A_h \ \hbox{\rm and}  \sum_{i \not \equiv 0 \modd p} s_i h^{i-1}h' y^i   \in \A_h
  \\
\hspace{-.55cm} \iff &\ h^{i-1}\,\vert\,(c_ix^{p-1}-r_i') \ \ \hbox{\rm for all} \ \ i \equiv 0 \modd p, \  i > 0,   \ \ \hbox{\rm and}  \nonumber \\
 &  \ h \,\vert\, s_i h'  \ \ \hbox{\rm for all} \ \ i \not \equiv 0 \modd p. \nonumber 
\end{align}  

Hence,  if $D \in \der(\A_h)$,  then  $h \,\vert \, s_i h'$  for all $i \not \equiv 0 \modd p$   by \eqref{eq:Dhy},   and  we know by Lemma \ref{L:defpi}  that  $\pi_h$ divides each such $s_i$.  Then there exist $b_i \in \FF[x]$ so that $r_i = b_i \pi_h h^{i-1}$ for each $i \not \equiv 0 \modd p$, and $b =\sum_{i \not\equiv 0 \modd p} b_i \pi_h h^{i-1} y^i \in\norm_{\not \equiv 0}$ by Theorem \ref{T:norm}\,(b).  Then  $\ad_b$ and $D$ belong to  $\mathcal E = \{F \in \der(\A_1) \mid F(\A_h) \subseteq \A_h\}$.  Setting $c = a-b  = \sum_{i \equiv 0 \modd p} r_i y^i  \in \mathsf{C}_{\A_1}(x)$, we have that $z E_y + \ad_c = D - \ad_b \in \mathcal E$.  Thus both $\ad_b$ and $zE_y + \ad_c$ are derivations of $\A_1$ that restrict to derivations of $\A_h$.

From  $E_y(x) = 0$ and $E_y(\hat y) = x^{p-1}h$ (Lemma \ref{lem:ee}\,(e)), we see that $E_y = D_{x^{p-1}h} \in \mathcal D_\DD \subseteq \der(\A_h)$.   Also, from Proposition \ref{prop:autg}\,(ii), we have $\ad_r = -D_{\delta(r) } \in \mathcal D_{\DD}$ for all $r \in \DD$.  As a result,  if $z,a,b,c$ are as above, then $zE_y + \ad_a = D_f +  \tilde z E_y + \ad_{\tilde c} + \ad_b$, where $\tilde z = \sum_{i \equiv 0 \modd p, i > 0} c_i y^i$,  $\tilde c = \sum_{i \equiv 0 \modd p, i > 0} r_i y^i$ and $f = c_0hx^{p-1} - \delta(r_0)\in \DD$,   and  $\tilde z E_y + \ad_{\tilde c} \in \der(\A_h)$. 
\end{proof}

\subsection{The restriction map $\mathsf{Res} : \der (\A_h) \to \der ( \centh)$}  \hfil \smallskip

When  $\chara (\FF) = p > 0$, \ $\centh = \FF[x^p, \ze ]$, where $\ze = h^p y^p = \hat y^p - \frac{\delta^p(x)}{h} \hat y$.   The map $\mathsf{Res} : \der (\A_h) \to \der ( \centh)$ given by restricting a derivation to  $\centh$ is a morphism of Lie algebras.  In this section, we investigate this map and  describe its  kernel and image.  This will enable
us to determine $\der(\A_h)$ in the next section.  The derivation $\delta^p$ plays a significant role.   As $\delta^p$ sends $x$ to $\delta^p(x)$,  then   $\delta^p = \delta^p(x) \frac{d}{dx}$ and 
\begin{equation}\label{eq:delp} \delta^p(r) = \delta^p(x) r' \qquad \hbox{\rm for all} \ \ 
r\in \DD.  \end{equation}   

\begin{lemma}\label{L:deltap} Let $\ze = h^p y^p \in \centh$, and write $h^{p-1}=\sum_{i=0}^{p-1} \overbar h_{i}x^{i}$  with $\overbar h_{i}\in\FF[x^{p}]$ for all $i$. 
\begin{itemize}
\item [{\rm (a)}] For any $r\in\DD$, $D_r(\ze)=\delta^{p-1}(r) - \frac{\delta^p(x)}{h} r = \left( rh^{p-1}\right)^{(p-1)}$.
\item [{\rm (b)}] $\delta^p(x)= -\left( h^{p-1}\right)^{(p-1)} h =   \overbar h_{p-1}h$  so that $\delta^p =  \overbar h_{p-1} \delta$  and $D_1 ( \ze) = - \overbar h_{p-1}$.
\end{itemize}
\end{lemma}
\begin{proof} (a)   For any  $r \in \DD$, we have 
\begin{align*} D_r(\ze) &=  D_r(\hat y^p - \textstyle{\frac{ \delta^p(x)}{h}} \hat y) = \sum_{n=0}^{p-1} \hat y^n r  \hat y^{p-1-n} - \textstyle{\frac{ \delta^p(x)}{h}} r  \\
&=  \sum_{n=0}^{p-1}\sum_{j=0}^n {n \choose j}  \delta^j(r) \hat y^{p-1-j} -\textstyle{\frac{ \delta^p(x)}{h}}r  \\
&=  \sum_{j=0}^{p-1}\Bigg(\sum_{n=j}^{p-1} {n \choose j}\Bigg)  \delta^j(r) \hat y^{p-1-j}-\textstyle{\frac{ \delta^p(x)}{h}}r  
\ = \  \delta^{p-1}(r) - \textstyle{\frac{ \delta^p(x)}{h}}r. \end{align*} 
The fact that $D_r(\ze) = (r h^{p-1})^{(p-1)}$ comes from (c) of Corollary \ref{cor:D_g(a_n)}.

(b) Taking $r=1$ in part\,(a) yields $\left( h^{p-1}\right)^{(p-1)} = \delta^{p-1}(1) - \frac{\delta^p(x)}{h} = - \frac{\delta^p(x)}{h}$, and thus $\delta^p(x) = - \left( h^{p-1}\right)^{(p-1)}h$.   Since   $\left( x^{i}\right)^{(p-1)}=0$ for $0 \leq i<p-1$ and $\left( x^{p-1}\right)^{(p-1)}=-1$, it follows that $\left( h^{p-1}\right)^{(p-1)}=\left( \sum_{i=0}^{p-1} \overbar h_{i}x^{i}\right)^{(p-1)}= -\overbar h_{p-1}$.  Hence,  $\delta^p(x) = \overbar h_{p-1}h$, and $\delta^p = \delta^p(x) \frac{d}{dx} = \overbar h_{p-1} h \frac{d}{dx} = \overbar h_{p-1} \delta$ by \eqref{eq:delp}.
\end{proof}

\begin{prop}\label{prop:kerres} 
The kernel of the restriction map \ $\mathsf{Res} : \der (\A_h) \to \der ( \centh)$  is  
\vspace{-.35cm}
\begin{center}{$\mathsf{\ker\,Res} = \mathcal D_{\Theta} +  \{ \ad_a \mid a \in \norm\},$} \end{center} 
where $\mathcal D_{\Theta} = \{D_r \mid r \in \Theta\}$ and $\Theta = \left \{r \in \DD \mid  \delta^{p-1}(r) = \textstyle{ \frac{\delta^p(x)}{h}} r \right \}$.
\end{prop} 
\begin{proof}The right side is contained $\mathsf{ker\,Res}$  by (a) of Lemma \ref{L:deltap}  
and the fact that $\centh \subseteq \cent1$.  For the other direction,  suppose that  $D  \in \mathsf{ker\,Res}$.   In view of Lemma \ref{lem:reduceTo-a=a_p}, we may suppose    $D = D_r + u \be_x + \tilde z E_y + \ad_b + \ad_{\tilde c}$   for some $r \in \DD$, $u \in \centh$, $\tilde z =  \sum_{i \equiv 0 \modd p, i>0} c_i y^i \in \mathsf{Z}(\A_1)$ with $c_i \in \FF[x^p]$, \ 
$b \in \norm_{\not \equiv 0}$, and $ \tilde c  \in  \sum_{i \equiv 0 \modd p, i > 0} \DD y^i$.
Since $\ad_b \in  \mathsf{ker\,Res}$, we can assume that $E = D_r + u \be_x + \tilde z E_y +\ad_{\tilde c}  \in  \mathsf{ker\,Res}$.   Applying $E$ to $x^p$,  we see that $u = 0$.     Since $\ad_{\tilde c}(\ze) = 0$,  we have 
\begin{eqnarray*} 0 &=& \big (D_r + \tilde zE_y\big)(\ze) =  \delta^{p-1}(r) - \textstyle{\frac{ \delta^p(x)}{h}}r + 
 \tilde z E_y(h^p y^p) \\
&=&  \delta^{p-1}(r) -  \textstyle{\frac{ \delta^p(x)}{h}}r -  \tilde z h^p \\
&=& \delta^{p-1}(r) - \textstyle{\frac{ \delta^p(x)}{h}}r  - \sum_{i \equiv 0 \modd p, i>0} c_i h^p y^i. \end{eqnarray*}
From this we deduce  that $\tilde z = 0$  and $\delta^{p-1}(r) =  \frac{\delta^p(x)}{h}r$. 
Therefore, $\ad_{\tilde c} = E - D_r  \in \der(\A_h)$, \,   $r \in \Theta$,   and
$D \in  \mathcal D_\Theta +  \{ \ad_a \mid a \in \norm\}$. \end{proof}  

In light of Proposition \ref{prop:kerres}, we would like to determine more information about $\Theta$.  

\begin{prop}\label{P:alt-omega}  Let $h^{p-1}=\sum_{i=0}^{p-1} \overbar h_{i}x^{i}$, with $\overbar h_{i}\in\FF[x^{p}]$ for all $i$, as in Lemma \ref{L:deltap},  and let $\mathsf{Res} : \der (\A_h) \to \der ( \centh)$ be the restriction map.
 \begin{itemize}  
\item [{\rm (a)}]  Let $\vartheta: \DD \rightarrow \FF[x^p]$ be the $\FF[x^p]$-module map given by $\vartheta(r) = D_r(\ze).$
Then  \begin{eqnarray*}\Theta  &=&  \{ r \in \DD \mid   \delta^{p-1}(r) = \textstyle{\frac{\delta^p(x)}{h}} r \} 
= \{ r \in \DD \mid   \delta^{p-1}(r) = \overbar h_{p-1} r \} \\
 &=& \ker \vartheta = \{r \in \DD  \mid D_r \in \mathsf{\ker\,Res\,}\}   \\
 &=&   \{r \in \DD \mid (r h^{p-1})^{(p-1)} = 0\}  \\
&=& \left \{ r \in \DD \mid rh^{p-1} \in\im\, \textstyle{\frac{d}{dx}}\right \} \  = \ \{ r \in \DD \mid rh^{p} \in  \im \, \delta  \}.\end{eqnarray*}
In particular, $\Theta$ contains $\im \, \delta$.
\item [{\rm (b)}] $\Theta$ is a free $\FF[x^p]$-module of rank $p-1$ and $\delta^{p-1}\neq 0$. If $\delta^{p}=0$ then $\FF[x^{p}]\subseteq \Theta$; \ if $\delta^{p}\neq 0$ then $\FF[x^{p}]\cap \Theta= 0$.
\item [{\rm (c)}] $\im\, \vartheta = \{ D_{r}(\ze) \mid r\in\DD \}=\FF[x^{p}] \overbar h$, where $\overbar h$ is the greatest common divisor in $\FF[x^{p}]$ of $\{ \overbar h_{i} \mid 0\leq i<p\}$.  Hence, $\mathsf{Res}(\mathcal D_\DD) = \FF[x^p] \overbar h \frac{d}{d\ze}$.
 \item [{\rm (d)}]  Let $\ms_i \in \FF[x^p]$ be such that $\overbar h  = \sum_{i=0}^{p-1} \ms_i \overbar h_i$, and set $\ms =
-\sum_{i=0}^{p-1} \ms_i  x^{p-1-i}$.   Then $\mathsf{Res}(D_{\ms}) = \overbar h \frac{d}{d\ze}$ and
$\DD = \FF[x^p] \ms \,\oplus\, \Theta$.
\item [{\rm (e)}]  For all $f\in\DD$, $\left( f'f^{p-1}\right)^{(p-1)}=-(f')^{p}$. In particular, $D_{\frac{h'}{\varrho_{h}}}(\ze)=-\frac{(h')^{p}}{\varrho_{h}}$.
\end{itemize}
\end{prop} 

\noindent \emph{Proof.} \  
(a)  \  Let $r\in\DD$. Then by Lemma \ref{L:deltap}\,(a), 
\begin{eqnarray*}
r\in\Theta & \iff & \left( rh^{p-1}\right)^{(p-1)}=0\\
& \iff & rh^{p-1} \in \sum_{i=0}^{p-2}\FF[x^{p}]x^{i} = \im\, \frac{d}{dx} \\
& \iff &   rh^{p} \in \im \, \delta.
\end{eqnarray*}
In particular, $\delta(r)h^{p}=\delta(rh^{p})\in \im \, \delta$ for all $r\in\DD$, so (a) holds. 

(b) and (c)  \   For the  $\FF[x^{p}]$-module map $\vartheta:\DD\rightarrow\FF[x^{p}]$ given by $\vartheta(r)=\left( rh^{p-1}\right)^{(p-1)}$,    $\im\, \vartheta$ is the ideal of $\FF[x^{p}]$ generated by $\{ \vartheta(x^{j}) \mid 0\leq j<p \}$. Note that $x^{j}h^{p-1}=\sum_{i=0}^{p-1}\overbar h_{i}x^{i+j}$, so $\vartheta(x^{j}) = -\overbar h_{p-1-j}$. Since $h\neq 0$, we cannot have $\overbar h_{i}=0$ for all $0\leq i<p$, thus $\im\, \vartheta= \FF[x^{p}]\,\overbar h$, where $0\neq \overbar h \in\FF[x^{p}]$ is the greatest common divisor of $\{ \overbar h_{i} \mid 0\leq i<p\}$. In particular, $\im\, \vartheta$ is a free $\FF[x^{p}]$-module of rank one, and it follows that $\Theta=\ker\, \vartheta$ is free of rank $p-1$. 

If  $\delta^{p-1} = 0$, then $\delta^{p}=0$ and $\Theta=\DD$, which is a contradiction, as $\DD$ has rank $p$ as an $\FF[x^{p}]$-module. Thus  $\delta^{p-1} \neq 0$. Suppose that $\delta^{p}=0$. Then $\Theta = \{ r \in \DD \mid   \delta^{p-1}(r) = 0 \}$,  and it is clear that 
$\FF[x^{p}]\subseteq \Theta$. Suppose now that $\delta^{p}\neq 0$. Then,  $\delta^{p}(x)\neq 0$.  If $r\in\FF[x^{p}]\cap \Theta$,  then $0=\delta^{p-1}(r) = \frac{\delta^p(x)}{h} r$, so $r=0$ and $\FF[x^{p}]\cap \Theta = 0$, as asserted in (b).

 (d)   As  $\vartheta(x^{p-1-i}) =  -\overbar h_i$, we have  $\mathsf{Res}(D_{x^{p-1-i}}) = -\overbar h_{i}
\frac{d}{d\ze}$ for $0 \leq i < p$.   Now if $\ms_i \in \FF[x^p]$, $0 \leq i < p$,  are taken so that $\overbar h  = \sum_{i=0}^{p-1} \ms_i \overbar h_i$, then for $\ms =
-\sum_{i=0}^{p-1} \ms_i  x^{p-1-i}$,  it follows that  $\dms = -\sum_{i=0}^{p-1} \ms_i D_{x^{p-1-i}}$  and $\mathsf{Res}(\dms)
= \left( \sum_{i=0}^{p-1} \ms_i \overbar h_{i}\right) \frac{d}{d\ze} = \overbar h \frac{d}{d\ze}$.  

Suppose $r\in \DD$.  Then by (c),  there exists $u \in \FF[x^p]$ such that   $\mathsf{Res}(D_r) =  u \mathsf{Res}(D_{\ms})$.
Hence, $\mathsf{Res}(D_{r-u\ms}) = 0$, \ $r -u \ms = t \in \Theta$,  and $r = u\ms + t$.    This shows that $\DD = \FF[x^p] \ms + \Theta$.   Since $\vartheta(u\ms) = u\overbar h \neq 0$ for all nonzero $u \in \FF[x^p]$,   it is apparent the sum is direct. 

It remains to prove part  (e). 
We assume the stated equality holds for $f, g\in\DD$ and show it for $f+g$.   Now
 \begin{align*}
(f+g)'(f+g)^{p-1}
&=   f'\sum_{k=0}^{p-1}(-1)^{k}f^{k}g^{p-1-k} + g'\sum_{k=0}^{p-1}(-1)^{k}f^{k}g^{p-1-k}\\
&\hspace{-1.3cm}= f'f^{p-1} + g'g^{p-1} + f'\sum_{k=0}^{p-2}(-1)^{k}f^{k}g^{p-1-k} + g'\sum_{k=1}^{p-1}(-1)^{k}f^{k}g^{p-1-k}\\
&\hspace{-1.3cm}=  f'f^{p-1} + g'g^{p-1} + \sum_{k=0}^{p-2}(-1)^{k}\left(f'f^{k}g^{p-1-k} -f^{k+1}g'g^{p-2-k}\right)\\
&\hspace{-1.3cm}=  f'f^{p-1} + g'g^{p-1} + \sum_{k=0}^{p-2}(-1)^{k}\frac{1}{k+1}\left(f^{k+1}g^{p-1-k}\right)'.
\end{align*}
Since $\left( \im \frac{d}{dx} \right)^{(p-1)} = 0$, we see that  $f \mapsto  (f' f^{p-1})^{(p-1)}$ is 
 an additive mapping on $\DD$.
Hence, it will be enough to show that $\left( f'f^{p-1}\right)^{(p-1)}=-(f')^{p}$ for $f=\gamma x^{m}$, with $m\geq 0$ and $\gamma\in\FF$. This is immediate from 
\begin{eqnarray*}
\left( f'f^{p-1}\right)^{(p-1)} &=& \left( \gamma^{p}mx^{mp-1}\right)^{(p-1)} =\gamma^{p}mx^{(m-1)p}\left( x^{p-1}\right)^{(p-1)}\\ &=& -\gamma^{p}mx^{(m-1)p} =-\left(\gamma mx^{m-1}  \right)^{p}\\ &=& -(f')^{p},
\end{eqnarray*}
so the equality in (e) holds for all $f\in\DD$. Taking $f=h$ gives
\begin{equation*}
 \ \  \qquad  D_{\frac{h'}{\varrho_{h}}}(\ze) = \left( \frac{h'}{\varrho_{h}}h^{p-1}\right)^{(p-1)}=\frac{1}{\varrho_{h}}\left( h'h^{p-1}\right)^{(p-1)}=-\frac{(h')^{p}}{\varrho_{h}}.  \hspace{1.9cm} \square
\end{equation*} 

\begin{remark}
The map $\vartheta :\DD\rightarrow\FF[x^p]$, $r\mapsto\left( rh^{p-1}\right)^{(p-1)}$,  can be thought of as an inner product with $-(\overbar h_{p-1}, \ldots, \overbar h_{0})$:  \  If we identify  $r=\sum_{k=0}^{p-1}r_{k}x^{k}\in \bigoplus_{k=0}^{p-1}\FF[x^{p}]x^{k}$ with the tuple $(r_{0}, \ldots, r_{p-1})$, we can view $\vartheta$ as the map \break  $(r_{0}, \ldots, r_{p-1})\mapsto  -\sum_{i=0}^{p-1} r_{i} \overbar h_{p-1-i}$. Then $\Theta$ is the orthogonal complement of the line generated by  $(\overbar h_{p-1}, \ldots, \overbar h_{0})$.
\end{remark}

\begin{exam}\label{Ex:Psi}   Assume  $h = g^m$, where $m \geq 0$ and  $g = x-\gamma$ for some  $\gamma \in \FF$.  Then $\DD = \bigoplus_{i \geq 0}  \FF g^i$,  and  
$$\im \, \delta  = \bigoplus_{i= 0}^{p-2} \FF[g^p] g^{m+i}  = \bigoplus_{{j \geq m} \atop {j \not \equiv m-1 \modd p}} \FF g^j.$$
Now for $r = \sum_{i \geq 0} r_i g^i$ with $r_i \in \FF$ for all $i$, 
\begin{align*} r \in \Theta  & \iff  r h^{p} = \sum_{i \geq 0} r_i  g^{i+mp} \in \im \, \delta = \bigoplus_{{j \geq m} \atop {j \not \equiv m-1 \modd p}} \FF g^j\\  & \iff  r_i = 0 \ \ \hbox{\rm for}\ \  i \equiv m-1 \ \modd p.\end{align*}
Hence, \begin{center} {$\Theta= \displaystyle{ \bigoplus_{{j\geq 0} \atop {j \not\equiv m-1\modd p} }\FF g^j.}$}\end{center}  
Recall $\delta_0(r) = \delta(r \pi_h h^{-1}) =  (r \pi_h h^{-1})' \, h$.   If $p \nmid m$,  then $\pi_h=g$ and from this we see 
$\delta_0(g^j) = \delta(g^{j+1-m}) = (j+1-m)g^j$, so that $g^j \in \im \,\delta_0$ exactly when $j \not \equiv m-1 \modd p$.     If $p \mid m$,  then $\pi_h=1$ and $\delta_0(g^j) = \delta(g^{j-m}) = jg^{j-1}=\frac{d}{dx} (g^j)$, so $\im \, \delta_0=\im\,\frac{d}{dx}$. In either event,  we have 
$$\Theta = \im \, \delta_0 = \bigoplus_{{j\geq 0} \atop {j \not\equiv m-1\modd p} }\FF g^j
=  \left(\bigoplus_{{0 \leq j < m} \atop  {j \not\equiv m-1\modd p} }\FF g^j\right) \oplus \im\,\delta.$$

Some cases of special interest  are
\begin{itemize}
\item for $h=1$, \  $\Theta = \im\, \delta =  \bigoplus_{j=0}^{p-2}\FF[x^{p}]x^{j}= \im\,\frac{d}{dx}$;
\item for $h= x$, \  $\Theta = \im \, \delta = \bigoplus_{j=1}^{p-1}\FF[x^{p}]x^{j}$;
\item for $h= x^n$ with $2 \leq n < p$, \ \ $\Theta = \Big(\bigoplus_{j=0}^{n-2} \FF x^j \Big)  \oplus  \im \, \delta$.  \end{itemize}  \end{exam}   

In view of Proposition \ref{prop:kerres},  we investigate the following. 

\begin{prop}\label{prop:inderres}  Suppose $D_r + \ad_a  \in \inder(\A_h)$ for some $r \in \DD$ and $a \in \norm$.  Then  
 $r \in \im \, \delta$,  $a \in \A_h + \cent1$,  and  $\ad_a,  D_r  \in \inder(\A_h)$.   Consequently, 
$$\mathcal D_\Theta \cap \{\ad_a \mid a \in \norm\} = \mathcal D_{\im\,\delta},$$
where  $\mathcal D_\Theta = \{D_r \mid r \in \Theta\}$ \ and \  $\mathcal D_{\im\,\delta} = \{D_r \mid r \in \im \, \delta\}.$  \end{prop}
\begin{proof}  For the first statement,  suppose  that $D_r + \ad_a = \ad_v$ for some $v \in \A_h$.   Then it follows from
$D_r = \ad_{v-a}$ that $v-a \in \C_{\A_1}(x)$.   Writing $v-a = \sum_{i \equiv 0 \modd p} w_i y^i$, where $w_i \in \DD$
for all $i$,   we have
$r = D_r(\hat y) = [v-a,\hat y] = \sum_{i \equiv 0 \modd p} [w_i y^i, yh] =  - \sum_{i \equiv 0 \modd p}  w_i' h y^i.$
As a result,   $r = -w_0'h \in \im\,\delta$ and $w_i' = 0$ for all $i > 0$.     Hence, $w_i \in \FF[x^p]$ for all $i > 0$ 
and $w  = \sum_{i \equiv 0 \modd p, i > 0} w_i y^i \in \cent1$.  Now $a = (v - w_0) - w \in \A_h + \cent1$, which
implies that $\ad_a = \ad_{v-w_0}$ and $D_r$ are in $\inder(\A_h)$.  

The  assertion about $\mathcal D_{\Theta}$  follows from  what we have just shown  and the fact that $D_{\delta(g)} = -\ad_g$ for all $g \in \DD$  by  (ii) of Proposition \ref{prop:autg}. 
\end{proof}  

From  Proposition  \ref{prop:inderres},  we can conclude the following:
\begin{cor}\label{cor:induce} \, The kernel of the  induced map \,  $\overbar{\mathsf{Res}}: \hoch(\A_h)  \rightarrow \der(\centh)$ is 
\begin{eqnarray*}
\mathsf{ker\,\overbar{Res}}&=&
 \big(\mathcal D_\Theta + \{ \ad_a \mid a \in \norm\} \big)/\inder(\A_h) \\ 
&\cong& \big( \mathcal D_\Theta/ \mathcal D_{\im\,\delta}\big) \,   \oplus  \,  \big( \{ \ad_a \mid a \in \norm\}/\inder(\A_h) \big) \\
&\cong& \big(\Theta/ \im \, \delta\big) \, \oplus \,  \big(\norm /\big(\A_h+\cent1\big)\big),
\end{eqnarray*}
where the isomorphisms are as $\FF[x^{p}]$-modules. 
\end{cor}

Next, we investigate the image of  the map $\mathsf{Res}$.  Recall from Proposition \ref{P:alt-omega}\,(c)  that 
$\mathsf{Res}(\mathcal{D}_{\DD})=\FF[x^{p}]\overbar h\frac{d}{d\ze} = \FF[x^p] \mathsf{Res}(D_{\ms})$, where $\ms$ is
as in (d) of that proposition.  
Now using Lemma \ref{lem:ee}\,(c)  and $\be_x (\ze)= \frac{1}{\varrho_h} E_x (h^{p})\ze=-\frac{(h')^{p}}{\varrho_h}\ze$, we have 
\begin{equation}\label{eqn:E_x-center}
\be_x(x^{jp})  = -\frac{h^p}{\varrho_h} j x^{(j-1)p}  \quad \text{and} \quad  \be_x( \ze^k) = - k \ze^{k}\frac{(h')^{p}}{\varrho_h},  
\end{equation}
and thus,
\begin{equation*}
\mathsf{Res}(\be_x)=- \frac{1}{\varrho_h}\left(h^p\frac{d}{d(x^{p})}+ (h')^{p}\ze\frac{d}{d\ze}\right).
\end{equation*}
In particular,  for  
\begin{equation}\label{eq:brf}   \brf \, = \, \ze D_{\frac{h'}{\varrho_{h}}}-\be_x,   \quad \hbox{\rm we have} \quad  \mathsf{Res}(\brf) =\frac{h^{p}}{\varrho_{h}}\frac{d}{d(x^{p})}   \end{equation} 
by Proposition \ref{P:alt-omega}\,(e).

\begin{thm}\label{T:imres} Assume $\chara(\FF) = p > 0$, and let  $\mathsf{Res}: \der(\A_h) \rightarrow \der(\centh)$ be the restriction map  and $\overbar{\mathsf{Res}}: \hoch(\A_h)  \rightarrow \der(\centh)$ be the induced map.   Then the following hold. 
\begin{itemize} 
\item[{\rm (a)}]  $\im\,\mathsf{Res}= \im\,\mathsf{\overbar{Res}}$ is a free $\centh$-submodule of $\der(\centh)$ of rank 2 generated over $\centh$  by $\frac{h^{p}}{\varrho_{h}}\frac{d}{d\left(x^{p}\right)}$ and $\overbar h \frac{d}{d\ze}$, where $\overbar h$ is as in Proposition~\ref{P:alt-omega}\,(c).  

\item[{\rm (b)}]  If  $t_1 = x^p$, $t_2 = \ze$,  and if  $\centh$ is identified  with $\FF[t_1,t_2]$, then $\im\, \mathsf{Res}$ is isomorphic to the subalgebra of  the Witt algebra $\der(\FF[t_1,t_2])$  generated over $\FF[t_1,t_2]$ by  $\mathsf{d}_1 = \frac{h^{p}}{\varrho_{h}} \frac{d}{dt_1}, \, \mathsf{d_2} =\overbar h \frac{d}{dt_2}$, where  
$$[\mathsf{d}_1, \mathsf{d}_2] =\frac{d}{dt_1}\big(\overbar h\big)\, \frac{h^p}{\varrho_h \overbar h} \, \mathsf{d}_2.$$  
\end{itemize}
\end{thm}

\begin{proof}  
By the above and  Proposition \ref{P:alt-omega}, for part (a) it suffices to show that $$\im\,\mathsf{Res}\subseteq \centh\mathsf{Res}(\mathcal D_\DD) + \centh\mathsf{Res}(\be_x).$$   Given $D\in\der(\A_h)$,  we have established that there exist $g\in\DD$, $u\in\centh$, $z\in\cent1$, $b\in\norm_{\not \equiv 0}$ and $c  \in \mathsf{C}_{\A_1}(x)$, as in Lemma \ref{lem:reduceTo-a=a_p}, such that $D=D_g+u\be_x+\ad_b+E$, where $E=zE_y + \ad_c$ and $D_g$, $u\be_x$, $\ad_b$, $E\in\der{\A_h}$. Clearly, $\mathsf{Res}(D_g)$, $\mathsf{Res}(u\be_x)$, and $\mathsf{Res}(\ad_b)=0$ belong to  $\centh\mathsf{Res}(\mathcal D_\DD) + \centh\mathsf{Res}(\be_x)$, so it remains to argue that the same holds for $\mathsf{Res}(E)$. Note that $E(x)=0$, so $[E(\hat y), x]=0$, showing that $E(\hat y)\in\mathsf{C}_{\A_h}(x)=\centh\DD$.   Thus, $E \in 
\centh \mathcal{D}_\DD$ and   $\mathsf{Res}(E) \in \centh \mathsf{Res}(\mathcal{D}_\DD)$.  

For part (b),  observe that 
\begin{eqnarray} \ \qquad  [\mathsf{Res}(\brf), \mathsf{Res}(\dms)] &=&   \left[\frac{h^p}{\varrho_h}\frac{d}{d(x^p)},\, \overbar h \frac{d}{d\ze}\right] \\
&=& \frac{d}{d(x^p)}\big(\overbar h\big)  \frac{h^p}{\varrho_h \overbar h} \,\, \overbar h \frac{d}{d\ze} =\frac{d}{d(x^p)}\big(\overbar h\big) \frac{h^p}{\varrho_h \overbar h}\,\mathsf{Res}(\dms). \nonumber  \end{eqnarray}
\noindent  The result is apparent from that, since $\mathsf{d}_1 =   \frac{h^p}{\varrho_h}\frac{d}{dt_1} = \mathsf{Res}(\brf)$  and 
$\mathsf{d}_2 = \overbar h \frac{d}{dt_2} = \mathsf{Res}(\dms)$, where $t_1 = x^p$, $t_2 = \ze$. 
\end{proof}

\begin{exam}\label{EX:t1t2}  
Assume $h=x^{m}$, with $m\geq 0$. Write $m=kp+n$ with $k\geq 0$ and $0\leq n<p$, and set  $t_{1}=x^{p}$ and $t_{2}=\ze$, so that $\centh=\FF[t_{1}, t_{2}]$. Then $\frac{h^{p}}{\varrho_{h}}=t_{1}^{m-k}$  and  
\begin{equation}\label{eq:tone} \overbar h = \begin{cases}  t_1^{m-k} & \qquad  \hbox{\rm if} \  n = 0\\
 t_1^{m-k-1}  & \qquad  \hbox{\rm if} \  n \neq  0. \end{cases}\end{equation} 
 Thus,  $\im\,\mathsf{Res}$ is the Lie subalgebra of $\der(\FF[t_{1}, t_{2}])$ generated over $\FF[t_{1}, t_{2}]$ by 
\begin{eqnarray*}  t_{1}^{m-k}\frac{d}{dt_{1}}\quad\mbox{and} \quad t_{1}^{m-k}\frac{d}{dt_{2}}  && \quad  \hbox{\rm if} \ \  n = 0\\
\ \ t_{1}^{m-k}\frac{d}{dt_{1}}\quad\mbox{and} \quad t_{1}^{m-k-1}\frac{d}{dt_{2}} && \quad  \hbox{\rm if} \ \  n \neq  0.
\end{eqnarray*}   
Special cases  of this result are displayed in the table below:
\begin{equation*} 
\begin{tabular}[t]{|c||c|c|c|c|}
\hline
        $h$ 
       & $m$ & $k$ & $n$ & generators  
	\\
	\hline  \hline
$1$ & 
$0$ & $0$ & $0$ &  $\frac{d}{dt_{1}}, \  \frac{d}{dt_{2}}$  
\\  \hline
$x$ & 
$1$ & $0$ & $1$  &  $t_1\frac{d}{dt_{1}}, \  \frac{d}{dt_{2}}$  
\\  \hline
$x^2$ ($p > 2$) & 
$2$ & $0$ & $2$ & $t_1^2\frac{d}{dt_{1}}, \  t_1\frac{d}{dt_{2}}$ \\  \hline
$x^2$ ($p = 2$) & 
$2$ & $1$ & $0$ & $t_1\frac{d}{dt_{1}}, \  t_1\frac{d}{dt_{2}}$   \\
\hline
\end{tabular} 
\end{equation*}
When  $h=1$,  then $\overbar{\mathsf{Res}}$ is surjective, and by Corollary~\ref{cor:induce} we also know $\overbar{\mathsf{Res}}$ is injective, as $\Theta= \im \, \delta$, so we retrieve a previously established result:  the induced map   $\overbar{\mathsf{Res}}: \hoch(\A_1)  \rightarrow \der(\cent1)$ is an isomorphism (see Theorem \ref{thm:derA1}\,(b)). 
\end{exam}  

\subsection {Main theorems about derivations} \hfil
\medskip

Assume $\overbar h \in \FF[x^p]$ and $\ms\in \DD$ are as in Proposition \ref{P:alt-omega}, so that under the restriction map, 
$\mathsf{Res}(D_{\ms}) = \overbar h \frac{d}{d\ze}$.    
Recall from \eqref{eq:brf} that  the derivation  $\brf = \ze D_{\frac{h'}{\varrho_h}} - \be_x \in \der(\A_h)$  has the property that
$\mathsf{Res}(\brf) =  \frac{h^p}{\varrho_h} \frac{d}{d(x^p)}$.   Then $\mathsf{Res}$ maps
$\centh \dms \oplus \centh \brf$ isomorphically onto  $\im\,\mathsf{Res}$ as $\centh$-modules by Theorem \ref{T:imres}, 
which leads to  our main result on derivations.
 \medskip 

\begin{thm}\label{T:decomp2}  Assume $\chara(\FF) = p > 0$.  Then as 
a  $\centh$-module,
\begin{equation}\label{eq:decomp2}  \der(\A_h) = \centh \dms \oplus \centh \brf \oplus \Big(  \mathcal D_\Theta  +  \{ \ad_a \mid  a \in \norm \} \Big), \end{equation} where 
\begin{itemize}
\item[{\rm (i)}]   $D_r(x) = 0, \  D_r(\hat y) = r,$ for all  $r \in \DD$;
\item[{\rm (ii)}]  $\mathcal D_{\Theta} = \{D_r \mid r \in \Theta\}$  and  $\Theta =  \{r \in \DD \mid \mathsf{Res}(D_r) = 0\}$  as in Proposition \ref{P:alt-omega}\,(a);
 \item[{\rm (iii)}] $\dms$ is as in Proposition  \ref{P:alt-omega}\,(d);   
 \item[{\rm (iv)}] $\brf = \ze D_{\frac{h'}{\varrho_h}} - \be_x   =   \ze D_{\frac{h'}{\varrho_h}}-\frac{h^p}{\varrho_h} E_x$.   Hence,  $\brf(x) = -\frac{h^p}{\varrho_h} y^{p-1}$, \ and  \\
$\displaystyle{\brf(\hat y) =  \frac{h^p}{\varrho_h}\sum_{k=1}^{p-2} {\frac{(-1)^{k}}{(k+1)k}} h^{(k+1)} y^{p-k}   + \frac{h^p}{\varrho_h}\left(\partial_p(h) y  + \parp(h')\right),}$ \\
where  $\parp$  is as in  \eqref{eq:projdef}. 
\end{itemize} 
\end{thm} 

\begin{proof}   Suppose $D \in \der(\A_h)$.  Then there exist $u,v \in \centh$ such that 
$\mathsf{Res}(D) = u  \overbar h \frac{d}{d\ze}  + v  \frac{h^p}{\varrho_h} \frac{d}{d(x^p)} = u\mathsf{Res}(\dms) + v\mathsf{Res}(\brf) = \mathsf{Res}(u \dms+ v \brf)$.  Consequently,  $D-u \dms - v \brf$ belongs to $ \ker\, \mathsf{Res}$,  which is  $\mathcal D_{\Theta} +  \{ \ad_a \mid  a \in \norm\}$  by Proposition  \ref{prop:kerres}. This implies that $D$ belongs to the right-hand side of \eqref{eq:decomp2}.    But since the right-hand side is clearly
contained in $\der(\A_h)$, we have the result.  The action of $\brf$ on $x$ and $\hat y$ is
a consequence of  Lemma \ref{lem:ee}. \end{proof}

\begin{cor}\label{C:decomp2} There exists  a finite-dimensional subspace $\mathsf{S}$ of $\DD$   such that  $\Theta =\mathsf{S}  \oplus \im \, \delta$  and 
$$\der(\A_h)  = \centh \dms \oplus \centh \brf  \oplus \Big( \mathcal{D}_{\mathsf S} \oplus   \{\ad_a \mid a \in \norm\}\Big)$$
as a $\centh$-module, where $\mathcal{D}_{\mathsf{S}} = \{D_s \mid s \in \mathsf{S}\}$ and $\mathsf{S} = 0$ if $\Theta = \im\, \delta$.  \end{cor}   \smallskip

The  information in Examples \ref{Ex:Psi} and \ref{EX:t1t2},  coupled with Theorem \ref{T:decomp2},   enables us to
determine  $\der(\A_h)$ explicitly  for any $h = x^m$. 

\begin{cor}\label{E:small h} Let  $h = x^m$, where $m = kp +n$,  $k \geq 0$, and $0 \leq n < p$.   
Then
 \begin{itemize}
\item[{\rm (i)}] $ \der(\A_h) = \centh D_{x^{p-1}} \oplus\centh x^{m(p-1)}E_x 
\oplus \mathcal{D}_{\mathsf{S}}\oplus \{\ad_a\,|\, a \in \norm\}$  if $n = 0$, and 
\item[{\rm (ii)}] $\der(\A_h) = \centh D_{x^{n-1}}  \oplus    \centh x^{(m-k)p}E_x   \oplus \mathcal{D}_{\mathsf{S}}  \oplus   \{\ad_a\,|\, a \in \norm\}$ 
if  $1 \leq n < p,$
\end{itemize}
where $\mathsf{S} = \spann_\FF\{x^i  \mid  0 \leq i < m, \  i \not \equiv n-1 \modd p\}$ in {\rm (i)} and {\rm (ii)}.   \end{cor}
\begin{proof}  (i) If $n = 0$, then as in \eqref{eq:tone}  we have  \  $\overbar h =  (x^p)^{m-k} = h^{p-1}$,  and so   $\ms = -x^{p-1}$.   
Since $h' = 0$,   $\brf = - \frac{h^p}{\varrho_h} E_x  = -x^{m(p-1)}E_x$.

(ii) If  $n \neq 0$,  $h^{p-1} = (x^p)^{m-k-1}\cdot x^{p-n}$,   $\overbar h =  (x^p)^{m-k-1}$, and $\ms = -x^{n-1}$. 
Since $h' = n x^{m-1}$ and $\varrho_h = x^{kp}$,  we have $\brf = \ze D_{\frac{h'}{\varrho_h}}-\be_x  = n\ze D_{x^{n-1}}-x^{(m-k)p} E_x$.

In both (i) and (ii),  the subspace $\mathsf{S}$ can be determined 
from Example \ref{Ex:Psi}.  \end{proof}

Here are a few particular instances of these results.
\vspace{.05cm}

 \begin{exam}\label{Ex:powerx}{\quad}   
 \begin{itemize}
\item When $h = 1,$  then  $\ms = -x^{p-1}$, \  $D_{\ms} = -E_y$,  and  $\brf = -E_x$,  so that   
\begin{center}{$ \der(\A_1) = \cent1 E_x \oplus \cent1 E_y \oplus \inder(\A_1)$   \qquad (Theorem \ref{thm:derA1}).}\end{center}  
 \item When $h = x,$  then   $\ms = -1$,  \ $D_{\ms} = -D_1$,  $\brf = \ze D_{1} - x^p E_x$, and  
\begin{equation*} \der(\A_x) \, = \,  \mathsf{Z}(\A_x) D_1 \oplus \mathsf{Z}(\A_x)x^pE_x \oplus
\inder(\A_x).\end{equation*}
(That $\{\ad_a \mid a \in \mathsf{N}_{\A_1}(\A_x)\} = \inder(\A_x)$ follows from Theorem \ref{T:pfreeHH1} below, or this could be deduced from Theorem \ref{T:norm}.)
\item When $h = x^n, \ 2 \leq n < p$,  then $\mathsf{S} = \spann_\FF\{x^i \mid 0 \leq i \leq n-2\}$ and  \end{itemize} 
\hspace{.5cm} $\der(\A_{x^n}) \,=\, \mathsf{Z}(\A_{x^n}) D_{x^{n-1}} \oplus \mathsf{Z}(\A_{x^n})x^{np}E_x \oplus
\mathcal D_{\mathsf{S}} \oplus \{\ad_a \mid a \in \mathsf{N}_{\A_1}(\A_{x^n})\}.$  
\end{exam}\smallskip

The next example generalizes the $n = 0$ case above.
 \begin{exam}\label{Ex:h=rho}   Assume $h \in \FF[x^p]$.   Then $\overbar h = h^{p-1}$;  \
 $\ms = -x^{p-1}$;  $\Theta = \{r \in \DD \mid  r h^{p-1} \in 
\im \, \frac{d}{dx}\} = \im \, \frac{d}{dx}$ as $h^{p-1}\in\FF[x^p]$ and $r'h^{p-1}=(rh^{p-1})'$.
Since $\delta_0(r) = (r h^{-1})' h = r' \in \im\, \frac{d}{dx}$, we have $\im \, \delta_0 = \im \,\frac{d}{dx} = \Theta$.
Now $\brf = \ze D_{\frac{h'}{\varrho_h}} - \be_x = -\lambda h^{p-1}E_x$, where $\lambda$ is the leading coefficient of $h$.  Thus,  
$$\der(\A_h) = \centh  D_{x^{p-1}}  \oplus \centh h^{p-1} E_x  \oplus \mathcal D_{\mathsf S} 
 \oplus   \{\ad_a \mid a \in \norm\},$$
 where $\mathsf S = \spann_\FF\{ x^i \mid  0 \leq i < \degg h, \  i \not \equiv -1 \modd p\}.$   \end{exam}

\begin{prop}\label{C:inner}   Suppose $D = u D_{\ms} + v\brf + D_r + \ad_a \in \inder(\A_h)$, where $u,v \in \centh, \,  r \in \Theta$,
and $a \in \norm$.  Then $u = 0 = v$,  $r \in \im\,\delta$ and $a \in \A_h + \cent1$.
Thus,   $\hoch(\A_h)  = \der(\A_h)/\inder(\A_h) \cong  \centh \dms \oplus \centh \brf  \oplus \mathcal H$,   where 
\begin{eqnarray*}  \mathcal H &=& \mathsf{ker\,\overbar{Res}} =  \Big(\mathcal D_\Theta + \{ \ad_a \mid  a \in \norm \} \Big)/ 
\Big(\mathcal D_{ \im\,\delta} + \{\ad_a \mid a \in \A_h\}\Big),\\
&\cong& \big( \Theta/ \im \, \delta\big) \ \oplus \ \big( \norm/ \big(\A_h+\cent1\big)\big), \end{eqnarray*} 
\noindent and this decomposition of $\mathcal H$ is as an $\FF[x^{p}]$-module.\end{prop}
\begin{proof} Applying $D$ to $\centh$ shows that $u = 0 = v$.  The remaining
assertions come directly from Proposition \ref{prop:inderres}. \end{proof} 

\subsection{$\hoch(\A_h)$ as a $\centh$-module} \hfil \smallskip
 
Proposition \ref{C:inner} gives a $\centh$-module decomposition of $\hoch(\A_h)$, since $\mathsf{\overbar{Res}}$ is a $\centh$-module map.  The main result of this section is Theorem \ref{T:pfreeHH1}, which provides necessary and sufficient conditions for $\hoch(\A_h)$ to be a free $\centh$-module.   Our proof of this result uses the map $\delta_0: \DD \rightarrow \DD$ with $\delta_0(r) =   \delta(r a_0)$, where $a_0 = \pi_h h^{-1}$, along with the properties in Section \ref{subsec:delta0} that $\delta_0$  satisfies.

 \begin{lemma}\label{L:del0again}   Let  $\Theta = \{r \in \DD \mid  \mathsf{Res}(D_r)  = 0\}$ as in Proposition
\ref{P:alt-omega}\,(a).   Then
\begin{itemize}
\item [{\rm (i)}]  $ \im\,\delta \, \subseteq \, \im\,\delta_0 \subseteq \Theta$;
\item [{\rm (ii)}] $\delta_0(1) = 0$ if and only if $\frac{h}{\pi_h \varrho_h}  \in \FF^*$;
\item [{\rm (iii)}]   $\im\,\delta_0$ is a free $\FF[x^p]$-submodule of $\DD$ of rank $p-1$;
\item [{\rm (iv)}]  If $\frac{h}{\pi_h \varrho_h} \in \FF^*$,  then $\im \, \delta_0 = \Theta$, and
$\DD = \FF[x^p] \ms \oplus \Theta = \FF[x^p] \ms \oplus \im \, \delta_0$, where $\ms$ is as in 
(d)  of Proposition  \ref{P:alt-omega}. 
\end{itemize}
\end{lemma}

\begin{proof} (i) Recall from (a) of Lemma \ref{L:del0} that $D_{\delta_0(r)} = -\ad_{ra_0}$ for $r \in \DD$.  This implies that  $\mathsf{Res}(D_{\delta_0(r)}) = 0$, where $\mathsf{Res}$ is the restriction to $\centh$,  and hence that  $\im \, \delta_0 \subseteq \Theta$.   That $\im \, \delta \subseteq \im \, \delta_0$
follows easily from  the fact  $\delta (r) = \delta ( r \frac{h}{\pi_h} \frac{\pi_h}{h}) = \delta_0 (r \frac{h}{\pi_h})$  for all  $r \in \DD$. 

(ii)  By Lemma \ref{L:del0ker}\,(a),  \, $\delta_0(1) = 0$ if and only if  $1 \in
\ker\delta_0 =  (\DD \cap \centh)\frac{h}{\pi_h \varrho_h}
= \FF[x^p] \frac{h}{\pi_h \varrho_h}$;    whence $\delta_0(1) = 0$ if and only if $\frac{h}{\pi_h \varrho_h}  \in \FF^*$. 

(iii) The identity $\delta_0(rs) = r \delta_0(s) + r' s \pi_h = r \delta_0(s)$, which holds for all $r \in \FF[x^p]$
by (b) of Lemma \ref{L:del0}, implies that $\im\,\delta_0$ is an $\FF[x^p]$-submodule of the free $\FF[x^p]$-module $\DD$.    
As $\FF[x^p]$ is a Dedekind domain, it is hereditary, so $\im\, \delta_0$ is free,  and the  short exact sequence
$$ 0 \rightarrow \ker \delta_0 \rightarrow \DD \xrightarrow{\delta_0} \im\, \delta_0 \rightarrow 0$$
splits.    Since $\ker \delta_0 = \FF[x^p] \frac{h}{\pi_h \varrho_h}$ has rank $1$,  it follows that 
$\im\,\delta_0$ has rank $p-1$.   

(iv)  Assume $\frac{h}{\pi_h \varrho_h} \in \FF^*$. 
Let us first dispose of the case that $h \in \FF[x^p]$.  Then $\pi_h = 1$,   $\frac{h}{\varrho_h} \in \FF^*$, and $\delta_0 = \frac{d}{dx}$,   so that $\im\, \delta_0 
=  \im \,\frac{d}{dx}$.  From 
Example \ref{Ex:h=rho},  we have  $\ms =-x^{p-1}$, \  $\Theta = \im \, \frac{d}{dx}$,  and  
$\DD = \FF[x^p]\ms \oplus \im \,\frac{d}{dx} = \FF[x^p] \ms \oplus \im\, \delta_0$. 

Henceforth, we  assume $h \not \in \FF[x^p]$.  Suppose we can show that in this case there exists $\kappa \in \DD$ such that $\DD = \FF[x^p] \kappa \oplus
\im \, \delta_0$.   Then since  $\im \, \delta_0 \subseteq \Theta$ by (i),   and $\DD\neq \Theta$ by  Proposition \ref{P:alt-omega},  it follows that $\kappa \not \in \Theta$.   Any $r \in \Theta$ must
have trivial projection onto $\FF[x^p] \kappa$, as $\mathsf{Res}(D_r) = 0$.  Hence,  $\Theta \subseteq \im\,\delta_0$, 
equality  would  hold,  and (iv)  would follow from Proposition \ref{P:alt-omega}.  

By (iii), it will be enough to show that the $\FF[x^{p}]$-module $\DD/\im\, \delta_0$ is torsion free, as this will imply it is free, so that the natural epimorphism $\DD\rightarrow \DD/\im\, \delta_0$ will yield the decomposition $\DD=\mathsf{K} \oplus\im\, \delta_0$, for some rank-one free $\FF[x^{p}]$-submodule $\mathsf{K}=\FF[x^{p}]\kappa$.
\medskip 

\textit{Claim:} {\it The $\FF[x^{p}]$-module $\DD/\im\, \delta_0$ is torsion free.}

\smallskip
\textit{Proof of the claim:}  We will show that whenever  $s\in\DD$, $0\neq  w \in\FF[x^{p}]$, and  $w s \in\im\, \delta_0$, then $s\in\im\, \delta_0$. We can assume  $ w\notin\FF$.

First notice that $\DD=\FF[x^{p}]x^{p-1} \oplus \im \frac{d}{dx}$, so that $\DD/\im\frac{d}{dx}$ is a torsion-free $\FF[x^{p}]$-module. This means that if  $w \in\FF[x^{p}]$ divides $r'$, for some $r\in\DD$, then $r'=w \tilde r\,' $ for some $\tilde r\in\DD$. 

By assumption $\frac{h}{\pi_{h}\varrho_h}\in\FF^{*}$, so  we have that
 $\delta_0(r) = r\delta_0(1) + r' \pi_h = r' \pi_h$ by  (ii). Thus, we need to show that $w \mid r'\pi_{h}$ implies $w\mid r'$, for all $r\in\DD$.    Since we are in the case $h \not \in \FF[x^p]$, we can assume  
 $\pi_{h}=\pr_{1}\cdots \pr_{\ell'}$, where the $\pr_{i}$  are distinct monic prime factors of $h$ in $\DD$ and  $\pr_{i}\not\in \FF[x^p]$ for all $i=1,\dots,\ell'$.    Suppose  that $w \mid r'\pi_{h}$  for some $r\in\DD$. Let $v$ be a prime factor of $w$ in $\DD$, and let  $\alpha\geq 1$  be the largest power of $v$ that divides $w$.  Since $w \in \FF[x^p]$,   
this implies that $v^{\alpha}\in\FF[x^{p}]$. The claim will be proved if we show that $v^{\alpha}$ divides $r'$. This is clear if $v$ and $\pr_{i}$ are coprime for all $i$, so we can assume, without loss of generality, that $v=\pr_{1}$. Since $\pr_{1}\not\in \FF[x^p]$, it follows that $p\mid\alpha$, say $\alpha=pn$ for some $n \geq 1$, and $\pr_{1}^{pn-1}$ divides $r'$. In particular, $\pr_{1}^{p(n-1)}\in\FF[x^{p}]$ divides $r'$, so by the above there exists $\tilde r\in\DD$ so that $r'=\pr_{1}^{p(n-1)} \tilde r\,'$. Moreover, $\pr_{1}^{p-1}$ divides $\tilde r\,'$.   We will finish the proof of the claim by showing that this implies that $\pr_{1}^{p}$ divides $\tilde r\,'$. This will be accomplished in three steps:
\smallskip

\textbf{Step 1:} Assume $\pr_{1}=x$. Then $t x^{p-1}=\tilde r\,'$, for some $t \in\DD$. In particular, $t x^{p-1}\in  \im \frac{d}{dx}  =\bigoplus_{i=0}^{p-2}\FF[x^{p}]x^{i}$, so $t \in\bigoplus_{i=1}^{p-1}\FF[x^{p}]x^{i}$. Hence $x$ divides $t$, and $\pr_{1}^{p}=x^{p}$ divides $\tilde r\,'$.

\smallskip  

\textbf{Step 2:} Assume $\degg \pr_{1}=1$. Then there is $\xi\in\FF$ so that $\pr_{1}=x-\xi$. Note that the automorphism $\sigma_{\xi}:\DD\rightarrow\DD$ given by $x\mapsto x+\xi$ commutes with the derivation  $\frac{d}{dx}$, as $(x+\xi)'=1$. Thus, if we apply $\sigma_{\xi}$ to the relation  $\tilde r\,'=\pr_{1}^{p-1}t$ we obtain 
$$\sigma_{\xi}(\tilde r)\,' = \sigma_{\xi}(\tilde r\,') = \sigma_\xi(\pr_1)^{p-1}\sigma_{\xi}(t) =
x^{p-1}\sigma_{\xi}(t).$$
Then by \textbf{Step 1} we have that  $\sigma_{\xi}(\tilde r\,')=x^{p}\tilde t$, for some $\tilde t\in\DD$. Applying $\sigma_{\xi}^{-1}=\sigma_{-\xi}$ 
to that relation, we obtain $\tilde r'=(x-\xi)^{p}\, \sigma_{-\xi}(\tilde t)$, so that $\pr_{1}^{p}$ divides $\tilde r\,'$.

\smallskip

\textbf{Step 3:} The general case. Consider the factorization $\mathsf{f}_{1}^{\beta_{1}}\cdots \mathsf{f}_{k}^{\beta_{k}}$ of $\pr_{1}$ into 
linear factors over the algebraic closure $\overline{\FF}$ of $\FF$. As $\pr_{1}\not\in \FF[x^p]$, we have that ${\pr_{1}}' \neq 0$, so $\pr_{1}$ and $\pr_{1}'$ are coprime. This implies that $\beta_{j}=1$ for all $j$,  and thus $\pr_{1}^{p-1}=\mathsf{f}_{1}^{p-1}\cdots \mathsf{f}_{k}^{p-1}$. Since $\degg \mathsf{f}_{j}=1$, we can apply \textbf{Step 2} to conclude that for all $j$, $\mathsf{f}_{j}^{p}$ divides $\tilde r\,'$ in $\overline{\FF}[x]$. Hence, $\pr_{1}^{p}$ divides $\tilde r\,'$, and this occurs in $\FF[x]$, as $\pr_{1}^{p}$ and $ \tilde r\,'$ are in $\FF[x]$.

Thus, the claim is established,  and there is $\kappa\in\DD$ so that $\DD=\FF[x^{p}]\kappa\oplus\im\, \delta_0$.
As we have argued earlier, this is sufficient to give the assertions in  (iv).    \end{proof}

\begin{thm}\label{T:pfreeHH1}
Assume $\chara(\FF) = p > 0$, and let $\dms$ and $\brf$ be as in  Theorem \ref{T:decomp2}. Then $\hoch(\A_h) = \der(\A_h)/\inder(\A_h)$ is a  free $\centh$-module if and only if $\frac{h}{\pi_{h}}\in\FF^{*}$. When $\frac{h}{\pi_{h}}\in\FF^{*}$, then  
\begin{equation*}
\der(\A_h)=\centh D_{\ms}\oplus\centh \brf \oplus \inder(\A_h),
\end{equation*}
so that  $\hoch(\A_h)$ is a free $\centh$-module of rank $2$ with $\centh$-basis $\{ D_{\ms}, \brf\}$.
\end{thm}

\begin{proof}
Suppose first that $\hoch(\A_h)$ is a  free $\centh$-module.  As  $\centh$ is a domain, $\hoch(\A_h)$ is torsion free over $\centh$. Note that  $h^{p}\ad_{a_{1}}=\ad_{h^{p}a_{1}}=\ad_{h^p\pi_{h}y}\in\inder(\A_h)$,  so $\ad_{a_{1}}\in\inder(\A_h)$, because $h^{p}\in\centh$. This implies that $\pi_{h}=[\pi_{h}y, x]=\ad_{a_{1}}(x)\in[\A_{h}, \A_{h}]\subseteq h\A_{h}$, by \cite[Lem.\ 6.1]{BLO1}. Hence $h$ divides $\pi_{h}$ and $\frac{h}{\pi_{h}}\in\FF^{*}$.

Conversely, assume $\frac{h}{\pi_{h}} = \lambda \in\FF^{*}$.  Then $a_0 = \pi_h h^{-1} = \lambda^{-1}$,  
and $\delta_0(r) = \delta(\lambda^{-1} r)$ for all $r \in \DD$.     Therefore,  $\im\, \delta = \im \, \delta_0 = \Theta$, 
where the last equality follows from  (iv) of  Lemma~\ref{L:del0again}.    By (a) of Corollary  \ref{C:decomp2},  
$\der(\A_h) = \centh D_{\ms} \oplus \centh \brf \oplus \{\ad_a \mid a \in \norm\}$.   
Now suppose $a \in \norm$.   As in Remark \ref{R:norm},    $a = b+c$ where $b \in \norm_{\not \equiv 0}$,  and
$c \in \norm_{\equiv 0}$.     Because  $\frac{h}{\pi_{h}}\in\FF^{*}$,  we know  $b \in \A_h$. 
 By
Lemma \ref{L:adc=Df},   $\ad_c = D_f$ for some $f \in \mathsf{C}_{\A_h}(x) = \centh \DD$.   As
$\DD = \FF[x^p] \ms  \oplus \Theta = \FF[x^p] \ms  \oplus \im \, \delta$,  it follows that   
$\mathsf{C}_{\A_h}(x)  = \centh \ms \oplus \centh \im \, \delta$.   We may assume $f = u \ms +\sum_i v_i \delta(r_i)$
for some $u,v_i \in \centh$ and $r_i \in \DD$.   But then  $\ad_c = D_f = u \dms + \sum_i v_i D_{\delta(r_i)} =
 u \dms - \sum_i v_i \ad_{r_i}$ by (ii) of Proposition \ref{prop:autg}.    The directness of the decomposition in Theorem \ref{T:decomp2} forces 
 $u = 0$, and $\ad_c = - \sum_i v_i \ad_{r_i} = -\sum_i \ad_{v_i r_i} \in \inder(\A_h)$.   This shows that $\{\ad_a \mid a \in \norm\} = 
 \inder(\A_h)$ and completes the proof.   \end{proof}

\begin{remark} When  $h = x$,  then $\frac{h}{\pi_h} \in \FF^*$,  so Theorem  \ref{T:pfreeHH1}
gives the result $\{\ad_a \mid a \in \mathsf{N}_{\A_1}(\A_x)\} = \inder(\A_x)$  mentioned in Example \ref{Ex:powerx}.   \end {remark}  

\begin{remark} When $\frac{h}{\pi_h}\in\FF^*$,  it follows  from  Theorem \ref{T:pfreeHH1} and Proposition \ref{C:inner}  that $\mathcal H =\mathsf{ker\,\overbar{Res}}=0$.   Hence,  in this case,  $\hoch(\A_h)$ is isomorphic
via the map $\overbar{\mathsf{Res}}$  to the subalgebra  of the Witt algebra $\der(\FF[t_1,t_2])$ generated over $\FF[t_1,t_2]$
by the derivations   $\mathsf{d}_1 = h^p \frac{d}{dt_1}, \mathsf{d}_2 = \overbar h \frac{d}{dt_2}$, where $t_1 = x^p$ and $t_2 = \ze$, 
(see Theorem \ref{T:imres} for details).\end{remark}

\subsection{Products in $\der(\A_h)$}\label{sec:prodsDerChar-p} \hfil  \smallskip
 
Suppose $u,v \in \centh$ and $D,E \in \der(\A_h)$.    Then
 \begin{equation}\label{eq:genprod}  [uD, vE]  =  uD(v)E - vE(u)D + uv[D,E].  \end{equation} 
Equation \eqref{eq:genprod} tells us that  to compute products in $\der(\A_h)$, it suffices to know the action of  
the restriction $\mathsf{Res}(D)$ on $\centh = \FF[x^p,\ze]$ for all derivations  $D$ in $\mathcal B = \left \{ \dms,  \brf, D_r, \ad_a \mid r \in \Theta, a \in  \norm \right\}$, where $\dms$ and $\brf = \ze D_{\frac{h'}{\varrho_h}} - \be_x$ are
as in Theorem \ref{T:decomp2},   and the commutator $[D,E]$ for all pairs $D\neq E$ in $\mathcal B$. 
The first  part is easy, since
\begin{gather}\begin{split}\label{eq:restrict} \mathsf{Res}(\dms)= \overbar h \frac{d}{d\ze}, \qquad  \mathsf{Res}(\brf) = \frac{h^p}{\varrho_h} \frac{d}{d(x^p)}, \quad  \hbox{\rm and}  \\
\mathsf{Res}(D_r) = 0 = \mathsf{Res}(\ad_a) \ \  \forall \, r \in \Theta, \, a \in \norm. \end{split}  \end{gather} 

Now it follows from Theorem \ref{T:norm}
 that any   $a \in \norm$  has the form $a = b+c$, where $b \in \norm_{\not \equiv 0}$,
$c \in \norm_{\equiv 0}$, and $b$ is a sum of terms of the form $r  a_n$  with   
 $a_n = \pi_h h^{n-1}y^n$ for $n \geq 1$,  and $r \in \DD$.    Lemma \ref{L:adc=Df}  says that  $\ad_c  = D_f = \sum_i z_i D_{r_i}$ for some $f = \sum_i z_i r_i \in \mathsf{C}_{\A_h}(x)= \centh \DD$.   Hence, we are able  to reduce our considerations to products of the form  in (a)-(e) below, 
so that the commutator  of any pair of derivations in $\mathcal B$ can be
deduced from the next proposition. 
 \smallskip
 
 \begin{prop}\label{prop:bracketp} Let $a_n = \pi_h h^{n-1}y^n$ for all $n \geq 0$,   and assume  $a_{-1} = 0$.  
 The Lie brackets in $\der(\A_h)$ satisfy the following, where $\delta_0(r) = (r \pi_h h^{-1})'h$, as in (\ref{eq:del0}).
\begin{itemize}
\item[{\rm (a)}]  $[D_f, D_g]= 0$ \  for all $f,g \in \DD$.
\item[{\rm (b)}]  $[D_g, \ad_{r a_n}]  = n \ad_{gr a_{n-1}} =  n \ad_{c a_{n-1}}$ in $\hoch(\A_h)$, where $c$ is the remainder of  the division of $gr$
  by $\frac{h}{\pi_h}$ in $\DD$.
   \item[{\rm (c)}] $[ \ad_{ra_{m}}, \ad_{sa_{n}} ] =  \ad_{qa_{m+n-1}} = \ad_{da_{m+n-1}}$  in $\hoch(\A_h)$  for all $r, s\in\DD$ and all $m, n\geq 0$,  where $q= mr \delta_0(s) -ns\delta_0(r)$,
 and $d$ is the remainder of the division in $\DD$ of $q$ by $\frac{h}{\pi_h}$. 
\item[{\rm (d)}]  Assume  $r \in \DD$ and  $m = kp+n$, where $k \geq 0$ and $0 \leq n < p$.  \  Then in $\hoch(\A_h)$, 
\begin{equation}\label{eq:bexprod}  [\be_x, \ad_{ra_m}]  = \ze^k\, [\be_x, \ad_{ra_n}]  = \begin{cases} \ze^{k+1} \ad_{\zeta_n a_{n-1}} & \hbox{\rm if} \ 1 \leq n < p, \\
 \ze^k\, [D_{\delta_0(r)}, \be_x]  & \hbox{\rm if} \  n = 0,
 \end{cases}\end{equation}
 where $\zeta_n=\frac{h}{\pi_h\varrho_h}\delta_0(r)+nr\frac{h'}{\varrho_h}$, and the product $[D_{\delta_0(r)}, \be_x]$ can be computed using (e). 
\item[{\rm (e)}]  For $g \in \DD$, \, $[D_g,  \be_x]  = D_e + \ad_b$, where $b = b_1 + b_2$ with 
$$b_1 = \frac{g h^{p-1}}{\varrho_h} y^{p-1} \in \norm, \ \ \  b_2 = \sum_{k=2}^{p-1} (-1)^k \frac{(gh^{-1})^{(k-1)} h^p}{\varrho_h} \frac{y^{p-k}}{p-k} \in \A_h,$$
and $e = \left( [D_g, \be_x] - \ad_b \right) (\hat y) \in \mathsf C_{\A_h}(x)$. 
 \end{itemize}
 \end{prop}
  
\begin{proof}  Part\,(a) is clear, and parts\,(b) and (c) are immediate from Lemma \ref{L:brackets}.
For (d), we have $a_m = \ze^k a_n$ so that 
\begin{align}\begin{split}\label{eq:am}
[\be_x, \ad_{ra_m}]  &=  [\be_x, \ze^k\ad_{r a_n}]
=\be_x(\ze^k)\ad_{r a_n} + \ze^k\,[\be_x, \ad_{r a_n}]\\
&= - k \ze^{k}\,\ad_{r \frac{(h')^{p}}{\varrho_h}a_n} + \ze^k\,[\be_x, \ad_{r a_n}]  = 
  \ze^k\,[\be_x, \ad_{r a_n}]\end{split}
\end{align}
by \eqref{eqn:E_x-center},  where the last equality holds because $h'a_n\in\A_h$ (see Theorem \ref{T:norm}(b)).  
In particular, when $n = 0$,  then $[\be_x, \ad_{ra_m}]  = \ze^k\, [\be_x, \ad_{ra_0}] = \ze^k [D_{\delta_0(r)}, \be_x]$
as claimed in (d),  since $\ad_{ra_0} = -D_{\delta_0(r)}$.

Assume $1 \leq n < p$.   Then  the equalities 
$[\be_x, \ad_{ra_n}]  =\frac{h^p}{\varrho_h} \ad_{E_x(r \pi_h h^{n-1})y^n}$
and 
\begin{equation*}
\frac{h^p}{\varrho_h} E_x(r \pi_h h^{n-1})y^n = \frac{1}{\varrho_h}\sum_{k=1}^{p-1} \frac{(-1)^{k-1}}{k} (r \pi_h h^{n+p-1})^{(k)} y^{n+p-k} -\frac{h^p}{\varrho_h}\parp(r \pi_h h^{n-1})y^n
\end{equation*}
follow directly from Lemma \ref{lem:ee}.  By Lemma \ref{lem:deriv(h^k pi)}, we have that $(r \pi_h h^{n+p-1})^{(k)}\in \DD h^{n+p-k+1} + \DD h^{n+p-k} h'$ for all $k\geq 2$, so that 
$$\frac{1}{\varrho_h}\sum_{k=2}^{p-1} \frac{(-1)^{k-1}}{k} (r \pi_h h^{n+p-1})^{(k)} y^{n+p-k} \in\A_h,$$ as $\varrho_h$ divides both $h$ and $h'$.   
Since $n<p$, \  $\frac{h^p}{\varrho_h}\parp(r \pi_h h^{n-1})y^n\in\A_h.$  Thus, modulo $\A_h$ we have
$$\frac{h^p}{\varrho_h} E_x(r \pi_h h^{n-1})y^n =  \frac{1}{\varrho_h} (r \pi_h h^{n+p-1})' y^{n+p-1} 
= \ze\,\zeta_n a_{n-1},$$
where $\zeta_n = \frac{h}{\pi_h\varrho_h}\delta_0(r)+nr\frac{h'}{\varrho_h}$.  This combined with \eqref{eq:am} gives (d) for $n \neq 0$.

To compute $[D_g, \be_x]$ in part (e), note that since $D_g(x) = 0$,  Lemma \ref{lem:D_g(y^n)-local} implies
\begin{align*}
[D_g, \be_x]( x) &=  \frac{h^p}{\varrho_h} \sum_{k=1}^{p-1} {p-1 \choose k} (gh^{-1})^{(k-1)} y^{p-1-k} \\
&= \sum_{k=1}^{p-1} (-1)^k \frac{(gh^{-1})^{(k-1)}h^{p}}{\varrho_h}y^{p-1-k}. 
 \end{align*}  
Let   \begin{equation}\label{eq:aelt}  b =  \sum_{k=1}^{p-1} (-1)^k  \frac{(gh^{-1})^{(k-1)}h^{p}}{\varrho_h}\frac{y^{p-k}}{p-k} \in \A_1. \end{equation}
Observe that   $\ad_b(x) =  [D_g, \be_x](x) \in \A_h$,  and
\begin{equation}\label{eq:b1}b_1 = \frac{g h^{p-1}}{\varrho_h}y^{p-1} =   g \frac{h}{\pi_h\varrho_h} (\pi_h h^{p-2} y^{p-1}) 
=  g \frac{h}{\pi_h\varrho_h} a_{p-1} \in \norm. \end{equation} 
It  is easy to deduce from Lemma \ref{lem:kDeriv_gh-inverse} that $\frac{(gh^{-1})^{(k-1)} h^k}{\varrho_h} \in \DD$
 for all $k \geq 2$, and thus 
$$b_{2} =   \sum_{k=2}^{p-1} (-1)^k  \frac{(gh^{-1})^{(k-1)}h^{p}}{\varrho_h}\frac{y^{p-k}}{p-k}  \in \A_h.$$ 
As a result,  $b = b_1 + b_2 \in \norm$.  

Now  $G = [D_g, \be_x] - \ad_b \in \der (\A_h)$ satisfies $G(x) = 0$ so that 
$0 = [ G( \hat y) , x]$. This shows that  $e = G( \hat y)  \in \mathsf{C}_{\A_h}(x)$.  But then $(D_e-G)(x) = 0 = (D_e-G)(\hat y)$,
which implies  that $G = D_e$.   Consequently,  $[D_g, \be_x] = D_e +\ad_b$, as desired. 
\end{proof}   

It remains to determine the expression for $e = \big( [D_g, \be_x] - \ad_b \big) (\hat y)$ in part (e) of Proposition \ref{prop:bracketp}.  We do so by considering the terms of $[D_g, \be_x] (\hat y)$ that centralize $x$.  Define the projection map $\PP: \A_{1}\rightarrow \mathsf{C}_{\A_1}(x)$ by $\PP (ry^{k})=ry^{k}$ if $p\mid k$ and $\PP(ry^{k})=0$ otherwise. Note that $\PP(\A_{h})=\mathsf{C}_{\A_h}(x)$ and $\PP(ra)=r\PP(a)$ for all $r\in\DD$ and $a\in\A_{1}$.

\begin{lemma}\label{lem:projectCent(x)}
Let $g, r\in\DD$. Then 
\begin{itemize}
\item[{\rm (a)}] $\PP(D_{g}(h^{n}y^{n}))=h^{n}\left( gh^{-1}\right)^{(n-1)}$ for $1\leq n\leq p$;
\item[{\rm (b)}]  $\PP([ry^{n}, \hat y])=rh^{(n+1)}$ for $1\leq n< p$ and $\PP([r, \hat y])=-r'h$.
\end{itemize}
\end{lemma}
\begin{proof}
Corollary \ref{cor:D_g(a_n)}\,(a) implies  
$D_{g}(h^{n}y^{n})=\sum_{k=1}^{n} {n\choose k}h^{n}\left(gh^{-1}\right)^{(k-1)}y^{n-k}$
for all  $1\leq n\leq p$, and (a) is a direct consequence of this.  
Now   \eqref{eq:yhatprod} says  
$[ry^n, \hat y] = -(rh)'y^n + \sum_{k=1}^{n+1} {{n+1} \choose k} rh^{(k)}y^{n+1-k}$.   
Applying the map $\PP$ to that yields (b). 
\end{proof}

\begin{prop}\label{prop:bracketp2}
For $g \in \DD$, write $[D_g,  \be_x]  = D_e + \ad_b$, with $e \in \mathsf C_{\A_h}(x)$ and $b \in \norm$ as in Proposition \ref{prop:bracketp}.   Assume $\parp$ is as in  \eqref{eq:projdef}.   Then  
\begin{equation}\label{E:e-term}
e=\frac{1}{\varrho_h}\left( \sum_{k=1}^{p-1} \frac{(-1)^{k-1}}{k} (gh^{p-1})^{(k)} h^{(p-k)} \right) 
 \ +\ \frac{h^{p-1}}{\varrho_h} \Big(h\parp(g)-g\parp(h)\Big) \  \in \DD.  \end{equation}
\end{prop}

\noindent \emph{Proof.} \ 
Note that $\PP\left( (D_e + \ad_b)(\hat y) \right) = \PP(e + [b, \hat y]) = e+ \PP([b, \hat y])$, so by (\ref{eq:aelt}) and Lemma \ref{lem:projectCent(x)}, we have 
\begin{eqnarray*}
\PP\left( (D_e + \ad_b)(\hat y) \right) &=& e+ \frac{1}{\varrho_h}\sum_{k=1}^{p-1} \frac{(-1)^{k-1}}{k} \PP\left( \left[ (gh^{p-1})^{(k-1)}y^{p-k}, \hat y \right] \right) \\
&=& e+ \frac{1}{\varrho_h}\sum_{k=1}^{p-1} \frac{(-1)^{k-1}}{k} (gh^{p-1})^{(k-1)}h^{(p+1-k)}\\
&=& e+ \frac{1}{\varrho_h}\sum_{k=1}^{p-2} \frac{(-1)^{k}}{k+1} (gh^{p-1})^{(k)}h^{(p-k)}.
\end{eqnarray*}

\noindent On the other hand,
\begin{align*}
(D_e + \ad_b)(\hat y) &= [D_g, \be_x](\hat y) =  D_g (\be_x (\hat y))-\be_x(g)\\
&\hspace{-1.4cm} =  \frac{1}{\varrho_h}D_{g}(h' h^{p}y^p) +\frac{1}{\varrho_h}\sum_{k=1}^{p-2}\frac{(-1)^{k-1}}{(k+1)k} h^{(k+1)}h^{k}\, D_{g}(h^{p-k}y^{p-k}) \\ 
&\hspace{-1.35cm} \quad   -\frac{h^{ p-1}}{\varrho_h}\parp(h)\, D_g(hy)  -\frac{h^{p}}{\varrho_h}\sum_{k=0}^{p-2}\frac{(-1)^k}{k+1}\, g^{(k+1)} y^{p-1-k} 
+\frac{h^{p}}{\varrho_h} \parp(g).  
\end{align*}
Hence,   
\begin{eqnarray*}
\PP\left( (D_e + \ad_b)(\hat y) \right) &=& 
 \frac{1}{\varrho_h}\sum_{k=1}^{p-2}\frac{(-1)^{k-1}}{(k+1)k} h^{(k+1)}(gh^{p-1})^{(p-k-1)} \\
&& \ \  + \frac{1}{\varrho_h}h' (gh^{p-1})^{(p-1)} + \frac{h^{p-1}}{\varrho_h} \big(h\parp(g)-g \parp(h) \big).
\end{eqnarray*}
\indent  Equating both expressions for $\PP\left( (D_e + \ad_b)(\hat y) \right)$ gives 
\begin{align*}
\varrho_h e &=  h' (gh^{p-1})^{(p-1)}+ h^{p-1}\left(h\parp(g)-g\parp(h)\right) \\
&  \quad +\sum_{k=1}^{p-2}\frac{(-1)^{k-1}}{(k+1)k} h^{(p-k)}(gh^{p-1})^{(k)} 
+\sum_{k=1}^{p-2} \frac{(-1)^{k-1}}{k+1} (gh^{p-1})^{(k)}h^{(p-k)}\\
&= 
\sum_{k=1}^{p-1}\frac{(-1)^{k-1}}{k} (gh^{p-1})^{(k)}h^{(p-k)} 
 +h^{p-1}\big(h \parp(g) -g \parp(h) \big).  \hspace{1.6cm} \square 
\end{align*}

\end{document}